\documentclass[a4paper,twopage,reqno,12pt]{amsart}
\usepackage[top=30mm,right=30mm,bottom=30mm,left=30mm]{geometry}

\usepackage{tabularx}
\usepackage{graphicx}
\usepackage{amsmath}
\usepackage{amssymb}
\usepackage{amsfonts}
\usepackage{amsthm}
\usepackage{amstext}
\usepackage{amsbsy}
\usepackage{amsopn}
\usepackage{amscd}
\usepackage{enumerate}
\usepackage{color}
\usepackage[colorlinks]{hyperref}
\usepackage[hyperpageref]{backref}
\usepackage{multirow}
\usepackage{longtable}
\usepackage{float}
\usepackage{lscape}
\usepackage{pdflscape}
\usepackage{url}

\newtheorem{theorem}{Theorem}[section]
\newtheorem{corollary}[theorem]{Corollary}
\newtheorem{lemma}[theorem]{Lemma}
\newtheorem{proposition}[theorem]{Proposition}

\theoremstyle{definition}

\newtheorem{example}[theorem]{Example}
\newtheorem{remark}[theorem]{Remark}

\numberwithin{equation}{section}

\newcommand{\PSp}{\mathrm{PSp}}

\renewcommand{\leq}{\leqslant}
\renewcommand{\geq}{\geqslant}

\begin{document}
\title[Classification of flag-transitive $2$-designs]{Classification of the non-trivial $2$-$(k^{2},k,\lambda )$ designs, with $\lambda \mid k$, admitting a flag-transitive almost simple automorphism group}

\author[]{ Alessandro Montinaro}

%
%

\address{Alessandro Montinaro, Dipartimento di Matematica e Fisica “E. De Giorgi”, University of Salento, Lecce, Italy}
\email{alessandro.montinaro@unisalento.it}

\subjclass[MSC 2020:]{05B05; 05B25; 20B25}%
\date{\today}%

\begin{abstract}
Non-trivial $2$-$(k^{2},k,\lambda )$ designs, with $\lambda \mid k$,
admitting a flag-transitive almost simple automorphism group are classified.




\end{abstract}

\maketitle

\section{Introduction and Main Result}

A $2$-$(v,k,\lambda )$ \emph{design} $\mathcal{D}$ is a pair $(\mathcal{P},%
\mathcal{B})$ with a set $\mathcal{P}$ of $v$ points and a set $\mathcal{B}$
of blocks such that each block is a $k$-subset of $\mathcal{P}$ and each two
distinct points are contained in $\lambda $ blocks. We say $\mathcal{D}$ is 
\emph{non-trivial} if $2<k<v$. All $2$-$(v,k,\lambda )$ designs in this paper
are assumed to be non-trivial. An automorphism of $\mathcal{D}$ is a
permutation of the point set which preserves the block set. The set of all
automorphisms of $\mathcal{D}$ with the composition of permutations forms a
group, denoted by $\mathrm{Aut(\mathcal{D})}$. For a subgroup $G$ of $%
\mathrm{Aut(\mathcal{D})}$, $G$ is said to be \emph{point-primitive} if $G$
acts primitively on $\mathcal{P}$, and said to be \emph{point-imprimitive}
otherwise. A \emph{flag} of D is a pair $(x,B)$ where $x$ is a point and $B$
is a block containing $x$. If $G\leq \mathrm{Aut(\mathcal{D})}$ acts
transitively on the set of flags of $\mathcal{D}$, then we say that $G$ is 
\emph{flag-transitive} and that $\mathcal{D}$ is a \emph{flag-transitive
design}.\ 

The $2$-$(v,k,\lambda )$ designs $\mathcal{D}$ admitting a flag-transitive
automorphism group $G$ have been widely studied by several authors. In 1990,
a classification of those with $\lambda =1$ and $G\nleq A\Gamma L_{1}(q)$
was announced by Buekenhout, Delandtsheer, Doyen, Kleidman, Liebeck and Saxl
in \cite{BDDKLS} and proven in \cite{BDD}, \cite{Da}, \cite{De0}, \cite{De}, 
\cite{Kle}, \cite{LiebF} and \cite{Saxl}. Since then a special attention was
given to the case $\lambda >1$. A classification of the flag-transitive $2$%
-designs with $\gcd (r,\lambda )=1$, $\lambda >1$ and $G\nleq A\Gamma
L_{1}(q)$, where $r$ is the replication number of $\mathcal{D}$, has been
announced by Alavi, Biliotti, Daneshkakh, Montinaro, Zhou and their
collaborators in \cite{glob} and proven in \cite{A}, \cite{A1}, \cite{ABD0}, 
\cite{ABD1}, \cite{ABD2},\cite{ABD}, \cite{BM}, \cite{BMR}, \cite{MBF}, \cite%
{TZ}, \cite{Zie}, \cite{ZD}, \cite{ZZ0}, \cite{ZZ1}, \cite{ZZ2}, \cite{ZW}
and \cite{ZGZ}. Moreover, recently the flag-transitive $2$-designs with $%
\lambda =2$ have been investigated by Devillers, Liang, Praeger and Xia in 
\cite{DLPX}, where it is shown that apart from the two known symmetric $2$-$%
(16,6,2)$ designs, $G$ is primitive of affine or almost simple type.
Moreover, a classification is provided when the socle of $G$ is isomorphic
to $PSL_{n}(q)\trianglelefteq G$ and $n\geq 3$.

The investigation of the flag-transitive $2$-$(k^{2},k,\lambda )$ designs,
with $\lambda \mid k$, has been recently started in \cite{MF}. The reason of studying such $2$-designs is that they represent a natural generalization of the affine planes in
terms of parameters, and also because, it is shown in \cite{Monty} that, the blocks of imprimitivity of a family of flag-transitive, point-imprimitive symmetric $2$-designs investigated in \cite{PZ} have the structure of the $2$-designs analyzed here.
It is shown in \cite{MF} that, apart from the smallest Ree group, a
flag-transitive automorphism group $G$ of a $2$-$(k^{2},k,\lambda )$ design $%
\mathcal{D}$, with $\lambda \mid k$, is either an affine group or an almost
simple classical group. Moreover, when $G\cong ~^2 G_{3}(3)$, the
smallest Ree group, $\mathcal{D}$ is isomorphic either to the $2$-$%
(6^{2},6,2)$ design or to one of the three $2$-$(6^{2},6,6)$ designs
constructed in \cite{MF}. All the four $2$-designs have the $36$ secants of
a non-degenerate conic $\mathcal{C}$ of $PG_{2}(8)$ as a point set and $6$%
-sets of secants in a remarkable configuration as a block set. Clearly, $%
G\cong ~^2 G_{3}(3)$ is a special case of an almost simple
classical group, since $~^2 G_{3}(3)\cong P\Gamma L_{2}(8)$.

The result contained in the present paper, together with that obtained in \cite{MF}, is a complete classification of $(\mathcal{D},G)$ when $G$ is
almost simple. More precisely, the following result is obtained: 

\medskip 

\begin{theorem}
\label{main}Let $\mathcal{D}$ be a $2$-$(k^{2},k,\lambda )$ design, with $\lambda
\mid k$, admitting a flag-transitive automorphism group $G$ of almost simple
type. Then one of the following holds:

\begin{enumerate}
\item $\mathcal{D}$ is isomorphic to the $2$-$(6^{2},6,2)$ design
constructed in \cite{MF} and $PSL_{2}(8)\trianglelefteq G\leq P\Gamma
L_{2}(8)$.

\item $\mathcal{D}$ is isomorphic to one of the three $2$-$(6^{2},6,6)$
designs constructed in \cite{MF} and $G\cong P\Gamma L_{2}(8)$.

\item $\mathcal{D}$ is isomorphic to the $2$-$(12^{2},12,3)$ design
constructed in Example \ref{Ex5} and $G\cong PSL_{3}(3)$.

\item $\mathcal{D}$ is isomorphic to the $2$-$(12^{2},12,6)$ design
constructed in Example \ref{Ex5} and $G\cong PSL_{3}(3):Z_{2}$.
\end{enumerate}
\end{theorem}

\bigskip

A complete classification of $(\mathcal{D},G)$, with $G$ of affine type and $G\nleq A\Gamma
L_{1}(q) $, is contained in \cite{Mo1}.

\bigskip
\bigskip

The proof of Theorem \ref{main} is outlined as follows. In Lemma \ref{PP} it is shown that $G$
acts point-primitively on $\mathcal{D}$ and the point-stabilizer $G_{x}$ is
a large (maximal) subgroup of $G$. In Theorem \ref{MF1}
and in Lemma \ref{lambada} it proven that, if $G$ is non-abelian simple, $G$
is classical and $\lambda \geq 2$ respectively. Then we use the results
contained in \cite{Bre}, \cite{BP} and \cite{GPPS} to show $Soc(G)_{x}$ is a
large subgroup of $Soc(G)$. Finally, we complete the proof by using the
classification of the large maximal subgroups of simple classical groups
contained in \cite{AB} together with some results arising from the geometry
of the classical groups.

More details about the proof strategy are provided at the beginning of each
section.

\bigskip

\section{Examples}

All the details about the examples of $2$-designs corresponding to cases (1)
and (2) of Theorem \ref{main} are contained in \cite{MF}. Hence, in this
small section we focus on the constructions of the two examples of $2$-designs
corresponding to cases (3) and (4). The tools used
for the constructions are the results about tactical
configurations contained in \cite{Demb}, \cite{HM}, combined with some
group-theoretical information about $\mathrm{Aut(}PSL_{3}(3)\mathrm{)}$
contained in \cite{At}.

\bigskip

Let $PSL_{3}(3)\trianglelefteq G\leq PSL_{3}(3):\left\langle \sigma
\right\rangle $, where $\sigma $ is the symmetric polarity of $PG_{2}(3)$
defining the conic $\mathcal{C}:$ $XZ-Y^{2}=0$. Let $P$ be the subgroup of $%
G^{\prime }$ generated by the following elements:%
\[
\eta =\left( 
\begin{array}{ccc}
1 & 1 & -1 \\ 
0 & 1 & 1 \\ 
1 & 1 & 0%
\end{array}%
\right) \; \textit{ and } \;\psi =\left( 
\begin{array}{ccc}
1 & 1 & -1 \\ 
0 & 1 & -1 \\ 
0 & 0 & 1%
\end{array}%
\right) \text{.} 
\]%
It is easy to see that, $P$ is a Frobenius group of order $39$ with kernel $%
\left\langle \eta \right\rangle $ and complement $\left\langle \psi
\right\rangle $. Then $N_{G}(P)$ is either $P$, or $P:\left\langle \sigma
\right\rangle $, according to whether $G$ is either $PSL_{3}(3)$, or $%
PSL_{3}(3):\left\langle \sigma \right\rangle $, respectively.

Let $L=\left\langle \alpha ,\gamma \right\rangle $, where 
\[
\alpha =\left( 
\begin{array}{ccc}
0 & 0 & 1 \\ 
0 & -1 & 0 \\ 
1 & 0 & 0%
\end{array}%
\right) \; \textit{ and }\;\gamma =\left( \allowbreak 
\begin{array}{ccc}
0 & 0 & 1 \\ 
0 & -1 & -1 \\ 
1 & -1 & 1%
\end{array}%
\right) \text{.} 
\]%
Then $L\cong \Omega _{3}(3)\cong A_{4}$ lies in the stabilizer in $%
PSL_{3}(3) $ of $\mathcal{C}$.

\bigskip

Let $\mathcal{D}=(\mathcal{P},\mathcal{B},\mathcal{I})$ be the incidence
structure, where $\mathcal{P}$ and $\mathcal{B}$ are the sets of the right
cosets of the subgroups $N_{G}(P)$ and $L$ in $G$, respectively, and where 
\[
\mathcal{I}=\left\{ (N_{G}(P)x,Ly)\in \mathcal{P}\times \mathcal{B}:Px\cap
Ly\neq \varnothing \right\} . 
\]

\bigskip

\begin{example}
\label{Ex5}The following hold:

\begin{enumerate}
\item $\mathcal{D}$ is a $2$-$(12^{2},12,3)$ design admitting $G\cong
PSL_{3}(3)$ as the unique flag-transitive automorphism group.

\item $\mathcal{D}$ is a $2$-$(12^{2},12,6)$ design admitting $G\cong
PSL_{3}(3):\left\langle \sigma \right\rangle $ as the unique flag-transitive
automorphism group.
\end{enumerate}
\end{example}

\begin{proof}
Assume that $G\cong PSL_{3}(3)$. Then $P$ is maximal in $G$ by \cite{At},
and hence $N_{G}(\left\langle \eta \right\rangle )=P$. Then the
incidence structure $\mathcal{D}=(\mathcal{P},\mathcal{B},\mathcal{I})$
above defined is flag-transitive with parameters $(v,b,k,r)=(144,468,12,39)$
by \cite{HM}, Lemmas 1 and 2.

The group $P$ fixes exactly one element of $\mathcal{P}$, namely itself, and 
$\left\langle \eta \right\rangle $ acts semiregularly on $\mathcal{P-}%
\left\{ P\right\} $. It follows from \cite{At} that $N_{G}(\left\langle \psi
\right\rangle )\cong D_{18}$. Hence, the number of elements in $\mathcal{P}$
fixed by $\left\langle \psi \right\rangle $ is $6$. Since $P$ is a Frobenius
group and since $\left\langle \eta \right\rangle $ acts semiregularly on $%
\mathcal{P-}\left\{ P\right\} $, the five points fixed by $\left\langle \psi
\right\rangle $ and distinct from $P$ lie in five distinct $\left\langle
\eta \right\rangle $-orbits by \cite{Pass}, Proposition 4.2. Thus $\mathcal{P%
}=\left\{ Px:x\in G\right\} $ is partitioned in one $P$-orbit of length $1$,
five ones of length $13$ and two ones of length $39$.

Each $P$-orbit is of the form $PyP$ with $y\in G$. If we set $\beta =\alpha
^{\gamma }$, then $\left\langle \alpha ,\beta \right\rangle \cong E_{4}$,
and it is not difficult to check that
\small 
\begin{equation}
P\alpha P\cap L=\left\{ \alpha ,\gamma ^{\beta },\left( \gamma ^{\beta
}\right) ^{-1}\right\}, \; P\beta P\cap L=\left\{ \beta ,\gamma
^{\alpha },\left( \gamma ^{\alpha }\right) ^{-1}\right\} \text{ and } \; PxP\cap
L=\left\{ x\right\}  \label{centrl}
\end{equation}%
\normalsize
for each \small{$x\in \left\{ 1,\alpha \beta ,\gamma ,\gamma ^{-1},\gamma ^{\alpha
\beta },\left( \gamma ^{\alpha \beta }\right) ^{-1}\right\} $}. \normalsize Thus, the
double cosets $PyP$ with $y\in \left\{ 1,\alpha ,\beta ,\alpha \beta ,\gamma
,\gamma ^{-1},\gamma ^{\alpha \beta },\left( \gamma ^{\alpha \beta }\right)
^{-1}\right\} $ are exactly all the eight $P$-orbits partitioning $\mathcal{P%
}$.

For each \small $y\in \left\{ \alpha ,\beta ,\alpha \beta ,\gamma ,\gamma
^{-1},\gamma ^{\alpha \beta },\left( \gamma ^{\alpha \beta }\right)
^{-1}\right\} $, \normalsize the incidence structure $\mathcal{K}_{y}=\left(
PyP,LP\right) $ is a tactical configuration by \cite{Demb}, 1.2.6. It follows from (\ref{centrl})
that, the parameters $(v_{y},b_{y},k_{y},r_{y})$ of $\mathcal{K}_{y}$ are either equal
to $(39,39,3,3)$, or to $(13,39,1,3)$, according to whether $y\in \left\{
\alpha ,\beta \right\} $ or $y\in \left\{ \alpha \beta ,\gamma ,\gamma
^{-1},\gamma ^{\alpha \beta },\left( \gamma ^{\alpha \beta }\right)
^{-1}\right\} $ respectively. Thus $r_{y}=3$ in each case,
and hence the number of blocks of $\mathcal{D}$ incident with $P$ and $Px$,
where $x\neq 1$, is $3$. Then $3$ is the number of blocks incident with any two
distinct points of $\mathcal{P}$, as $\mathcal{D}$ is a flag-transitive
tactical configuration. Therefore, $\mathcal{D}=(\mathcal{P},\mathcal{B})$
is a $2$-$(12^{2},12,3)$ design admitting $G$ as a flag-transitive
automorphism group.

Assume that $\mathrm{Aut(\mathcal{D})}\neq G$. Then $G\trianglelefteq 
\mathrm{Aut(\mathcal{D})}\leq \mathrm{Aut(}G\mathrm{)}$, where $\mathrm{Aut(}%
G\mathrm{)}\cong G:\left\langle \sigma \right\rangle $ and $\sigma $ is
defined above, as a consequence of the O'Nan-Scott Theorem (e.g. see \cite%
{DM}, Theorem 4.1A), since $v=2^{4}\cdot 3^{2}$. Then $\mathrm{Aut(\mathcal{D%
})}=\mathrm{Aut(}G\mathrm{)}$, since $G\neq \mathrm{Aut(\mathcal{D})}$ by
our assumption. By \cite{AtMod}, one $\mathrm{Aut(}G\mathrm{)}_{P}$-orbit on 
$\mathcal{P}$, say $\mathcal{O}$, is of length $26$ and is union of two $%
G_{P}$-orbits of length $13$, the remaining $\mathrm{Aut(}G\mathrm{)}_{P}$%
-orbits on $\mathcal{P}-\mathcal{O}$ are also $G_{P}$-orbit. Thus, if $\mathcal{B}_{P}$ denotes the set of blocks of $\mathcal{D}$ incident with $P$, any block
in $\mathcal{B}_{P}$ intersects $\mathcal{O}$ in exactly $4$ points. Hence, $(\mathcal{O},\mathcal{B}_{P})$ is a tactical configuration by \cite%
{Demb}, 1.2.6, and its parameters are $(v_{0},b_{0},k_{0},r_{0})=(26,39,4,6)$. So,
there are exactly $6$ blocks of $\mathcal{D}$ incident with $P$ and with any
fixed point of $\mathcal{O}$. This is impossible, since $\mathcal{D}$ is a $2$%
-design with $\lambda =3$. Thus $\mathrm{Aut(\mathcal{D})}=G$. Clearly, $G$
is also the minimal flag-transitive automorphism group of $\mathcal{D}$, and
(1) is proven.

Finally, set $A=\mathrm{Aut(}G\mathrm{)}$. Then $A$ has a unique primitive
permutation representation of degree $144$ and is the one on the set $%
\mathcal{P}^{\prime }=\left\{ N_{A}(P)x:x\in A\right\} $ by \cite{At}. Also, $%
G$ permutes primitively the elements of $\mathcal{P}^{\prime }$ again by \cite%
{At}. Hence, we may consider the isomorphic copy $\mathcal{D}_{0}=(\mathcal{P%
}^{\prime },\mathcal{B}_{0})$ of the $2$-design $\mathcal{D}$ constructed in
(1), where $\mathcal{B}_{0}$ consists of suitable $12$-subsets of $\mathcal{P%
}^{\prime }$. Note that, $\mathcal{B}_{0}^{\sigma }\neq \mathcal{B}_{0}$ as $%
\mathrm{Aut(\mathcal{D}}_{0}\mathrm{)}=G$ by (1). Thus, $\mathcal{D}_{0}^{\sigma }=(\mathcal{P}^{\prime },%
\mathcal{B}_{0}^{\sigma })$ is a isomorphic to $\mathcal{D}$, but $\mathcal{D}^{\sigma} \neq \mathcal{D}$. Therefore, $\mathcal{D}%
^{\prime }=(\mathcal{P}^{\prime },\mathcal{B}_{0}\cup \mathcal{B}%
_{0}^{\sigma })$ is a $2$-$(144,12,6)$ design admitting $A$ as a
flag-transitive automorphism group. Arguing as in (1), we see that $A$ is
the unique flag-transitive automorphism group of $\mathcal{D}^{\prime }$.

If $\mathcal{B}^{\prime }=\left\{ Ly:y\in A\right\} $, the incidence
structure $(\mathcal{P}^{\prime },\mathcal{B}^{\prime },\mathcal{I}^{\prime
})$, where 
\[
\mathcal{I}=\left\{ (N_{A}(P)x,Ly)\in \mathcal{P}^{\prime }\times \mathcal{B}%
^{\prime }:N_{A}(P)x\cap Ly\neq \varnothing \right\} \text{,} 
\]%
is an isomorphic copy of $\mathcal{D}^{\prime }$ by \cite{HM}, Lemmas 1, and
we obtain (2).
\end{proof}

\bigskip

Recall that, if $\Gamma _{1}$ and $\Gamma _{2}$ are subgroups of a group $%
\Gamma $ such $\Gamma =\Gamma _{1}\Gamma _{2}\Gamma _{1}$, then $(\Gamma
,\Gamma _{1},\Gamma _{2})$ is said \emph{triple factorization}. Moreover, $%
(\Gamma ,\Gamma _{1},\Gamma _{2})$ is \emph{non-degenerate} if $\Gamma \neq
\Gamma _{1}\Gamma _{2}$. See \cite{AP} for more details and results on
triple factorizations. As shown in the following corollary, Example \ref%
{Ex5} also provides a non-degenerate triple factorization.

\bigskip

\begin{corollary}
$\left( G,N_{G}(P),L\right) $ is a non-degenerate triple factorization.
\end{corollary}

\begin{proof}
It follows from the proof of Example \ref{Ex5} that the block-$N_{G}(P)$%
-orbit $LN_{G}(P)$ non-trivially intersect each point-$N_{G}(P)$-orbit. Thus $%
G=N_{G}(P)LN_{G}(P)$ is a triple factorization of $G$ by \cite{AP},
Proposition 3.1. Moreover, it is non-degenerate, since the cardinality of $%
N_{G}(P)L$ is smaller that the order of $G$.
\end{proof}

\bigskip

\section{Preliminary Reductions}
\medskip
We first collect some useful results on flag-transitive designs.
\medskip
\begin{lemma}\label{desdes}
Let $\mathcal{D}$ be a $2$-$(k^{2},k,\lambda )$ design and let $b$ be the
number of blocks of $\mathcal{D}$. Then the number of blocks containing each
point of $\mathcal{D}$ is a constant $r$ satisfying the following:

\begin{enumerate}
\item $r=\lambda (k+1)$;

\item $b=\lambda k(k+1)$;

\item $\left( r/\lambda \right) ^{2}>k^{2}$.
\end{enumerate}
\end{lemma}

\bigskip

\begin{lemma}
\label{PP}If $\mathcal{D}$ is a $2$-$(k^{2},k,\lambda )$ design, with $%
\lambda \mid k$, admitting a flag-transitive automorphism group $G$, then
the following hold:

\begin{enumerate}
\item $G$ acts point-primitively on $\mathcal{D}$.

\item If $x$ is any point of $\mathcal{D}$, then $G_{x}$ is a large subgroup
of $G$.

\item $\left\vert y^{G_{x}}\right\vert =(k+1)\left\vert B\cap
y^{G_{x}}\right\vert $ for any point $y$ of $\mathcal{D}$, with $y\neq x$,
and for any block $B$ of $\mathcal{D}$ incident with $x$. In particular, $%
k+1 $ divides the length of each point-$G_{x}$-orbit on $\mathcal{D}$
distinct from $\left\{ x\right\} $.
\end{enumerate}
\end{lemma}

\begin{proof}
The assertion (1) follows from \cite{Demb}, 2.3.7.c, since $r=(k+1)\lambda
>(k-3)\lambda $.

The flag-transitivity of $G$ on $\mathcal{D}$ implies $\left\vert
G\right\vert =k^{2}\left\vert G_{x}\right\vert $, $\left\vert
G_{x}\right\vert =\lambda \left( k+1\right) \left\vert G_{x,B}\right\vert $
and hence $\left\vert G\right\vert <\left\vert G_{x}\right\vert ^{3}$, which
is the assertion (2).

Let $y$ be any point of $\mathcal{D}$, $y\neq x$, and $B$ be any block of $%
\mathcal{D}$ incident with $x$. Since $(y^{G_{x}},B^{G_{x}})$ is a tactical
configuration by \cite{Demb}, 1.2.6, it follows that $\left\vert
y^{G_{x}}\right\vert \lambda =r\left\vert B\cap y^{G_{x}}\right\vert $.
Hence $\left\vert y^{G_{x}}\right\vert =(k+1)\left\vert B\cap
y^{G_{x}}\right\vert $ as $r=(k+1)\lambda $. This proves (3).
\end{proof}

\bigskip

\begin{theorem}
\label{MF1}Let $\mathcal{D}$ be a $2$-$(k^{2},k,\lambda )$, with $\lambda
\mid k$, admitting a flag-transitive automorphism group $G$. Then $G$ is
point primitive and one of the following holds:

\begin{enumerate}
\item $G$ is an affine group.

\item $G$ is an almost simple classical group.

\item $\mathcal{D}$ is isomorphic to the $2$-$(6^{2},6,2)$ design constructed
in \cite{MF}, and $~^2 G_{2}(3)^{\prime } \trianglelefteq G \leq ~^2 G_{2}(3)$.

\item $\mathcal{D}$ is isomorphic to one of the three $2$-$(6^{6},6,6)$ designs
constructed in \cite{MF}, and $G\cong $ $~^2 G_{3}(3)$.
\end{enumerate}
\end{theorem}

See \cite{MF} for a proof.\\

\medskip 

In this paper we classify the $2$-$(k^{2},k,\lambda )$ designs, with $%
\lambda \mid k$, admitting a flag-transitive automorphism group when $G$ is
as in (2) of Theorem \ref{MF1}. Since $%
~^2 G_{2}(3)^{\prime } \cong PSL_{2}(8)$, case (3) and (4) of Theorem \ref%
{MF1} are also special cases of (2).\\

\medskip

From now on we assume that $\mathrm{Soc(}G\mathrm{)}$, the socle of $G$, is
a non-abelian simple classical group. $\mathrm{Soc(}G\mathrm{)}$ will simply
be denoted by $X$. Thus $X\trianglelefteq G\leq \mathrm{Aut(}X\mathrm{)}$,
where $X$ is isomorphic to one of the following groups:

\begin{enumerate}
\item $PSL_{n}(q)$, with $n \geq 2$ and $(n,q)\neq (2,2),(2,3)$.

\item $PSp_{n}(q)^{\prime }$, with $n$ even, $n \geq 2$ and $(n,q)\neq (2,2),(2,3)$.

\item $PSU_{n}(q^{1/2})$, with $n \geq 2$, $q$ a square and $(n,q^{1/2})\neq (2,2),(2,3)$ or 
$(3,2)$.

\item $P\Omega_{m}(q)$, with $n \geq 3$, $nq$ odd and $(n,q)\neq (3,3)$.

\item $P\Omega _{n}^{\varepsilon }(q)$, with $\varepsilon = \pm$, $n$ even, $n \geq 4$ and $(n,\varepsilon) \neq (4,+)$. 
\end{enumerate}

Clearly, $q=p^{f}$, where $p$ is a prime and $f\geq 1$. It should be noted that the case $%
PSp_{4}(2)^{\prime }\cong A_{6}$ is already ruled out in Lemma 3.5 of \cite%
{MF}. Therefore, in case (2), we may actually assume that $X\cong PSp_{n}(q)$.

\bigskip

\begin{lemma}
\label{lambada}$\lambda \geq 2$.
\end{lemma}

\begin{proof}
Assume that $\lambda =1$. Then $\mathcal{D}$ is a translation plane of order 
$k=p^{t}$, $t\geq 1$, and $X$ is an elementary
abelian $p$-group of order $k$ by \cite{Wa}, Theorems 2 and 4. This is a
contradiction, since $X$ is a non-abelian simple group by our
assumption. Thus, $\lambda \geq 2$.
\end{proof}

\bigskip 

In the sequel, we always assume that $\lambda \geq 2$ without recalling
Lemma \ref{lambada}. 

\bigskip

\begin{lemma}
\label{Orbits}Let $\mathcal{D}$ be a $2$-$(k^{2},k,\lambda )$ design, with $%
\lambda \mid k$, admitting a flag-transitive automorphism group $G$. If $x$
is any point of $\mathcal{D}$, then $\frac{k+1}{\gcd (k+1,\left\vert \mathrm{%
Out(}X\mathrm{)}\right\vert )}$ divides $\left\vert y^{X_{x}}\right\vert $
and hence $\left\vert X_{x}\right\vert $.
\end{lemma}

\begin{proof}
Let $x$ be any point of $\mathcal{D}$. If $y$ is a point of $\mathcal{D}$,
with $y\neq x$, then $\left\vert y^{X_{x}}\right\vert =\frac{\left\vert
B\cap y^{G_{x}}\right\vert \left( k+1\right) }{\mu }$, where $\mu \left\vert
y^{X_{x}}\right\vert =\left\vert y^{G_{x}}\right\vert $, by Lemma \ref{PP}%
(3), as $X_{x}\trianglelefteq G_{x}$. On the other hand, $\mu $ divides $%
\left\vert \mathrm{Out(}X\mathrm{)}\right\vert $, as $\mu =\frac{\left[
G_{x}:X_{x}\right] }{\left[ G_{x,y}:X_{x,y}\right] }$. Therefore $\frac{k+1}{%
\gcd (k+1,\left\vert \mathrm{Out(}X\mathrm{)}\right\vert )}$ divides $%
\left\vert y^{X_{x}}\right\vert $ and hence $\left\vert X_{x}\right\vert $.
\end{proof}

\bigskip

Recall that, if $M$ is a maximal subgroup of an almost simple group $\Gamma $
such that $M\cap \mathrm{Soc(}\Gamma \mathrm{)}$ is a non-maximal in $%
\mathrm{Soc(}\Gamma \mathrm{)}$, then $M$ is called a \emph{novel} maximal
subgroup, or, simply, \emph{novelty}. If $K$ is a maximal subgroup of $%
\mathrm{Soc(}\Gamma \mathrm{)}$ containing $M\cap \mathrm{Soc(}\Gamma 
\mathrm{)}$ and such that $N_{\Gamma }(K)\mathrm{Soc(}\Gamma \mathrm{)}\neq
\Gamma $, then $M$ is called a \emph{type 1 novelty with respect to }$K$.
More details on types and properties of novelties of almost simple groups can be found in \cite{BHRD}.

\bigskip

\begin{lemma}
\label{Nov}Let $\Omega $ be a set of
imprimitivity for $X$ on the point set of $\mathcal{D}$ such that $X_{\Lambda
}$ is maximal in $X$ for $\Lambda \in \Omega $. If $k+1$ does not divide $\left\vert 
\mathrm{Out(}X\mathrm{)}\right\vert $, then the followings hold:

\begin{enumerate}
\item $X_{\Lambda }^{G_{x}}$ is split in at least two distinct conjugacy classes under $X$.

\item $G_{x}$ is a type 1 novelty with respect to $X_{\Lambda }$.
\end{enumerate}
\end{lemma}

\begin{proof}
Let $\Omega $ be a set of imprimitivity for $X$ on the point set of $\mathcal{%
D}$ such that $X_{\Lambda }$ is maximal in $X$ for $\Lambda \in \Omega $.
Clearly, $X_{\Lambda }=N_{X}(X_{\Lambda })$. Let $x\in \Lambda $, since $%
X_{x}\trianglelefteq G_{x}$ and $G_{x}$ is maximal in $G$, it results $G_{x}=N_{G}(X_{x})$ and hence $X_{x}\leq N_{X}(X_{x})\leq
G_{x}$. Thus, $X_{x}=N_{X}(X_{x})$.

If $N_{G_{x}}(X_{\Lambda })=G_{x}$, then $G_{x}\leq X_{\Lambda }G_{x}$ and
hence either $X_{\Lambda }\leq G_{x}$ or $G=X_{\Lambda }G_{x}$, as $G_{x}$ is
maximal in $G$. The former implies $X_{\Lambda }\leq X_{x}$, whereas $X_{x}<X_{\Lambda }$, and this case is excluded. Thus, $%
G=X_{\Lambda }G_{x}$ and hence $X_{\Lambda }\vartriangleleft G$, as $%
N_{G_{x}}(X_{\Lambda })=G_{x}$ by our assumption. So, we get $X_{\Lambda
}\vartriangleleft X$, which is a contradiction, as $X$ is simple and $X_{x}<X_{\Lambda}<X$. Thus, $\left[
G_{x}:N_{G_{x}}(X_{\Lambda })\right] >1$.

Assume that $X_{y}<X_{\Lambda }$ for some $y\notin \Lambda $. Clearly, $%
\Lambda $ is a $X_{\Lambda }$-orbit and is union of $X_{y}$-orbits distinct from $%
\left\{ y\right\} $. Each $X_{y}$-orbit distinct from $\left\{ y\right\} $
is of length divisible by $\frac{k+1}{(k+1,\left\vert \mathrm{Out(}X\mathrm{)%
}\right\vert )}$ by Lemma \ref{Orbits}. Moreover, $\frac{k+1}{(k+1,\left\vert \mathrm{Out(}X%
\mathrm{)}\right\vert )}>1$, since $k+1$ does not divide $\left\vert \mathrm{%
Out(}X\mathrm{)}\right\vert $ by our assumption. So $\frac{k+1}{(k+1,\left\vert \mathrm{Out(}%
X\mathrm{)}\right\vert )}$ divides $\left\vert \Lambda \right\vert $, and
hence $k^{2}$, since $k^{2}=\left[ X:X_{\Lambda }\right] \left[ X_{\Lambda
}:X_{x}\right] =\left[ X:X_{\Lambda }\right] \left\vert \Lambda \right\vert $%
, a contradiction. Thus, for each $X_{x^{\prime
}}<X_{\Lambda }$ if, and only if, $x^{\prime }\in \Lambda $.

By (1), there is $\alpha \in G_{x}-N_{G_{x}}(X_{\Lambda })$ such that $X_{\Lambda }^{\alpha
}\neq X_{\Lambda }$. Also, $X_{x}\leq X_{\Lambda }^{\alpha }\cap X_{\Lambda }$, with both $X_{\Lambda }^{\alpha }$ and $X_{\Lambda }$ maximal in $X$. Suppose that there is $\beta \in X$ such that $X_{\Lambda }^{\alpha
}=X_{\Lambda }^{\beta }$. Then $X_{\Lambda }^{\alpha }=X_{\Lambda ^{\beta }}$
and hence $X_{\Lambda }^{\alpha }$ preserves the block of imprimitivity $%
\Lambda ^{\beta }$. Then the above argument implies $x\in \Lambda ^{\beta }$
and hence $\Lambda ^{\beta }=\Lambda $, as $x\in \Lambda $ too. Then $%
X_{\Lambda }^{\alpha }=X_{\Lambda ^{\beta
}}=X_{\Lambda }$, but $\alpha \in G_{x}-N_{G_{x}}(X_{\Lambda })$, and so we
obtain a contradiction. Thus $X_{\Lambda }^{\alpha }$ and $X_{\Lambda }$
lies in distinct conjugacy classes under $X$. This proves (1).

Assume that $G=N_{G}(X_{\Lambda })X$. Since $\left[ G_{x}:N_{G_{x}}(X_{%
\Lambda })\right] >1$, there are $X_{x_{1}},X_{x_{2}}<X_{\Lambda }$ such
that $X_{x_{1}}$ and $X_{x_{2}}$ are conjugate in $N_{G}(X_{\Lambda })$ but
not in $X_{\Lambda }$ by \cite{BHRD}, Proposition 1.3.10. Nevertheless, $%
x_{1},x_{2}\in \Lambda $ by the above argument. Hence $X_{x_{1}}$ and $%
X_{x_{2}}$ are conjugate in $X_{\Lambda }$, as this one acts transitively on 
$\Lambda $, and we reach a contradiction. Thus $G\neq N_{G}(X_{\Lambda })X$ and hence $%
G_{x}$ is a type 1 novelty with respect to $X_{\Lambda }$. Thus, we obtain (2).
\end{proof}

\bigskip

\section{Reductions for $X$ based on primitive prime divisors of its order}

The first part of this section is devoted to the resolution of the case $%
X\cong PSL_{2}(q)$, which is achieved by combining some group-theoretical
results on the structure of $PSL_{2}(q)$ given in \cite{Hup}, together with some numerical
constraints on the diophantine equation $k^{2}=[X:X_{x}]$ provided in \cite{Rib}.

The second part focuses on the case $n\geq 3$, where some tools are developed in order to control the structure of $X_{x}$. More precisely, it is shown that $%
\left\vert X\right\vert $ is divisible by $\Phi _{ef}^{\ast }(p)$, the
primitive part of $p^{e}-1$ (more details on $\Phi _{ef}^{\ast }(p)$ are
provided below). Then, in Theorem \ref{daleko}, is proven the central result
of this section which states that either $X_{x}$ is large subgroup of $X$
and $\left( \Phi _{ef}^{\ast }(p),\left\vert X_{B}\right\vert \right) >1$,
or $\left( \Phi _{ef}^{\ast }(p),k+1\right) >1$, or $\Phi _{ef}^{\ast
}(p)\mid k^{2}$ and $\left( \Phi _{ef}^{\ast }(p),\left\vert
X_{B}\right\vert \right) =\Phi _{ef}^{\ast }(p)^{1/2}$. In the final part of this section,
a lower bound for $\left\vert X_{B}\right\vert _{p}$ is exhibited, which will play an
important role to exclude the case $\left( \Phi _{ef}^{\ast }(p),k+1\right) >1$.

\bigskip

Throughout the paper, if $m$ is any integer and $w$ is a prime, the symbols $m_{w}$ and $m_{w'}$ will denote the integers $\max_{i \geq 0}\lbrace w^{i} : w^{i} \mid m \rbrace$ and $m/m_{w}$ respectively.
\bigskip

\begin{proposition}
\label{Due nonVale}If $X\cong PSL_{2}(q)$, then $q=8$ and one of the
following holds:

\begin{enumerate}
\item $\mathcal{D}$ is isomorphic to the $2$-$(6^{2},6,2)$ design
constructed in \cite{MF} and $PSL_{2}(8)\trianglelefteq G\leq P\Gamma
L_{2}(8)$.

\item $\mathcal{D}$ is isomorphic to one of the three $2$-$(6^{2},6,6)$ designs
constructed in \cite{MF} and $G\cong P\Gamma L_{2}(8)$.
\end{enumerate}
\end{proposition}

\begin{proof}
Let $X\cong PSL_{2}(q)$. Then $\frac{k+1}{(k+1,(2,q-1)f)}\mid \left\vert
X_{x}\right\vert $, as $\left\vert \mathrm{Out(}X\mathrm{)}\right\vert
=(2,q-1)f$, by Lemma \ref{Orbits}. Moreover, $q\neq 4,5,9$, since $%
PSL_{2}(4)\cong PSL_{2}(5)\cong A_{5}$ and $PSL_{2}(9)\cong A_{6}$, and both these
cases are ruled out in Lemma 3.5 of \cite{MF}.

Assume that $k+1$ divides $(2,q-1)f$. Thus, $k^{2}<(2,q-1)^{2}f^{2}$. Moreover, $%
q>11$ and $f\geq 2$, as $\mathcal{D}$ is non-trivial. Hence $%
q+1<(2,q-1)^{2}f^{2}$, since $q+1$ is the minimal primitive permutation
representation of $X$ for $q>11$ by \cite{KL}, Theorem 5.2.2. So, $%
q=3^{3} $ and $3\leq k\leq 5$. Actually, $k=5$ as $k+1$ must divide $%
(2,q-1)f $. However, this case cannot occur, since $PSL_{2}(3^{3})$ has no
transitive permutation representations of degree less then $28$.

Assume that $\frac{k+1}{(k+1,(2,q-1)f)}>1$. Let $U$ be a Sylow $u$-subgroup
of $X_{x}$, where $u$ is a prime divisor of $\frac{k+1}{(k+1,(2,q-1)f)}$,
and assume that $U$ fixes a further point $y$ of $\mathcal{D}$. So, $%
\left\vert y^{X_{x}}\right\vert $ is coprime to $u$. This is a
contradiction, since $\frac{k+1}{(k+1,(2,q-1)f)}$ divides $\left\vert
y^{X_{x}}\right\vert $ by Lemma \ref{Orbits}. Therefore $x$ is the unique
point fixed by $U$ and hence $N_{X}(U)\leq X_{x}$. Then $X_{x}=N_{X}(U)$ and
either $X_{x}\cong U:Z_{\frac{q-1}{(2,q-1)}}$, or $X_{x}\cong D_{\frac{%
2(q+\epsilon )}{(2,q-1)}}$ and $\epsilon =\pm 1$, according to whether $u=p$
or $u\neq p$, respectively, by \cite{Hup}, Satz 8.5.(a).

Assume that $X_{x}\cong U:Z_{\frac{q-1}{(2,q-1)}}$. Then $k^{2}=\left[
X:X_{x}\right] =q+1$ and hence $k=3$ and $q=8$ by \cite{Rib}, A5.1, since $k \geq 3$.
However, this case is ruled out in Lemma 3.6 of \cite{MF}, since $%
PSL_{2}(8)\cong $ $^{2}G_{2}(3)^{\prime }$.

Assume that $X_{x}\cong D_{\frac{2(q+\epsilon )}{(2,q-1)}}$, where $\epsilon
=\pm 1$. Then $k^{2}=\frac{q(q-\epsilon )}{2}$. If $q$ is even, then $%
q/2=2^{f-1}$, with $f$ odd, and $z^{2}=2^{f}-\epsilon $ for some positive
integer $z$, since $k^{2}=\frac{q}{2}(q-\epsilon )$. Also $f>1$, since $X$
is simple, and hence $f\geq 3$. Then $\epsilon =-1$ and $f=3$ by \cite{Rib},
A3.1 and A5.1. Thus $q=8$ and $k=6$, and hence assertions (1) and (2) follow
from Theorem 3.10 of \cite{MF}, since $PSL_{2}(8)\cong $ $%
^{2}G_{2}(3)^{\prime }$.

Assume that $q$ is odd. Then $q$ is a square and $q^{1/2}$ is the highest
power of $p$ dividing $k$. Moreover, we saw that $q>9$. Clearly, $k+1$ divides 
$\frac{q(q-\epsilon )}{2}-1$ and hence $(q-2\epsilon )\frac{(q+\epsilon )}{2}
$. By Lemma \ref{PP}(3), $k+1$ divides $\left\vert G_{x}\right\vert $, which is 
$2f\left( q+\epsilon \right) $, and hence it divides $\frac{(q+\epsilon )}{2}%
f_{p^{\prime }}$, as $p\mid k$ and $q$ is odd. Then there is an integer $%
c\geq 1$ such that 
\begin{equation}
2c(k+1)=\left( q+\epsilon \right) f_{p^{\prime }}\text{.}  \label{equal}
\end{equation}%
From (\ref{equal}) we derive that $q^{1/2}$ divides $2c-\epsilon
f_{p^{\prime }}$, as $q^{1/2}\mid k$.

Note that $2c>f_{p^{\prime }}$. Indeed, if $2c\leq f_{p^{\prime }}$ then $%
k\geq q+\epsilon -1$ and hence $\left( q+\epsilon -1\right) ^{2}\leq \frac{%
q(q-\epsilon )}{2}$, which has no solutions for $q > 9$. Then $2c-\epsilon
f_{p^{\prime }}=\theta q^{1/2}$ for some $\theta \geq 1$, and (\ref{equal})
becomes $\left( \theta q^{1/2}+\epsilon f_{p^{\prime }}\right) k+\theta
q^{1/2}=qf_{p^{\prime }}$. Therefore, we obtain $k=q^{1/2}\frac{%
q^{1/2}f_{p^{\prime }}-\theta }{\theta q^{1/2}+\epsilon f_{p^{\prime }}}$.

If $\epsilon =1$, then $k<q^{1/2}f_{p^{\prime }}$ and hence $k^{2}=\frac{%
q(q-1)}{2}$ implies $q-1<2f^{2}$, which has no solutions for $q>9$.
Therefore, we have $\epsilon =-1$ and $k=q^{1/2}\frac{q^{1/2}f_{p^{\prime
}}-\theta }{\theta q^{1/2}-f_{p^{\prime }}}$. We have seen that $q-1\geq
2f^{2}$ as $q>9$. Then $q^{1/2}>f_{p^{\prime }}$ and hence $\theta
q^{1/2}-f_{p^{\prime }} > (\theta -1)q^{1/2}$. So, $k < q^{1/2}\frac{%
q^{1/2}f_{p^{\prime }}-\theta }{\left( \theta -1\right) q^{1/2}}%
<q^{1/2}f_{p^{\prime }}$ for $\theta >1$, and we obtain a
contradiction as above. Hence, $\theta =1$. Then $k=q^{1/2}\frac{%
q^{1/2}f_{p^{\prime }}-1}{q^{1/2}-f_{p^{\prime }}}$ which, substituted in $%
k^{2}=\frac{q(q+1)}{2}$, yields $f_{p^{\prime }}=\frac{q^{2}+q+2}{\sqrt{q}%
\left( q+3\right) }$. However, such value of $f_{p^{\prime }}$ is not an
integer as $q$ is odd and $q>1$. This completes the proof.
\end{proof}

\bigskip

\begin{remark}
\label{mag3}From now on we assume that $X\ncong PSL_{2}(q)$. Bearing in mind
the isomorphisms $PSp_{2}(q)\cong PSL_{2}(q)$, $P\Omega _{3}(q)\cong
PSL_{2}(q)$ for $q$ odd, $P\Omega _{4}^{-}(q)\cong PSL_{2}(q^{2})$, and $%
PSU_{2}(q^{1/2})\cong PSL_{2}(q)$ when $q$ is a square, and the fact that $%
P\Omega _{4}^{+}(q)$ is non-simple, in the sequel we assume that

\begin{enumerate}
\item $n\geq 3$ for $X \cong PSL_{n}(q)$;
\item $n\geq 3$ for $X \cong PSU_{n}(q^{1/2})$ for $q$ square and $(n,q^{1/2}) \neq (3,2)$;

\item $n\geq 4$ for $X \cong PSp_{n}(q)$;

\item $n\geq 5$ for $X \cong P\Omega _{5}(q)$;

\item $n\geq 6$ for $X \cong P\Omega _{n}^{\pm }(q)$.
\end{enumerate}
\end{remark}

\bigskip
Now we want to transfer some of the Saxl's arguments contained in \cite{Saxl} to our context in order to rule out the parabolic case. In \cite{Saxl} the author shows that $r \leq (v-1)_{p}$, where $r$ is the replication number of a linear space with $v$ points admitting $G$ as a flag-transitive automorphism group, whenever the conclusions of his Lemma 2.6 hold. Then, he proves that $v \geq (v-1)^{2}_{p} \geq r^2$ and hence he rules out the parabolic case when the conclusions of Lemma 2.6 hold. From a group-theoretical point of view, Saxl shows that
\begin{equation}\label{ThnxSaxl}
[G:G_{x}] < \left([G:G_{x}]-1\right) _{p}^{2}, 
\end{equation}
where $v=[G:G_{x}]$ and $G_{x}$ is a maximal parabolic subgroup of $G$, is never fulfilled whenever the conclusions of his Lemma 2.6 hold, and hence no linear spaces occur in these cases.\\
We may use the same argument of \cite{Saxl}, with his $r$ replaced by our $r/\lambda=k+1$, thus obtaining (\ref{ThnxSaxl}) and hence excluding the parabolic case. This is possible since Lemma 2.1.(ii)--(iii) of \cite{Saxl} are replaced by our Lemmas \ref{desdes}(3) and \ref{PP}(3) respectively. 

\bigskip

\begin{proposition}\label{greatsaxl}
Let $x$ be any point of $\mathcal{D}$, then $G_{x}$ is not a parabolic subgroup of $G$ and $p \mid k$.
\end{proposition}

\begin{proof}

Assume that $G_{x}$ is a parabolic subgroup of $G$.
Suppose that $X$ is not isomorphic to $PSL_{n}(q)$, or to $P\Omega _{n}^{+}(q)$ for $n/2
$ odd. Then $k+1=p^{s}$ for some $s\geq 1$ by \cite{Saxl}, Lemma
2.6, and by Lemma \ref{PP}(3). Then $k+1$ divides $\left(
[G:G_{x}]-1\right) _{p}$, as $k^{2}=[G:G_{x}]$, and hence we obtain (\ref{ThnxSaxl}) by Lemma \ref{desdes}(3). \\
Let $i$ be the dimension of the subspace of $PG_{n}(q)$ preserved by $G_{x}$. In \cite{Saxl} it is shown that $\left([G:G_{x}]-1\right) _{p}$ is a divisor of one of the following numbers:

\begin{enumerate}
\item $q$ for $X\cong PSp_{n}(q)$, $P\Omega _{n}(q)\,$;

\item $q$, $q^{2}$ or $q^{3}$ for $X\cong $ $PSU_{n}(q^{1/2})$, $n>3$ and $(n,q) \neq (4,2)$,
according to whether $i\neq \lbrack n/2 \rbrack -1$, or $n$ even and $i=(n-2)/2$, or $n$ odd
and $i=(n-3)/2$ respectively;

\item $q^{2}$ or $8$ for $P\Omega _{n}^{\varepsilon}(q)$, with $\varepsilon = \pm$, and  $n/2$ even and $n/2>4$ for $\varepsilon=+$;

\item $(3,q)q$ or $8$ for $P\Omega _{8}^{+}(q)$.
\end{enumerate}

Hence, the same argument used by Saxl in \cite{Saxl} can be deployed to show that (\ref{ThnxSaxl}) is not fulfilled in cases (1)--(4). So these cases are
ruled out. Clearly, $X\cong PSU_{4}(2)$ cannot occur since $PSU_{4}(2) \cong PSp_{4}(3)$. Finally, $X\cong $ $PSU_{3}(q^{1/2})$ cannot occur. Indeed, $%
k^{2}=q^{3/2}+1$ has no integer solutions by \cite{Rib}, A5.1, since $q^{1/2} \neq 2$.

Assume that $X\cong P\Omega _{n}^{+}(q)$ with $n/2$ odd. As pointed out in \cite{Saxl}, it is still true that $X$ has a subdegree which is a power of $p$, provided that $G_{x}$ is not parabolic of type $n/2$. Hence we still obtain (\ref{ThnxSaxl}). However, it is not fulfilled since $\left([G:G_{x}]-1\right) _{p}$ is $q^{2}$ or $8$. Hence, the cases where $G_{x}$ is parabolic, but not of type $n/2$, are ruled out.\\
Assume that $G_{x}$ is parabolic of type $n/2$. We may repeat the argument of Saxl in \cite{Saxl}, Section 5.(1), with $r/\lambda$ in the role of his $r$ and we obtain that $k^{2}>q^{n/4(n/2-1)}$ and $k+1=r/\lambda \leq q(q^{n/2}-1)$, and hence $n=10$, as $(r/\lambda)^{2}>k^{2}$. Thus
$$k^{2}=\left( q+1\right) ^{2}\left( q^{2}-q+1\right) \left(
q^{2}+1\right) \left( q^{4}+1\right).$$
Hence $ q^{2}-q+1$, which is equal to $(q-1)^2+(q-1)+1$, is a square, as it is coprime to the other factors of $\frac{k^{2}}{(q+1)^{2}}$. However this is impossible by \cite{Rib}, A7.1, since $q$ is a power of a prime.  

Finally, assume that $X\cong PSL_{n}(q)$. Again, proceeding as in \cite{Saxl}%
, Section 7.(1), with $r/\lambda $ in the role of $r$, we obtain $i\leq
2$. Assume that $i=2$. Then $k^{2}=\frac{(q^{n}-1)(q^{n-1}-1)}{(q^{2}-1)(q-1)%
}$. Since $\left( \frac{q^{n}-1}{q-1},\frac{q^{n-1}-1}{q-1}\right) =1$,
either $n$ is odd and $\frac{q^{n}-1}{q-1}$ is a square, or $n$ is even and $%
\frac{q^{n-1}-1}{q-1}$ is a square. Then either $(n,q)=(5,3)$ or $(n-1,q)=(5,3)$,
respectively, by \cite{Rib}, A8.1. However, both case are ruled out as they
yield $k^{2}=1210$ and $\allowbreak 11\,011$ respectively. Thus $i=1$, $k^{2}=\frac{q^{n}-1}{q-1}$, and hence $%
(n,q,k)=(4,7,20)$ or $(5,3,11)$ again by \cite{Rib}, A8.1. If $(n,q,k)=(4,7,20)$,
then $b=21\cdot 20\cdot \lambda $ with $\lambda \mid 20$ and $\lambda \geq 2$%
. However, $ PSL_{4}(7) \unlhd G\leq PGL_{4}(7)$ has no transitive permutation representations
of degree $b$ by \cite{BHRD}, Tables 8.8--8.9. So, $(n,q,k)=(5,3,11)$, and
hence $b=12\cdot 11^{2}$ since $\lambda \mid 11$ and $\lambda \geq 2$.
As $G\cong PSL_{5}(3)$ has no transitive permutation representations
of degree $12\cdot 11^{2}$ by \cite{BHRD}, Tables 8.18--8.19, this case is ruled out. Therefore $G_{x}$ is not a parabolic subgroup of $G$, and hence $[G:G_{x}]$ is divisible by $p$ by Tits' Lemma (see \cite{Sei}, Theorem 1.6). Thus $p \mid k$. 
\end{proof}

\bigskip

A divisor $r$ of $q^{e}-1$ that is coprime to each $q^{i}-1$ for $i<e$ is
said to be a \emph{primitive divisor}, and we call the largest primitive
divisor $\Phi _{e}^{\ast }(q)$ of $q^{e}-1$ the \emph{primitive part} of $%
q^{e}-1$. One should note that $\Phi _{e}^{\ast }(q)$ is strongly related to cyclotomy
in that it is equal to the quotient of the cyclotomic number $\Phi _{e}(q)$ and $(n,\Phi _{e}(q))$
when $e>2$. Also $\Phi _{e}^{\ast }(q)>1$ for $e>2$ and $(q,e)\neq (2,6)$ by
Zsigmondy's Theorem (for instance, see \cite{Rib}, P1.7).

The order of $X$ is clearly divisible by $\Phi _{ef}^{\ast }(p)$, where $e=n$ for $X$ isomorphic to $PSL_{n}(q)$, $%
PSp_{n}(q)$, $P\Omega _{n}^{-}(q)$, or $PSU_{n}(q^{1/2})$, with $q$ square
and $n$ odd, $e=n-1$ for $P\Omega _{n}(q)$, with $nq$ odd, or $%
PSU_{n}(q^{1/2})$, with $q$ square and $n$ even, and $e=n-2$ for $P\Omega _{n}^{+}(q)$.

\bigskip

\begin{lemma}
\label{cyclont}$\Phi _{ef}^{\ast }(p)>1$.
\end{lemma}

\begin{proof}
Assume that $\Phi _{ef}^{\ast }(p)=1$. Since $n\geq 3$ by Remark \ref{mag3},
it follows that $(ef,p)=(6,2)$ (see above). Then $X$ is isomorphic to one of
the groups $PSL_{3}(4)$, $PSU_{4}(2)$, $PSL_{6}(2)$, $PSp_{6}(2)$, $P\Omega
_{6}^{-}(2)$ or $P\Omega _{8}^{+}(2)$.

Assume that $X\cong PSL_{3}(4)$. Then $35\mid \left\vert X_{x}\right\vert $,
since $\left[ X:X_{x}\right] =k^{2}$, but $X$ does not contain such a group
by \cite{At}.

Assume that $X\cong PSU_{4}(2)$. Then $PSU_{4}(2)\trianglelefteq G\leq
PSU_{4}(2):Z_{2}$. Also $5\mid \left\vert X_{x}\right\vert $ and $\left[
X:X_{x}\right] =k^{2}$, with $k\mid 72$ and $k+1\mid \left\vert G\right\vert 
$. Easy computations show that $k=3,4,8$ or $9$. However, $X$ does not have $k^{2}$ as a transitive permutation of degree for any such values of $k$ by \cite{At}.

Assume that $X\cong PSL_{6}(2)$. Thus $2\cdot 5\cdot 31\mid \left\vert
X_{x}\right\vert $, since $\left[ X:X_{x}\right] =k^{2}$, and hence either $%
X_{x}\cong SL_{5}(2)$ or $X_{x}\cong E_{2^{5}}:SL_{5}(2)$ by \cite{BHRD},
Tables 8.24 and 8.25, and by \cite{At}. Since  $\left\vert X_{x}\right\vert _{7}=7$ and $\left\vert X \right\vert _{7}=7^{2}$, it follows that $\left[ X:X_{x}\right]_{7}=7 $, but this contradicts $\left[
X:X_{x}\right] =k^{2}$.

Assume that $X\cong PSp_{6}(2)$. Then $G=X$, since $\mathrm{Out(}X\mathrm{)}%
=1$. Also $\left[ G:G_{x}\right] =k^{2}$ implies $2\cdot 5\cdot 7\mid
\left\vert G_{x}\right\vert $. Then either $G_{x}=A_{7}$ and $k=24$, or $%
G_{x}\cong S_{8}$ and $k=6$ by \cite{At}. The former case is ruled out,
since it contradicts Lemma \ref{PP}(3). Indeed, $k+1=5^{2}$ does not divide $%
\left\vert G_{x}\right\vert $. Thus, $G_{x}\cong S_{8}$, $k=6$ and hence $%
\lambda =2,3$ or $6$, as $\lambda \mid k$ and $\lambda \geq 2$. However, $S_{8}$ does not have transitive permutation representations of degree 
$r=14,21$ or $42$ by \cite{At}. Thus $X\cong PSp_{6}(2)$ is
excluded.

Assume that $X\cong P\Omega _{6}^{-}(2)$. Then $5\mid \left\vert
X_{x}\right\vert $, since $\left[ X:X_{x}\right] =k^{2}$, and hence $k\mid
2^{3}\cdot 3^{2}$. Also $\frac{k+1}{(k+1,2)}$ divides $\left\vert
X\right\vert $ by Lemma \ref{Orbits}, since $\mathrm{Out(}X\mathrm{)}\cong
Z_{2}$. Easy computations show that $k=3,4,8$ or $9$. However, for these $k$
the group $X$ does not have a transitive permutation representation of
degree $k^{2}$ by \cite{At}, and so $X\cong P\Omega _{6}^{-}(2)$ is
ruled out too.

Assume that $X\cong P\Omega _{8}^{+}(2)$. Therefore $21\mid \left\vert
X_{x}\right\vert $, since $\left[ X:X_{x}\right] =k^{2}$, and hence $%
k^{2}\mid 2^{12}\cdot 3^{4}\cdot 5^{2}$. Also $\left[ X:X_{x}\right] $ must
be divisible by $120$, $135$ or $960$ by \cite{At}. Then $15$ divides $\left[
X:X_{x}\right] $ in each of these cases. Actually, $15\mid k$ as $\left[
X:X_{x}\right] =k^{2}$. Then $k=15\theta $, with $\theta \mid 2^{6}\cdot 3$
and $\theta >1$, and $\frac{k+1}{(k+1,2)}\mid \left\vert X\right\vert $ as $%
\mathrm{Out(}X\mathrm{)}\cong S_{3}$. It is straightforward to check that $%
\frac{k+1}{(k+1,2)}\nmid \left\vert X\right\vert $ for any $\theta \mid
2^{6}\cdot 3$ and $\theta >1$, and hence $X\cong P\Omega _{8}^{+}(2)$ cannot
occur.
\end{proof}

\bigskip

The following two lemmas play a central role in proving that $X_{x}$ is large subgroup of $X$ when $\left(
\Phi _{ef}^{\ast }(p),\lambda \left\vert G_{x,B}\right\vert \right) >1$.

\bigskip

\begin{lemma}
\label{Haide}Let $(x,B)$ be any flag of $\mathcal{D}$. Then $\left[ G_{x}:X_{x}%
\right] <1+ef$ provided that one of the following holds:

\begin{enumerate}
\item $\lambda \left\vert G_{x,B}\right\vert $ is odd;

\item $\lambda \left\vert G_{x,B}\right\vert \equiv 2 \pmod 4$, $\left(
\Phi _{ef}^{\ast }(p),\lambda \left\vert G_{x,B}\right\vert \right) >1$ and
either $X\cong PSU_{n}(q^{1/2})$, with $n$ even, or $X\cong P\Omega
_{8}^{+}(q)$.
\end{enumerate}
\end{lemma}

\begin{proof}
Suppose that $\lambda \left\vert G_{x,B}\right\vert $ is odd. Let $\sigma $
be an involution of $G$ and let $y$ be any point such that $y^{\sigma }\neq
y $. Then $\sigma $ permutes the $\lambda $ blocks incident with $%
y,y^{\sigma } $, and hence $\sigma $ fixes one of them, say $C$, as $\lambda 
$ is odd. Thus $k$ is even, since $\left\vert G_{C} \right\vert=k\left\vert G_{y,C}\right\vert$ and $\left\vert G_{y,C}\right\vert $ is odd,
and hence $\lambda (k+1)\left\vert G_{x,B}\right\vert $ is odd. Therefore, $\left[
G_{x}:X_{x}\right] $ is a divisor of $\left\vert \mathrm{Out(}X\mathrm{)%
}\right\vert _{2^{\prime }}$, since $\left\vert X_{x}\right\vert =\frac{\left\vert
G_{x}\right\vert }{\left[ G_{x}:X_{x}\right] }=\frac{\lambda (k+1)\left\vert
G_{x,B}\right\vert }{\left[ G_{x}:X_{x}\right] }$. Since $\left\vert \mathrm{%
Out(}X\mathrm{)}\right\vert _{2^{\prime }}$ is less than, or equal to, $nf$, $%
nf/2 $, $3f$ or $f$ for $X\cong PSL_{n}(q)$, $PSU_{n}(q^{1/2})$, $P\Omega
_{8}^{+}(q)$ or $X$ as in the remaining cases, respectively, by \cite{KL},
Table 5.1.A, and as $n\geq 3$ by Remark \ref{mag3}, it follows that $\left[ G_{x}:H_{x}\right] <1+ef$ for any $X$.

Suppose that $\lambda \left\vert G_{x,B}\right\vert \equiv 2 \pmod 4$, $%
\left( \Phi _{ef}^{\ast }(p),\lambda \left\vert G_{x,B}\right\vert \right)
>1 $, and also that either $X\cong PSU_{n}(q^{1/2})$, with $n$ even, or $X\cong P\Omega
_{8}^{+}(q)$. Assume that $k$ is odd. By \cite{LS}, one of the following holds:
\begin{itemize}
\item[(i).] $G_{x}$ is a parabolic subgroup of $G$, and $q$ is even.
\item[(ii).] $G_{x}$ is the stabilizer of a non-singular subspace of $PG_{n-1}(q)$, and $q$ is odd.
\item[(iii).] $G_{x}$ is a $\mathcal{C}_{4}$-subgroup of $G$.
\item[(iv).] $G_{x}$ stabilizes an orthogonal\ decomposition into $t$ subspaces each of dimension $n/t$ in its natural projective module $V$.
\item[(v).] $X\cong P\Omega _{8}^{+}(q)$, $q$ is a prime such that $q\equiv
\pm 3 \pmod 8$, and $X_{x}$ is isomorphic either to $\Omega _{8}^{+}(2)$
or to $2^{3}\cdot 2^{6}\cdot PSL_{3}(2)$.
\end{itemize}
Case (i) is ruled out by Proposition \ref{greatsaxl}. Since $\lambda \left\vert G_{x,B}\right\vert $
divides $\left\vert G_{x}\right\vert $, and since $\left( \Phi _{ef}^{\ast
}(p),\lambda \left\vert G_{x,B}\right\vert \right) >1$, it follows that $%
\left\vert G_{x}\right\vert $ is divisible by a primitive prime divisor $u$
of $p^{ef}-1$. Thus, case (iii) is excluded by \cite{KL}, Section 4.4 and Proposition 5.2.15.(i).\\
Assume that case (ii) holds. If $X\cong PSU_{n}(q^{1/2})$, with $n$ even, then $G_{x}$ is the
stabilizer of a non-singular point of $PG_{n-1}(q)$ by \cite{KL},
Proposition 4.1.4.(II), since $u\mid \left\vert G_{x}\right\vert $. Thus, 
\[
k^{2}=q^{\left( n-1\right) /2}\frac{q^{n/2}-1}{q^{1/2}+1} 
\]%
and hence $f\equiv 0 \pmod 4$, as $q^{\left( n-1\right) /2}$ must be a
square and $n$ is even. So $\frac{q^{n/2}-1}{q^{1/2}+1}$ is even, as $q$ is odd, whereas $k$
is odd.

If $X\cong P\Omega _{8}^{+}(q)$, then $G_{x}$ is either the stabilizer in $G$ of a
non-singular point, or the stabilizer in $G$ of a non-singular line of type $-$ of $PG_{7}(q)$ by \cite{KL}, Propositions 4.1.6.(II), since $q$ is odd and $\left( \Phi _{6f}^{\ast
}(p),\lambda \left\vert G_{x,B}\right\vert \right) >1$. In the former case, $k^{2}=q^{3}(q^{2}-1)\frac{q^{2}+1}{2}$ and hence $q^{2}-1$ is a square. So, $q^{2}=z^{2}+1$ for some positive integer $z$, but this contradicts \cite{Rib}, A3.1. Therefore, $G_{x}$ is the stabilizer in $G$ of a non-singular line of type $-$ of $PG_{7}(q)$, then  
\[
k^{2}=q^{6}\frac{(q-1)^{2}}{(4,q-1)}\allowbreak \left( q^{2}+1\right) \left(
q^{2}+q+1\right). 
\]%
So, $q^{2}+q+1$ is a square, but this contradicts \cite{Rib}, A7.1.

Assume that case (iv) holds. Then $%
u\mid \frac{n}{t}$ by \cite{KL}, Section 4.2, and hence $1+ef\leq u\leq n/2$ by \cite%
{KL}, Proposition 5.2.15.(i). This is impossible, since $n\geq 4$ in both cases by Remark %
\ref{mag3}.\\
Finally, assume that case (v) occurs. Then $\left\vert
G_{x}\right\vert \leq 24\left\vert X_{x}\right\vert $ and hence $\left\vert
P\Omega _{8}^{+}(q)\right\vert \leq 24^{3}\left\vert X_{x}\right\vert ^{3}$,
as $G_{x}$ is a large subgroup of $G$ by Lemma \ref{PP}(2). Then either $X_{x}\cong
\Omega _{8}^{+}(2)$ and $p=3,5$ or $11$, or $X_{x}\cong 2^{3}\cdot
2^{6}\cdot PSL_{3}(2)$ and $p=3$. However, all the cases are ruled out, as
they violate $\left[ X:X_{x}\right] =k^{2}$. Thus, $k$ is even and hence $%
\lambda (k+1)\left\vert G_{x,B}\right\vert \equiv 2 \pmod 4$. Arguing as
above, we see that $\left[ G_{x}:X_{x}\right] \mid \lambda (k+1)\left\vert
G_{x,B}\right\vert $, where $\left[ G_{x}:X_{x}\right] =\frac{\left\vert 
\mathrm{Out(}X\mathrm{)}\right\vert }{c}$ for some $c \geq 1$ such that $4\nmid \frac{\left\vert 
\mathrm{Out(}X\mathrm{)}\right\vert }{c}$.

If $X\cong PSU_{n}(q^{1/2})$, with $n$ even, then $\frac{\left\vert \mathrm{%
Out(}X\mathrm{)}\right\vert }{c}\leq nf/4<1+(n-1)f$, as $\left\vert \mathrm{%
Out(}X\mathrm{)}\right\vert =(q^{1/2}+1,n)f$ and as $f$ is even. Then $\left[
G_{x}:H_{x}\right] =\left\vert \mathrm{Out(}X\mathrm{)}\right\vert
/c<1+(n-1)f$ and we obtain the assertion in this case.

If $X\cong P\Omega _{8}^{+}(q)$, then $\frac{\left\vert 
\mathrm{Out(}X\mathrm{)}\right\vert }{c}\leq 6f$, as $4\nmid \frac{\left\vert 
\mathrm{Out(}X\mathrm{)}\right\vert }{c}$, and hence $\left[ G_{x}:H_{x}%
\right] \leq 6f<1+6f$. This completes the proof.
\end{proof}

\begin{lemma}
\label{tamo}Let $(x,B)$ be any flag of $\mathcal{D}$. $\,$If $\left( \Phi
_{ef}^{\ast }(p),\lambda \left\vert G_{x,B}\right\vert \right) >1$, then $%
\lambda \left\vert G_{x,B}\right\vert \geq \left[ G_{x}:H_{x}\right] $.
\end{lemma}

\begin{proof}
Suppose the contrary. Then $\lambda \left\vert G_{x,B}\right\vert < \left[ G_{x}:X_{x}\right] \leq
\left\vert \mathrm{Out(}X\mathrm{)}\right\vert $, as $\left[ G_{x}:H_{x}%
\right] $ divides $\left\vert \mathrm{Out(}X\mathrm{)}\right\vert $. Also, $%
\lambda \left\vert G_{x,B}\right\vert =\rho \theta $, where $\rho =\left(
\Phi _{ef}^{\ast }(p),\lambda \left\vert G_{x,B}\right\vert \right) $ and $%
\theta \geq 1$. Since $\rho >1$, $\rho =u_{1}^{m_{1}}\cdots u_{\ell
}^{m_{\ell }}$ with $u_{i}$ prime for $i=1,...,\ell $, and $u_{i}\neq u_{j}$
for $i\neq j$. Furthermore, $u_{i}\equiv 1 \pmod{ef}$ by \cite{KL}, Proposition
5.2.15.(i), and hence $u_{i}\geq 1+ef$. Therefore, either $\rho \geq 3\left(
1+ef\right) $ (actually, $\rho >6\left( 1+6f\right) $ for $P\Omega_{8}^{+}(q)$ by Remark \ref{mag3}), or $\rho =1+ef$ is a prime. Consequently, $\left\vert G_{x,B}\right\vert \geq \mu (1+ef) \theta$, where $\mu \geq 3$, or $\mu =1$ and $1+ef$ is a prime.

Assume that $X$ isomorphic to one of the groups $PSL_{n}(q)$
or $PSU_{n}(q^{1/2})$. Then $\left\vert \mathrm{Out(}X\mathrm{)%
}\right\vert \leq 2nf$ by \cite{KL}, Table 5.1.A, and hence $\mu \left( 1+ef\right) \theta \leq \lambda \left\vert G_{x,B}\right\vert \leq 2nf$. Then $\mu = 1$, $\rho =1+ef$ is a prime and either $e=n$ and $\theta =1$, or $e=n-1$, $X\cong PSU_{n}(q^{1/2})$, with $n$ even, and $\theta \leq 2$.

If $\theta =1$, then $\lambda \left\vert G_{x,B}\right\vert $ is odd
and hence $\rho =\lambda \left\vert G_{x,B}\right\vert \leq \left[ G_{x}:H_{x}%
\right] <1+ef$ by Lemma \ref{Haide}(1). However, this contradicts $\rho = 1+ef$.

If $\theta =2$, then $X\cong PSU_{n}(q^{1/2})$, with $n$ even. It results $%
2 \rho = \lambda \left\vert G_{x,B}\right\vert \leq \left[ G_{x}:H_{x}\right] <1+ef$
by Lemma \ref{Haide}(2), and we again reach a contradiction.

Assume that $X$ isomorphic to one of the groups $PSp_{n}(q)$ or $P\Omega
_{n}(q)$. Then $\left\vert \mathrm{Out(}X\mathrm{)}\right\vert \leq 2f$ by
\cite{KL}, Table 5.1.A, and hence $\theta \left( 1+ef\right) \leq \lambda \left\vert
G_{x,B}\right\vert \leq 2f$, whereas $e=n-2\geq 2$ by Remark \ref{mag3}.

Assume that $X$ isomorphic to $P\Omega _{n}^{\varepsilon }(q)$, with $%
\varepsilon =\pm $ and $(n,\varepsilon )\neq (8,+)$. Then $\left\vert 
\mathrm{Out(}X\mathrm{)}\right\vert \leq 8f$ by \cite{KL}, Table 5.1.A, and hence $%
\mu \left( 1+ef\right) \theta = \lambda \left\vert G_{x,B}\right\vert \leq 8f$%
. Then $\mu=\theta =1$, $n=6$ and $\rho =1+ef$ is a prime, and hence $\lambda \left\vert
G_{x,B}\right\vert =1+ef $, whereas $\lambda \left\vert
G_{x,B}\right\vert <1+ef $ by Lemma \ref%
{Haide}.

Finally, assume that $X\cong P\Omega _{8}^{+}(q)$. Then $\lambda \left\vert
G_{x,B}\right\vert \leq \left\vert \mathrm{Out(}X\mathrm{)}\right\vert \leq
24f$ and hence $\lambda \left\vert G_{x,B}\right\vert =\theta \rho $, with $%
\theta \leq 3$ and $\rho =1+6f$ a prime. Moreover, either $\theta =1,3$ and $\lambda
\left\vert G_{x,B}\right\vert $ is odd, or $\theta =2$ and $\lambda
\left\vert G_{x,B}\right\vert \equiv 2\pmod4$. Then $\lambda \left\vert
G_{x,B}\right\vert \leq \left[ G_{x}:H_{x}\right] <\rho $ by Lemma \ref%
{Haide}(2)--(3), and we again reach a contradiction. This completes the proof.
\end{proof}

\bigskip

The following theorem, which relies both on numerical properties of $\Phi
_{ef}^{\ast }(p)$ and on Lemmas \ref{Haide} and \ref{daleko},
states that either $X_{x}$ is a large subgroup of $X$ or $\Phi _{ef}^{\ast
}(p)$ is a square, or $\Phi _{ef}^{\ast }(p)$ has a non-trivial factor in
common with $k+1$. As we will see, the theorem is a tool that, together with \cite{GPPS} and \cite{BP}, allows us to control the structure of $X_{x}$.
\bigskip

\begin{theorem}
\label{daleko}Let $(x,B)$ be any flag in $\mathcal{D}$. Then one of the
following holds:

\begin{enumerate}
\item $\left( \Phi _{ef}^{\ast }(p),\left\vert X_{x}\right\vert \right) >1$
and one of the following holds:

\begin{enumerate}
\item $X_{x}$ is a large subgroup of $X$ and $\left( \Phi _{ef}^{\ast
}(p),\left\vert X_{B}\right\vert \right) >1$;

\item $\left( \Phi _{ef}^{\ast }(p),k+1\right) >1$.
\end{enumerate}

\item $\left( \Phi _{ef}^{\ast }(p),\left\vert X_{x}\right\vert \right) =1$, 
$\Phi _{ef}^{\ast }(p)\mid k^{2}$ and $\left( \Phi _{ef}^{\ast
}(p),\left\vert X_{B}\right\vert \right) =\Phi _{ef}^{\ast }(p)^{1/2}$.
\end{enumerate}
\end{theorem}

\begin{proof}
Since $r=\lambda (k+1)$ and $\left\vert G_{x}\right\vert =r\left\vert
G_{x,B}\right\vert $, it follows that%
\begin{equation}
\left\vert X_{x}\right\vert =\frac{\lambda (k+1)\left\vert
G_{x,B}\right\vert }{\left[ G_{x}:X_{x}\right] } \textit{.}  \label{zdravo}
\end{equation}

Assume that $\left( \Phi _{ef}^{\ast }(p),\left\vert X_{x}\right\vert
\right) >1$. If $\left( \Phi _{ef}^{\ast }(p),k+1\right) =1$, then $\left(
\Phi _{ef}^{\ast }(p),\lambda \left\vert G_{x,B}\right\vert \right) >1$ and
hence $\lambda \left\vert G_{x,B}\right\vert \geq \left[ G_{x}:H_{x}\right] $
by Lemma \ref{tamo}. Thus $\left\vert X_{x}\right\vert \geq k+1$ by (\ref{zdravo}%
), and hence $X_{x}$ is a large subgroup of $X$, as $k^{2}=[X:X_{x}]$.

Since $\lambda \left\vert G_{x,B}\right\vert \mid \left\vert
G_{B}\right\vert $ and $\left( \Phi _{ef}^{\ast }(p),\lambda \left\vert
G_{x,B}\right\vert \right) >1$, it follows that $\left( \Phi _{ef}^{\ast
}(p),\left\vert G_{B}\right\vert \right) >1$. Moreover, as $\Phi _{ef}^{\ast
}(p)\equiv 1 \pmod{ef}$ by \cite{KL}, Proposition 5.2.15.(i), it follows
that $\left( \Phi _{ef}^{\ast }(p),ef\right) =1$ and hence $\left( \Phi
_{ef}^{\ast }(p),\left\vert \mathrm{Out(}X\mathrm{)}\right\vert \right) =1$
(see \cite{KL}, Table 5.1.A). Thus $\left( \Phi _{ef}^{\ast
}(p),\left\vert X_{B}\right\vert \right) >1$, as $\left[ G_{B}:X_{B}\right]
\mid \left\vert \mathrm{Out(}X\mathrm{)}\right\vert $, which is the
assertion (1a).

Assume that $\left( \Phi _{ef}^{\ast }(p),\left\vert X_{x}\right\vert
\right) =1$. Then $\Phi _{ef}^{\ast }(p)\mid k^{2}$, since $\left[ X:X_{x}%
\right] =k^{2}$, and hence $\Phi _{ef}^{\ast }(p)\mid \left\vert X\right\vert $. Arguing as above, we obtain $\left( \Phi _{ef}^{\ast }(p),\left\vert X_{B}\right\vert \right) =\Phi
_{ef}^{\ast }(p)^{1/2}$, which is the assertion (2).
\end{proof}

\bigskip

The following Lemma is the analogue of
Corollary 3.10 of \cite{ABD3}, but we do not make any additional assumption on $X_{x}$ because of Proposition \ref{greatsaxl}.

\bigskip

\begin{lemma}
\label{Sing}The following hold:

\begin{enumerate}
\item $\left\vert X\right\vert < \left\vert \mathrm{Out(}X\mathrm{)}%
\right\vert ^{2}\cdot \left\vert X_{x}\right\vert \cdot \left\vert
X_{x}\right\vert _{p^{\prime }}^{2}$

\item If $\left\vert \mathrm{Out(}X\mathrm{)}\right\vert \leq \left\vert
X_{x}\right\vert _{p}$, then $X_{x}$ is a large subgroup of $X$.
\end{enumerate}
\end{lemma}

\begin{proof}
Since $p \mid k$ by Proposition \ref{greatsaxl}, it follows that $(p,r/\lambda )=1$, as $r/\lambda =k+1$, and hence $%
r/\lambda \leq \left\vert G_{x}\right\vert _{p^{\prime }}\leq \left\vert 
\mathrm{Out(}X\mathrm{)}\right\vert _{p^{\prime }}\left\vert
X_{x}\right\vert _{p^{\prime }}$. Since $k^{2}<\left( r/\lambda \right) ^{2}$%
, it results $\left\vert X\right\vert < \left\vert X_{x}\right\vert
\left\vert \mathrm{Out(}X\mathrm{)}\right\vert _{p^{\prime }}^{2}\left\vert
X_{x}\right\vert _{p^{\prime }}^{2}$, and hence (1) holds. Moreover, if $%
\left\vert \mathrm{Out(}X\mathrm{)}\right\vert \leq \left\vert
X_{x}\right\vert _{p}$, then $\left\vert X\right\vert <
\left\vert X_{x}\right\vert ^{3}$ and hence $X_{x}$ is a large subgroup of $%
X $, which is (2).
\end{proof}

\bigskip

The following lemma and the subsequent corollary provide a lower bound for $\left\vert
X_{B}\right\vert _{p}$.

\bigskip

\begin{lemma}
\label{utakmica}$\left\vert X\right\vert _{p}\leq \left\vert \mathrm{Out(}X%
\mathrm{)}\right\vert _{p}^{3}\left\vert X_{B}\right\vert _{p}^{3}$.
\end{lemma}

\begin{proof}
Since $\left[ G:G_{B}\right]=\lambda k(k+1)$, it follows that $\left\vert
G\right\vert _{p}=\lambda _{p}k_{p}(k+1)_{p}\left\vert G_{B}\right\vert _{p}$ by Proposition \ref{greatsaxl}. Thus, $\left\vert G\right\vert _{p}=\lambda _{p}k_{p}^{2}\left\vert
G_{x,B}\right\vert _{p}$. \\If $\lambda _{p}\leq \left\vert G_{x,B}\right\vert
_{p}$, then $\left\vert G\right\vert _{p}\leq \left\vert G_{B}\right\vert
_{p}^{2}$ and hence $\left\vert X\right\vert _{p}\leq \left\vert \mathrm{Out(%
}X\mathrm{)}\right\vert _{p}^{2}\left\vert X_{B}\right\vert _{p}^{2}$, as $\left[
G_{B}:X_{B}\right]$ divides $\left\vert \mathrm{Out(}X\mathrm{)}\right\vert $, whereas if $\lambda _{p}>\left\vert G_{x,B}\right\vert _{p}$, then 
$\left\vert G\right\vert _{p}\leq \left\vert G_{B}\right\vert _{p}^{3}$ and
hence $\left\vert X\right\vert _{p}\leq \left\vert G\right\vert _{p}\leq
\left\vert G_{B}\right\vert _{p}^{3}\leq \left\vert \mathrm{Out(}X\mathrm{)}%
\right\vert _{p}^{3}\left\vert X_{B}\right\vert _{p}^{3}$.
\end{proof}

\bigskip
It should be stressed out that, the following corollary works under the assumptions of Remark \ref{mag3}. Furthermore, the lower bounds for $\left\vert X_{B}\right\vert_{p}$ are not necessarily integers.

\bigskip

\begin{corollary}
\label{kralj} The following hold:

\begin{enumerate}
\item[(i).] $\left\vert X_{B}\right\vert _{p}\geq q^{\frac{\left( n+2\right)
\left( n-3\right) }{6}}$ for $X \cong PSL_{n}(q)$.

\item[(ii).]  $\left\vert X_{B}\right\vert _{p}\geq q^{\frac{ n^{2}-12}{12}}$ \hspace{4mm} for $X \cong PSp_{n}(q)$

\item[(iii).] $\left\vert X_{B}\right\vert _{p}\geq q^{\frac{\left( n+2\right)
\left( n-3\right) }{12}}$ for $X\cong PSU_{n}(q^{1/2})$.

\item[(iv).] $\left\vert X_{B}\right\vert _{p}\geq q^{\frac{(n-1)^{2}-12}{12}}$ for 
$X\cong P\Omega _{n}(q)$.

\item[(v).] $\left\vert X_{B}\right\vert _{p}\geq q^{\frac{n(n-2)-12}{12}}$ for $%
X\cong P\Omega _{n}^{\varepsilon }(q)$.
\end{enumerate}
In each case $\left\vert X_{B}\right\vert _{p} \geq p$
\end{corollary}

\begin{proof}
Assume that $X$ is isomorphic to $PSL_{n}(q)$. Then $%
\left\vert X\right\vert _{p}=q^{\frac{n\left( n-1\right) }{2}}$ and $%
\left\vert \mathrm{Out(}X\mathrm{)}\right\vert _{p}=(2f)_{p}$. Since $%
q=p^{f}\geq 2f\geq (2f)_{p}$, it follows that $\left\vert X_{B}\right\vert
_{p}\geq q^{\frac{\left( n+2\right) \left( n-3\right) }{6}}$ by Lemma \ref%
{utakmica}.

Assume that $X\cong \PSp_{n}(q)$. Hence, $\left\vert
X\right\vert _{p}=q^{\frac{n^{2}}{4}}$ and $\left\vert \mathrm{Out(}X%
\mathrm{)}\right\vert _{p}=f_{p}$. Then $\left\vert X_{B}\right\vert
_{p}^{3}\geq q^{\frac{n^{2}}{4}-3}$ by Lemma \ref{utakmica} and hence $%
\left\vert X_{B}\right\vert _{p}\geq q^{\frac{n^{2}-12}{12}}$.

Assume that $X\cong PSU_{n}(q^{1/2})$. Then $\left\vert X\right\vert _{p}=q^{%
\frac{n(n-1)}{4}}$ and $\left\vert \mathrm{Out(}X\mathrm{)}\right\vert
_{p}=f_{p}$. Since $q^{1/2}\geq f$, arguing as above, we obtain $\left\vert
X_{B}\right\vert _{p}\geq q^{\frac{\left( n+2\right) \left( n-3\right) }{12}%
} $.

Assume that $X\cong \Omega _{n}(q)$, with $nq$ odd. Hence, $\left\vert
X\right\vert _{p}=q^{\frac{(n-1)^{2}}{4}}$ and $\left\vert \mathrm{Out(}X%
\mathrm{)}\right\vert _{p}=f_{p}$. Then $\left\vert X_{B}\right\vert
_{p}^{3}\geq q^{\frac{(n-1)^{2}}{4}-3}$ by Lemma \ref{utakmica} and hence $%
\left\vert X_{B}\right\vert _{p}\geq q^{\frac{(n-1)^{2}-12}{12}}$.

Assume that $X\cong P\Omega _{n}^{\varepsilon }(q)$. Then $\left\vert
X\right\vert _{p}=q^{\frac{n(n-2)}{4}}$ and $\left\vert \mathrm{Out(}X%
\mathrm{)}\right\vert _{p}=\left( 3f\right) _{p}$ or $\left( 2f\right) _{p}$
according as $(n,\varepsilon ,p)$ is or is not $(8,+,3)$ respectively. In
both cases we have $q\geq \left\vert \mathrm{Out(}X\mathrm{)}\right\vert
_{p} $, and hence $\left\vert X_{B}\right\vert _{p}\geq q^{\frac{n(n-2)-12}{12%
}}$, again by Lemma \ref{utakmica}.\\
In each case, $\left\vert X_{B}\right\vert _{p}>1$ and hence $\left\vert X_{B}\right\vert _{p} \geq p$. 
\end{proof}

\bigskip

In the sequel, we denote by $A^{\symbol{94}}$ the pre-image of any group $A$ in the corresponding linear group. 

\bigskip

Recall that $e=n,n-1$ or $n-2$. If $e\leq n/2$, then either $n=2$ and $e=1$, or $n=4$ and $e=2$. However, such cases cannot occur by Remark \ref{mag3}. Thus, $n/2<e\leq n$ for each $X$, and hence, if $H$ is any subgroup of $X$ such that $(\Phi_{ef}^{\ast }(p),\left\vert H\right\vert)>1$, then  $H^{\symbol{94}}$ is classified in the Main Theorem of \cite{GPPS}, and therefore $H$ is known.

\bigskip

Now, we are in position to tackle cases (1) and (2) of Theorem \ref{daleko}. They are investigated in separate sections.

\bigskip

\section{Reduction to the case $\left( \Phi _{ef}^{\ast }(p),\left\vert
X_{x}\right\vert \right) >1$}

The aim of this section is to rule out case (2) of Theorem \ref{daleko}. The
proof strategy for doing so is as follows. By using \cite{GPPS}, \cite{KL}
and \cite{GLNP} we determine the structure of $X_{B}^{\symbol{94}}$. From
this we derive an upper bound for $\left\vert X_{B}\right\vert _{p}$. Then,
we show that such a bound is in contrast with the lower bound for $%
\left\vert X_{B}\right\vert _{p}$ determined in Corollary \ref{kralj}, and
hence no cases arise.

\bigskip

Assume $\left( \Phi _{ef}^{\ast }(p),\left\vert X_{x}\right\vert \right) =1$%
, $\Phi _{ef}^{\ast }(p)\mid k^{2}$ and $\left( \Phi _{ef}^{\ast
}(p),\left\vert X_{B}\right\vert \right) =\Phi _{ef}^{\ast }(p)^{1/2}$. Then $\Phi _{ef}^{\ast }(p)^{1/2}\mid \left( \left\vert X_{B}\right\vert ,%
\left[ X:X_{B}\right] \right) $. Also, $p\mid \left\vert X_{B}\right\vert $
by Corollary \ref{kralj} for $n\geq 4$. Then $\Phi _{ef}^{\ast
}(p)^{1/2}\mid \left( \left\vert X_{B}^{\symbol{94}}\right\vert ,\left[
SL_{n}(q):X_{B}^{\symbol{94}}\right] \right) $, with $\Phi _{ef}^{\ast
}(p)>1 $ a square. Moreover, $p\mid \left\vert X_{B}^{\symbol{94}}\right\vert 
$, as $\left\vert X_{B}^{\symbol{94}}\right\vert
_{p}=\left\vert X_{B}\right\vert _{p}$. Then $X_{B}^{\symbol{94}}$ is isomorphic to one of the subgroups classified in \cite{GPPS} and satisfying the previous additional constraints.

\bigskip

Let $Z$ denote the center of $GL_{n}(q)$. A subgroup $Y$ of $GL_{n}(q)$ is 
\emph{nearly simple} if $S\leq Y/(Y\cap Z)\leq Aut(S)$, for some non-abelian
simple group $S$, and if $N$ is the preimage of $S$ in $Y$, then $N$ is
absolutely irreducible on $V_{n}(q)$ and $N$ is not a classical group
defined over a subfield of $GF(q)$ (in its natural representation).

\bigskip

\begin{lemma}
\label{blisanac1}Let $\Phi _{ef}^{\ast }(p)>1$ be a square. If $Y$ is a
nearly simple subgroup of $SL_{n}(q)$ such that $p\mid \left\vert
Y\right\vert $ and $\Phi _{ef}^{\ast }(p)^{1/2}\mid \left( \left\vert
Y\right\vert ,\left[ SL_{n}(q):Y\right] \right) $, then one of the following
holds:

\begin{enumerate}
\item $e=n=4$, $q=7$, $\Phi _{4}^{\ast }(7)=5^{2}$ and $Z_{2}.A_{7} \unlhd Y \leq Z_{2}.S_{7} $.

\item $e=n=5$, $q=3$, $\Phi _{5}^{\ast }(3)=11^{2}$ and either $Y\cong
PSL_{2}(11)$ or $Y\cong M_{11}$.
\end{enumerate}
\end{lemma}

\begin{proof}
Assume that $S\cong A_{\ell }$, $\ell \geq 5$. Then $Y$ is one of the groups
listed in Example 2.6.(a)--(c) of \cite{GPPS}. If $Y$ is one of the groups
in Example 2.6.(b)--(c), then $\ell \leq 10$ and $\Phi _{ef}^{\ast
}(p)^{1/2}$ is one of the primes $5$ or $7$, as $n\geq 3$ (see Tables 2--3
in \cite{GPPS}). Then $\Phi _{ef}^{\ast }(p)$ is either $5^{2}$ or $7^{2}$.
Actually, only the former is admissible and occurs for $(ef,p)=(4,7)$ by 
\cite{GLNP}, Theorem 3. Then $Z_{2}.A_{7}\trianglelefteq Y$ by Tables 2--3
in \cite{GPPS}, since $p\mid \left\vert Y\right\vert $, and hence
$Z_{2}.A_{7} \unlhd Y \leq Z_{2}.S_{7} $ by \cite{AtMod}.

Suppose that $Y$ is one of the groups in Example 2.6.(a). Then $V_{n}(q)$, $%
q=p$, is the fully deleted permutation module for $S$ and $A_{\ell }\leq
Y\leq S_{\ell }\times Z_{p-1}$, where $\ell =n+\varepsilon $, $\varepsilon =1
$ or $2$ according as $p$ does not divide or does divide $n$
respectively. Let $u$ be any prime divisor of $\Phi _{ef}^{\ast }(p)$. Then $%
u^{i}\mid \left\vert Y\right\vert $, and hence $u^{i}\mid \left\vert
S\right\vert $, for some $i\geq 1$. Then $i< \frac{\ell -1}{u}$ by 
\cite{DM}, Exercise 2.6.8. On the other hand, $u=1+te$, for some $t\geq
1$, by \cite{KL}, Proposition 5.2.15.(i). Thus $1+i(1+te)<n+\varepsilon$, where $e=n,n-1$ or $n-2$. Then $e \neq n$, $i=1$, and either $t=1$, or $\varepsilon=2$, $e=1$, $n=p=3$ and $t=2$. The latter implies $u=p=3$, and hence a contradiction, whereas the former yields $(e,q,u)=(4,7,5)$ by \cite{GLNP}, Theorem 3, since $\Phi _{e}^{\ast }(q)=(e+1)^{2}$. Then $(\ell,n,\varepsilon) =(7,6,1)$, since $\ell=n+\varepsilon$, $e=n-1,n-2$ and $7 \mid \left\vert Y\right\vert $. However, $A_{7}$ does not afford an irreducible representation on $V_{4}(7)$ by \cite{AtMod}, and this case is excluded.

Assume that $S$ is sporadic. Then $S$ is one of the groups listed in Table 5
of \cite{GPPS}. Note that, the order of each of these groups is divisible by
exactly one primitive prime divisor of $p^{e}-1$, with $e=n,n-1$ or $n-2$. Thus $\Phi _{ef}^{\ast }(p)$ is the square of a prime, and hence $(ef,p,u)=(4,7,5)$ or $(5,3,11)$ by \cite{GLNP},
Theorem 3. Only the latter yields a group that is listed in Table 5 and this is $Y^{\prime }\cong M_{11}$. Then $Y\cong M_{11}$, since $\mathrm{Out(}M_{11}\mathrm{)}=1$ by \cite{AtMod}.

Assume that $S$ is a Lie type simple group in characteristic different from $%
p$. Then $Y$ can be determined from Tables 7--8 of \cite{GPPS}. As in the
sporadic case, the order of the groups in these two tables is divisible by
exactly one primitive prime divisor of $p^{e}-1$, with $e=n,n-1$ or $n-2$. Then $\Phi _{ef}^{\ast }(p)$ is the square of a prime and hence $(ef,p,u)=(4,7,5)$ or $(5,3,11)$ as above. If $S$ is listed in Table 7, then 
$S\cong Sp_{4}(3)$, but no groups $Y$ occur with order divisible by $7$. Hence, assume $S$ is one of the groups listed in Table 8. If $u=5$, then $e=4$ and it is easy to see that $S$ is
isomorphic to one of the groups $PSp_{4}(3)$, or $PSL_{2}(9)$. However, there are no corresponding $Y$ with order divisible by $7$. Thus, $u=11$, $u=2e+1$
and hence $s=11$ and $S\cong PSL_{2}(11)$. Then $Y\cong PSL_{2}(11)$
by \cite{AtMod}. 

Assume that $S$ is a Lie type simple group in characteristic $p$. Then $Y$ is
one of the groups in Example 2.8, and hence in Table 6, of \cite{GPPS}. Let 
$GF(q_{0})$ be a subfield of $GF(q)$ for which $Y\leq GL_{n}(q_{0})\circ
Z_{q-1}$, and $Y\cap GL_{n}(q_{0})$ cannot be realized modulo scalar over a
proper subfield of $GF(q_{0})$. Let $u$ be a prime divisor of\ $\Phi
_{ef}^{\ast }(p)$, then $u$ is a primitive prime divisor of $p^{ef}-1$ and
this forces $q_{0}=q$ by \cite{KL}, Proposition 5.2.15.(ii).\\ 
Assume that $n=8$, $e=6$ and $%
Y^{(\infty )}\cong SL_{2}(q^{3})$. Then $\left[ SL_{8}(q):Y\right] $ divides%
\begin{equation*}
\left[ SL_{8}(q):Y^{(\infty )}\right] =q^{25}\prod_{i=2,i\neq
6}^{8}(q^{i}-1) 
\end{equation*}
and hence $\Phi _{6f}^{\ast }(p)^{1/2}$ does not divide $\left[ SL_{8}(q):Y%
\right] $, a contradiction. A similar reasoning rules out the remaining
cases in Table 6.
\end{proof}

\begin{lemma}
\label{blisanac2}Let $\Phi _{ef}^{\ast }(p)>1$ be a square. If $\Phi
_{ef}^{\ast }(p)^{1/2}\mid \left( \left\vert Y\right\vert ,\left[ SL_{n}(q):Y%
\right] \right) $ and $p\mid \left\vert Y\right\vert $, then $Y$ is not isomorphic to any
of the groups listed in Examples 2.1, 2.3 or 2.5 of \cite{GPPS}.
\end{lemma}

\begin{proof}
Assume that $Y$ is isomorphic to one of the groups defined as in Example 2.1 of \cite{GPPS}. Then $\Phi
_{ef}^{\ast }(p)$ divides the order of $Y$, but it is coprime to $\left[ SL_{n}(q):Y%
\right] $, and this contradicts our assumptions.

Assume that $Y$ is isomorphic to one of the groups defined as in Example 2.3 of \cite{GPPS}. Then $%
Y\leq GL_{1}(q)\wr S_{n}$ and $u=e+1\leq n$ for each prime divisor $u$ of $\Phi
_{ef}^{\ast }(p)^{1/2}$. Then $e=n-1$ and $u=n$ or $e=n-2$ and $u=n-1$. Thus 
$\Phi _{ef}^{\ast }(p)=u^{2}=(e+1)^{2}$ and hence $(ef,p,u)=(4,7,5)$ by \cite%
{GLNP}, Theorem 3. Therefore, $f=1$, $u=n=5$ and hence $Y\leq GL_{1}(7)\wr S_{5}$%
. However, this group cannot occur since $7 \nmid \left\vert
Y\right\vert $.

Assume that $Y$ is isomorphic to one of the groups defined as in Example 2.5 of \cite{GPPS}. Let $%
u $ be a prime divisor of $\Phi _{ef}^{\ast }(p)$. Then $e=n=2^{m}$ and
either $u=n+1$ is a Fermat prime or $u=n-1$ is a Mersenne prime. Moreover, $%
\Phi _{ef}^{\ast }(p)^{1/2}=u$ by Table 1 of \cite{GPPS}. Hence $\Phi
_{ef}^{\ast }(p)=u^{2}$ and, as above, $(e,q,u)=(4,7,5)$ since $n$ is even. Thus $Y\leq \left(
(D_{8}\circ Q_{8})\cdot S_{5}\right) \circ Z_{6}$, which cannot occur since $%
7$ does not divide $\left\vert Y\right\vert $. This completes the proof
\end{proof}

\begin{proposition}
\label{hdu}Let $\Phi _{ef}^{\ast }(p)>1$ be a square. If $\Phi _{ef}^{\ast
}(p)^{1/2}\mid \left( \left\vert Y\right\vert ,\left[ SL_{n}(q):Y\right]
\right) $ and $p\mid \left\vert Y\right\vert $, then one of the following
holds:

\begin{enumerate}
\item $e=n$ and one of the following holds:

\begin{enumerate}
\item $Y\leq GL_{1}(q^{n})\cdot Z_{n}$ and $p\mid n$.

\item $e=n=4$, $q=7$,\ $\Phi _{4}^{\ast }(7)=5^{2}$ and $Z_{2}.A_{7} \unlhd Y \leq Z_{2}.S_{7} $.

\item $e=n=5$, $q=3$, $\Phi _{5}^{\ast }(3)=11^{2}$ and either $Y\cong
PSL_{2}(11)$ or $Y\cong M_{11}$.
\end{enumerate}

\item $e=n-1$ and $Y$ lies in a parabolic subgroup of type $P_{1}$ of $%
SL_{n}(q)$.

\item $e=n-2$ and $Y$ lies in a parabolic subgroup of type $P_{2}$ or $%
P_{1,n-1}$ of $SL_{n}(q)$.
\end{enumerate}
\end{proposition}

\begin{proof}
It follows from Lemmas \ref{blisanac1} and \ref%
{blisanac2} that either assertions (1.b) or (1.c) hold, or $Y$ is isomorphic to one of the
groups listed in Examples 2.2 and 2.4 of \cite{GPPS}.

Assume that $e=n$.\ If $Y$ is isomorphic to one of the groups listed as in Example 2.2 of \cite%
{GPPS}, then $Y$ is the stabilizer of a subspace or a quotient of a subspace $V_{n}(q)$ of dimension $m\geq e$, and hence this case cannot occur for $e=n$.
Then $Y$ is listed in Example 2.4 of \cite{GPPS} and hence $Y\leq
GL_{n/c}(q^{c})\cdot Z_{c}$ for some divisor $c$ of $n$ with $c>1$. Clearly, 
$\Phi _{nf}^{\ast }(p)$ divides the order of $GL_{n/c}(q^{c})\cdot Z_{c}$.
Then $\Phi _{nf}^{\ast }(p)$ divides the order of $SL_{n/c}(q^{c})$, as $%
\Phi _{nf}^{\ast }(p)\equiv 1 \pmod{nf}$ by \cite{KL}, Proposition
5.2.15.(i). Also $\Phi _{nf}^{\ast }(p)$ is coprime to $\left[
SL_{n}(q):SL_{n/c}(q^{c})\right] $, and hence $\Phi _{nf}^{\ast }(p)^{1/2}$
divides the order of $Y\cap SL_{n/c}(q^{c})$, since $Y\leq GL_{n/c}(q^{c})\cdot Z_{c}$.

Set $f^{\prime }=cf$. Clearly, $\Phi _{nf}^{\ast }(p)=\Phi _{\frac{n}{c}%
f^{\prime }}^{\ast }(p)$ and $\Phi _{\frac{n}{c}f^{\prime }}^{\ast
}(p)^{1/2} $ divides $\left\vert Y\cap SL_{n/c}(q^{c})\right\vert $ and $%
\left[ SL_{n/c}(q^{c}):Y\cap SL_{n/c}(q^{c})\right] $. Hence, we may apply
Lemmas \ref{blisanac1} and \ref{blisanac2} with $Y\cap SL_{n/c}(q^{c})$ and $%
SL_{n/c}(q^{c})$ in the role $Y$ and $SL_{n}(q)$ respectively, thus obtaining $Y\cap SL_{n/c}(q^{c})\leq GL_{n/c^{\prime }}(q^{c^{\prime
}})\cdot Z_{c^{\prime }}$ with $c\mid c^{\prime }$. Indeed, no cases arise from Lemma \ref{blisanac1}, since $c>1$. We may iterate the previous argument and eventually obtain that $Y\leq GL_{1}(q^{n})\cdot
Z_{n}$ and $p\mid n$, as $p\mid \left\vert Y\right\vert $, which is (1.a).

Assume that $e=n-1$. If $Y$ is a group listed in Example 2.2 of \cite{GPPS},
then $Y$ is the stabilizer of a subspace or a quotient of a subspace $%
V_{n}(q)$ of dimension $n-1$ and we obtain (2).

Assume that $Y$ is listed in Example 2.4 of \cite{GPPS}. Let $u$ is prime
divisor of $\Phi _{(n-1)f}^{\ast }(p)$, then $u=n$ and $Y\leq
GL_{1}(q^{n-1})\cdot Z_{n-1}$ by Example 2.4(a) of \cite{GPPS}. Since $%
Y $ contains a normal Sylow $u$-subgroup $U$ which stabilizes a unique
hyperplane of $V_{n}(q)$, we obtain (2).

Assume that $e=n-2$. If $Y$ is one of the groups listed in Example 2.2 of \cite%
{GPPS}, then $Y$ is the stabilizer of a subspace or a quotient of a subspace 
$Q$ of $V_{n}(q)$ of dimension $m=n-1$ or $n-2$. In the latter case $Y$ lies in a subgroup of type $P_{2}$ of $SL_{n}(q)$, and hence (3) holds.

Assume that $m=n-1$. Suppose that $\dim Q=1$. Then $Y\leq \lbrack
q^{n-1}]:GL_{n-1}(q)$, with $[q^{n-1}]$ fixing $Q$ pointwise, by \cite{KL}, Proposition 4.1.17.(II). Set $%
W=Y/(Y\cap \lbrack q^{n-1}])$ and $n^{\prime }=n-1$. Hence $e=n-2=n^{\prime
}-1$ and $W$ acts on $V_{n}(q)/Q$ inducing a subgroup of $GL_{n^{\prime
}}(q) $. Moreover, arguing as in the $e=n-1$ case, we see that $W$ is the
stabilizer of a subspace or a quotient of a subspace $V_{n}(q)/Q$ of
dimension $n^{\prime}-1$. Hence $Y$ lies in a parabolic subgroup of type $P_{2}$ or $%
P_{1,n-1}$ of $SL_{n}(q)$, which is (3). The same conclusions hold for $\dim
Q=n-1$.

Finally, assume that $Y$ is one of the groups in Example 2.4 of \cite{GPPS}. Then $n$
is even and $Y\leq GL_{n/2}(q^{2})\cdot Z_{2}$. If $f^{\prime}=2f$, then $\Phi _{\left(
n/2-1\right) f^{\prime }}^{\ast }(p)^{1/2}\mid \left( \left\vert
Y\right\vert ,\left[ SL_{n/2}(q^{2}):Y\right] \right) $ and $p\mid
\left\vert Y\right\vert $. Arguing as in the $e=n-1$ case, we get that $Y$
fixes subspace or a quotient of a subspace $V_{n/2}(q^{2})$ of dimension $%
n/2-1$. Then $Y$ is lies in a parabolic subgroup of type $P_{2}$ of $%
SL_{n}(q)$, and we obtain (3).
\end{proof}

\bigskip

\begin{corollary}
\label{rem}$X_{B}^{\symbol{94}}$ is isomorphic to one of the groups listed in Proposition %
\ref{hdu}.
\end{corollary}

\begin{proof}
Since $\Phi _{ef}^{\ast }(p)^{1/2}\mid \left( \left\vert X_{B}^{\symbol{94}%
}\right\vert ,\left[ SL_{n}(q):X_{B}^{\symbol{94}}\right] \right) $ and $%
q\mid \left\vert X_{B}^{\symbol{94}}\right\vert $, the assertion follows.
\end{proof}

\bigskip

Now, we analyze the cases $e=n,n-1$ or $n-2$ separately.

\bigskip

\begin{proposition}
\label{Trougao}If $\left( \Phi _{ef}^{\ast }(p),\left\vert X_{x}\right\vert
\right) =1$, then $e\neq n$.
\end{proposition}

\begin{proof}
Assume that $\left( \Phi _{nf}^{\ast }(p),\left\vert X_{x}\right\vert
\right) =1$. If $n=3$, then $X$ is isomorphic to one of the groups $%
PSL_{3}(q)$ or $PSU_{3}(q^{1/2})$. Then $\left\vert X\right\vert
_{p}=p^{3f}$, $\left\vert \mathrm{Out(}X\mathrm{)}\right\vert _{p} \leq (2f)_{p}$
and $\left\vert X_{B}\right\vert _{p}\leq (3,p)$. Then $p^{3f}\leq
(2f)_{p}^{3}(3,p)^{3}$ by Lemma \ref{utakmica}, and hence $p^{f}\leq
(2f)_{p}(3,p)$. If $p\neq 3$, then $p^{f}\leq 2f$ and hence $p=2$ and $f=1,2$%
. Then $X$ is isomorphic either $PSL_{3}(2)$ or to $PSL_{3}(4)$. However,
these cases cannot occur by Lemmas \ref{Due nonVale} and \ref{cyclont},
respectively. Thus $p=3$, $3^{f}\leq 3f$ and hence $f=1$. Then $X\cong
PSL_{3}(3)$, and we get a contradiction since $\Phi _{3}^{\ast }(3)=13$ is
not a square. Therefore, $n\geq 4$.

Assume that $X_{B}^{\symbol{94}}$ is nearly simple. If $n=5$ and $q=3$, then 
$X\cong PSL_{5}(3)$ and either $X_{B}\cong PSL_{2}(11)$ or $X_{B}\cong
M_{11} $. However, these groups cannot occur since $\left\vert
X_{B}\right\vert _{3} = 3^{2}$, whereas $\left\vert
X_{B}\right\vert _{3} \geq 3^{3}$ by Corollary \ref{kralj}(1). Then $n=4$, $%
q=7$, $X$ is isomorphic to one of the groups $PSL_{4}(7)$ or $PSp_{4}(7)$
(see Remark \ref{mag3}), and $X_{B}\cong A_{7}$. \\
If $X \cong PSp_{4}(7)$, then $b=[X:X_{B}]=54880$, since $A_{7}$ is maximal in $PSp_{4}(7)$, and this one has a unique conjugacy class of subgroups isomorphic to $A_{7}$, by \cite{BHRD}, Tables 8.13--8.13. Then $\lambda (k+1)k=54880$, with $2 \leq \lambda \leq k \leq 165$ and $\lambda \mid k$, since $b= \lambda (k+1)k$ and $\lambda \mid k$. However, the previous equation does not have solutions, and hence this case is ruled out.\\
If $X \cong PSL_{4}(7)$, then $[X:X_{B}]=919679040$. Hence $b=919679040 \cdot t$, for some $t$ divisor of $4$, since $b=\frac{[G:X]}{[G_B:X_B]}[X:X_B]$ and $\frac{[G:X]}{[G_B:X_B]}$ divides $\left\vert \mathrm{Out(}X\mathrm{)}\right\vert$, and since $\left\vert \mathrm{Out(}X\mathrm{)}\right\vert = 4$. Then $\lambda (k+1)k=919679040 \cdot t$, with $2 \leq \lambda \leq k \leq 21443 \cdot t$ and $\lambda \mid k$, since $b= \lambda (k+1)k$ and $\lambda \mid k$. However, the previous equation does not have solutions, and hence this case is excluded.\\
Assume that $X_{B}^{\symbol{94}}\leq GL_{1}(q^{n})\cdot Z_{n}$ and $p\mid n$%
. Then $\left\vert X_{B}\right\vert _{p}=\left\vert X_{B}^{\symbol{94}%
}\right\vert _{p}\leq n_{p}\leq n$ by Proposition \ref{hdu}(1). On the other
hand, by Corollary \ref{kralj}(i)--(iii), either $\left\vert X_{B}\right\vert _{p}\geq q^{%
\frac{n^{2}}{12 }}$ for $X\cong PSp_{n}(q)$, or $\left\vert X_{B}\right\vert _{p}\geq q^{%
\frac{\left( n+2\right) \left( n-3\right) }{6\theta }}$ with $\theta =1$ for $X\cong PSL_{n}(q)$ and $\theta =2$ for $X\cong PSU_{n}(q^{1/2})$ and $n$ odd. In the former case, $q^{%
\frac{n^{2}}{12 }} \leq n$ and hence $(q,n)=(2,4)$ but this case was already ruled out (see the remark before Lemma \ref{lambada}). Thus $q^{\frac{\left( n+2\right) \left( n-3\right) }{6\theta }}\leq n$ and hence
either $\theta =1$ and $(n,q)=(4,2),(4,3),(4,4),(4,5),(5,2)$, or $%
\theta =2$ and $(q,n)=(4,5),(9,5),(16,5)$. However, $%
\Phi _{nf}^{\ast }(p)$ is a square for none of these pairs, and hence they are excluded.
\end{proof}

\begin{proposition}
\label{demonstracija}If $\left( \Phi _{ef}^{\ast }(p),\left\vert X_{x}\right\vert
\right) =1$, then $e\neq n-1$.
\end{proposition}

\begin{proof}
Assume that $\left( \Phi _{(n-1)f}^{\ast }(p),\left\vert X_{x}\right\vert
\right) =1$. Then $X$ is isomorphic either to\ $\Omega _{n}(q)$ with $nq$ odd, or to $PSU_{n}(q^{1/2})$ with $n$ even. Suppose that
the latter occurs. Then $X_{B}^{\symbol{94}}$ fixes a non-degenerate $1$%
-dimensional subspace of $V_{n}(q)$ by Proposition \ref{hdu}(2) and by \cite{KL}, Propositions
4.1.4 and 4.1.20, since its order is divisible by $\Phi
_{(n-1)f}^{\ast }(p)^{1/2}$. Hence $X_{B}^{\symbol{94}}\leq GU_{n-1}(q^{1/2})$. Then $%
X_{B}^{\symbol{94}}\cap SU_{n-1}(q^{1/2})$ is a subgroup of $SL_{n-1}(q)$
such that $\Phi _{n-1f}^{\ast }(p)$ divides the order of $X_{B}^{\symbol{94}%
}\cap SU_{n-1}(q^{1/2})$ and the index of this one in $SL_{n-1}(q)$. Also $p$
divides the order of $X_{B}^{\symbol{94}}\cap SU_{n-1}(q^{1/2})$. Then we
may apply Proposition \ref{hdu}, with $X_{B}^{\symbol{94}}\cap SU_{n-1}(q^{1/2})
$ in the role of $Y$ and $SL_{n-1}(q)$ in the role of $SL_{n}(q)$, and we obtain that $X_{B}^{\symbol{94}}\cap SU_{n-1}(q^{1/2})$ is contained in $%
GL_{1}(q^{n-1})\cdot Z_{n-1}$, with $p \mid n-1$, since $q$ is a square. Then $\left\vert X_{B}^{\symbol{94}%
}\right\vert _{p}\leq (n-1)_{p}\leq n-1$, as $\left\vert X_{B}^{\symbol{94}%
}\cap SU_{n-1}(q^{1/2})\right\vert _{p}=\left\vert X_{B}^{\symbol{94}%
}\right\vert _{p}$. On the other hand, $\left\vert X_{B}\right\vert _{p}\geq
q^{\frac{\left( n+2\right) \left( n-3\right) }{12}}$ by Corollary \ref{kralj}(iii)%
. Therefore $q^{\frac{\left( n+2\right) \left( n-3\right) }{12}}\leq n-1$
and hence $n=4$ and $q=4$ or $9$, but in none of these cases $\Phi
_{2(n-1)}^{\ast }(p)$ is a square.

Assume that $X\cong $\ $\Omega _{n}(q)$ with $nq$ odd. Then $X_{B}^{\symbol{%
94}}$ preserves a non-degenerate $1$-dimensional subspace of $V_{n}(q)$ by
Proposition \ref{hdu} and by \cite{KL}, Propositions 4.1.6 and 4.1.14, since $\Phi
_{(n-1)f}^{\ast }(p)^{1/2}$ divides $\left( \left\vert X_{B}^{\symbol{94}%
}\right\vert ,\left[ SL_{n}(q):X_{B}^{\symbol{94}}\right] \right)$ and since $p$ divides $\left\vert X_{B}^{\symbol{94}}\right\vert$ by Corollary \ref{kralj}. Thus $X_{B}^{\symbol{94}}\leq \Omega _{n-1}^{-}(q).Z_{2}<SL_{n-1}(q)$ and hence $%
X_{B}^{\symbol{94}}$ is a subgroup of $SL_{n-1}(q)$ such that $\Phi
_{(n-1)f}^{\ast }(p)^{1/2}$ divides $\left( \left\vert X_{B}^{\symbol{94}%
}\right\vert ,\left[ SL_{n-1}(q):X_{B}^{\symbol{94}}\right] \right) $. Also $%
p\mid \left\vert X_{B}^{\symbol{94}}\right\vert $.
Then we may apply Proposition \ref{hdu}, with $X_{B}^{\symbol{94}}$ in the
role of $Y$ and $SL_{n-1}(q)$ in that of $SL_{n}(q)$, thus obtaining either $q=7$, $n-1=4$ and $Z_{2}.A_{7} \unlhd X_{B}^{\symbol{94}} \leq Z_{2}.A_{7}$, or $X_{B}^{\symbol{94}}\leq GL_{1}(q^{n-1})\cdot Z_{n-1}$ and $%
p\mid n-1$. The former case cannot occur, since $X \cong \Omega _{5}(7) \cong \PSp_{4}(7)$ contains $A_{7}$ but not its covering group $Z_{2}.A_{7}$. Thus $X_{B}^{\symbol{94}}\leq
GL_{1}(q^{n-1})\cdot Z_{n-1}$ and $p\mid n-1$, and hence $\left\vert
X_{B}\right\vert _{p}\leq (n-1)_{p}\leq n-1$. On the other hand, $q^{\frac{%
(n-1)^{2}-12}{12}}\leq \left\vert X_{B}\right\vert _{p}$ by Corollary \ref%
{kralj}(3). Then $q^{\frac{(n-1)^{2}-12}{12}}\leq n-1$ and hence $n=5$ and $%
q<4^{3}$, since $n\geq 5$ by Remark, \ref{mag3}. Then $q=7$, since $\Phi
_{4f}^{\ast }(p)$ must be a square, but $7$ does not divide $n-1=4$. So, this case cannot occur and the proof is completed.
\end{proof}

\begin{theorem}
\label{gotovo}$\left( \Phi _{ef}^{\ast }(p),\left\vert X_{x}\right\vert
\right) >1$.
\end{theorem}

\begin{proof}
Suppose the contrary. Then $e =n-2$ by Lemmas \ref{Trougao} and \ref%
{demonstracija}, and hence $X\cong P\Omega _{n}^{+}(q)$. By Proposition \ref{hdu}, $%
X_{B}^{\symbol{94}}$ lies in a parabolic subgroup of type $P_{2}$ or $%
P_{1,n-1}$ of $SL_{n}(q)$. If the former occurs, then $X_{B}^{\symbol{94}}$
preserves a non-degenerate $2$-dimensional subspace of $V_{n}(q)$ of type $-$ by \cite%
{KL}, Propositions 4.1.6 and 4.1.20, since $n\geq 6$ by Remark \ref{mag3}.\\
If $X_{B}^{\symbol{94}}$ lies in a parabolic subgroup of
type $P_{1,n-1}$ of $SL_{n}(q)$, then $X_{B}^{\symbol{94}}$ preserves
two subspaces $U$, $W$ of $V_{n}(q)$ such that $U<W$ and $\dim U=1$ and $%
\dim W=n-1$. Also, $U$ is non-degenerate by \cite{KL}, Proposition 4.1.20,
since $\Phi _{(n-2)f}^{\ast }(q)^{1/2}$ divides the order of $X_{B}^{\symbol{%
94}}$. Thus $W^{\perp }\neq U$, since $W^{\perp }<U^{\perp }$, and hence $X_{B}^{%
\symbol{94}}$ preserves $W^{\perp }$ with $\dim W^{\perp }=1$. Then $%
W^{\perp }\oplus U$ is a non-degenerate $2$-dimensional subspace of $V_{n}(q)
$ of type $-$ preserved by $X_{B}^{\symbol{94}}$. Therefore, $X_{B}^{\symbol{94}}$
preserves a non-degenerate $2$-dimensional subspace of $V_{n}(q)$ of type $-$ in each
case.

The group $X_{B}^{\symbol{94}}$ is contained in a maximal subgroup $M$ of $%
X^{\symbol{94}}$, where $M$ is either $\left( Z_{\frac{q+1}{(q-1,2)}}\times
\Omega _{n-2}^{-}(q)\right) .Z_{2}$ or $\left( Z_{\frac{q+1}{(q-1,2)}}\times
\Omega _{n-2}^{-}(q)\right) .[4]$ according to whether $q$ is
even or odd respectively, by \cite{KL}, Propositions 4.1.6. Thus 
\begin{equation}\label{melanpo}
\left\vert X_{B}\right\vert _{p}=\left\vert X_{B}^{\symbol{94}}\right\vert
_{p}=(2,q)\left\vert X_{B}^{\symbol{94}}\cap \Omega _{n-2}^{-}(q)\right\vert
_{p}. 
\end{equation}

Note that, $\Phi _{(n-2)f}^{\ast }(q)^{1/2}$ divides the order of $X_{B}^{%
\symbol{94}}\cap \Omega _{n-2}^{-}(q)$ and the index of it in $SL_{n-2}(q)$.
Also $p\mid \left\vert X_{B}^{\symbol{94}}\right\vert $ by Corollary \ref%
{kralj}, and hence $p$ divides the order of $X_{B}\cap \Omega _{n-2}^{-}(q)$ by (\ref{melanpo}), since $X \ncong P\Omega_{6}(2)^{+}$ by Lemma \ref{cyclont}.
Then we may apply Proposition \ref{hdu}, with $X_{B}\cap \Omega _{n-2}^{-}(q)$
in the role of $Y$ and $SL_{n-2}(q)$ in the role of $SL_{n}(q)$, thus obtaining either $q=7$, $n-2=4$ and $X_{B}^{\symbol{94}}\cap \Omega _{n-2}^{-}(q) \cong Z_{2}.A_{7}$, or $X_{B}^{\symbol{94}}\cap \Omega
_{n-2}^{-}(q)\leq GL_{1}(q^{n-2})\cdot Z_{n-2}$ with $p\mid n-2$. Since $X\cong P\Omega _{6}^{+}(7) \cong PSL_{4}(7)$, we may apply the same argument used in the proof of Proposition \ref{demonstracija} to rule out this case. Thus $X_{B}^{%
\symbol{94}}\cap \Omega _{n-2}^{-}(q)\leq GL_{1}(q^{n-2})\cdot Z_{n-2}$ and
hence $\left\vert X_{B}^{\symbol{94}}\cap \Omega _{n-2}^{-}(q)\right\vert
_{p}=(n-2)_{p}\leq (n-2)$. Therefore 
\begin{equation} \label{ciokciok}
q^{\frac{n(n-2)-12}{12}}\leq \left\vert X_{B}\right\vert _{p}\leq (2,q)(n-2) 
\end{equation}
by Corollary \ref{kralj}(v), where $n\geq 6$ by Remark \ref{mag3}. It is
straightforward to see that, only $n=6$ and $q=2,4$ or $8$ fulfill (\ref{ciokciok}). However, all these cases are ruled out. Indeed, $q=2$ cannot occur by Lemma \ref{cyclont}, whereas $q=4$ or $8$
cannot occur since $\Phi _{8}^{\ast }(2)=17$ and $\Phi _{12}^{\ast }(2)=13$ respectively.
\end{proof}

\section{$X_{x}$ is a large subgroup of $X$}

The aim of this section is to show that $\left( \Phi _{ef}^{\ast
}(p),k+1\right) >1$ implies that $X_{x}$ is a large subgroup of $X$. This allows us to treat Cases (1.a) and (1.b) of
Theorem \ref{daleko} simultaneously.

The proof strategy is as follows. Since $\left( \Phi _{ef}^{\ast
}(p),k+1\right) >1$, and since $k+1$ divides the length of each $G_{x}$%
-orbit distinct from $\left\{ x\right\} $ by Lemma \ref{PP}(3), we derive that $%
\Phi _{ef}^{\ast }(p)\mid \left\vert X_{x}\right\vert $ and, moreover, that $%
X_{x}$ contains a Singer cyclic subgroup of $X$ for $e=n$ by using simple
group-theoretical arguments. Then $X_{x}^{\symbol{94}}$ is a subgroup
of $SL_{n}(q)$ classified in Theorem 3.1 
of \cite{BP} for $e<n$. For $e=n$, we use instead the classification for $%
X_{x}^{\symbol{94}}$ contained in \cite{Bre}. By combining the fact that the candidates $X_{x}$ must fulfill $\left[ X:X_{x}\right]
=k^{2}$ with the information contained in Lemma 3.13 of \cite{ABD3}, we show that $\left\vert \mathrm{Out(}%
X\mathrm{)}\right\vert \geq \left\vert X_{x}\right\vert _{p}$. Hence, $%
X_{x}$ is a large subgroup of $X$ by Lemma \ref{Sing}(2).

\bigskip

\begin{lemma}
\label{beforeXmas}If $\left( \Phi _{ef}^{\ast }(p),k+1\right) >1$, then $%
\Phi _{ef}^{\ast }(p)\mid \left\vert X_{x}\right\vert $. Moreover, if $e=n$,
then $X_{x}$ contains a Singer cyclic subgroup of $X$.
\end{lemma}

\begin{proof}
Let $w$ be any primitive prime divisor $w$ of $q^{e}-1$ dividing $k+1$ and
let $W$ be a Sylow $w$-subgroup of $G_{x}$. Then $w\equiv 1 \pmod {ef}$ and
hence $(u,ef)=1$. Thus $w$ is coprime to $\left\vert 
\mathrm{Out(}X\mathrm{)}\right\vert $ (see \cite{KL}, Table 5.1.A). Therefore, $%
W\leq X_{x}$.

Assume that $W$ fixes a further point $y$ of $\mathcal{D}$. Then $\left\vert
y^{G_{x}}\right\vert $ has order prime to $w$. However, this is impossible,
since $k+1$ divides $\left\vert y^{G_{x}}\right\vert $ by Lemma \ref{PP}(3), and since $u \mid k+1$.
Thus $x$ is the unique point fixed by $W$. Then $N_{X}(W)\leq N_{G}(W)\leq
G_{x}$, and hence $N_{X}(W)\leq X_{x}$, as $W\leq X_{x}$. \\
The group $W$ preserves a unique $(e-1)$-subspace $\mathcal{H}$ of $PG_{n-1}(q)$ and acts
faithfully and irreducibly on it (see \cite{He1/2}, Theorem 3.5). So does $N_{X}(W)$. It is not difficult to
see that $W$ lies in a cyclic subgroup $Q$\ of $X_{\mathcal{H}}$ inducing a
Singer cyclic subgroup $\bar{Q}$ on $\mathcal{H}$. The order of $\bar{Q}$ is
well known and, for instance, can be found in \cite{Bre},\cite{Hes},\cite%
{Hup}, or determined in \cite{KL}, Section 4.3. Thus $\Phi _{ef}^{\ast }(p)$ divides the order of $\bar{Q}$, and
hence that of $Q$. Then $\Phi _{ef}^{\ast }(p)$ divides $\left\vert
N_{X}(W)\right\vert $, and hence $\left\vert X_{x}\right\vert $, and as $%
Q\leq N_{X}(W)$. Finally, if $e=n$, clearly $\mathcal{H}=PG_{n-1}(q)$ and hence $Q$ is a Singer cyclic subgroup of $X$.
\end{proof}

\medskip

Now, case $e=n,n-1$ and $n-2$ are analyzed separately.

\medskip

\begin{lemma}
\label{e=n}If $\left( \Phi _{nf}^{\ast }(p),k+1\right) >1$, then $X_{x}$ is
a large maximal $\mathcal{C}_{3}$-subgroup of $X$.
\end{lemma}

\begin{proof}
Assume that $\left( \Phi _{nf}^{\ast }(p),k+1\right) >1$. Then $X_{x}$ contains a Singer cyclic subgroup of $X$ by
Lemma \ref{beforeXmas}. Let $Y$ be a maximal subgroup of $X$ containing $X_{x}$%
. It follows from the Main Theorem of \cite{Bre} that, either $Y\in \mathcal{%
C}_{3}(X)$ or, since $\Phi _{nf}^{\ast }(p)>1$ and $\Phi _{nf}^{\ast
}(p)\mid \left\vert Y\right\vert $, one of the following holds:

\begin{enumerate}
\item $X\cong PSU_{n}(q)$, where $(n,q)=(3,3),(3,5)$ or $(5,2)$.

\item $X\cong PSp_{n}(q)$, where either $q$ is even and $\Omega
_{n}^{-}(q)\leq Y$, or $(n,q)=(8,2)$.

\item $X\cong P\Omega _{6}^{-}(3)$.
\end{enumerate}

The case $X\cong PSU_{3}(3)$ is ruled out in Lemma 3.6 of \cite{MF}, as $%
PSU_{3}(3)\cong G_{2}(2)^{\prime }$. The group $PSU_{3}(5)$
is ruled out by \cite{At}, as it does not have a transitive permutation representation of
square degree. The same conclusion holds for $PSp_{8}(2)$ by \cite{BHRD},
Tables 8.48--8.49. If $X$ is isomorphic either to $PSU_{5}(2)$ or to $%
P\Omega _{6}^{-}(3)$, then $X_{x}\cong PSL_{2}(11)$ and $k=2^{4}\cdot 3^{2}$%
, or $X_{x}\cong A_{7}$ and $k=2^{4}\cdot 3^{2}$ , respectively, by \cite{At}%
. However, these cases are ruled out since $k+1$ does not divide the order
of $\mathrm{Aut(}X\mathrm{)}$ and hence that of $G$.

Suppose that $X\cong PSp_{n}(q)$, with $q$ even, and $\Omega
_{n}^{-}(q)\leq Y$. Then $Y\cong SO_{n}^{-}(q)$ by \cite{KL}, Proposition
4.8.6. If $X_{x} \neq Y$, then $X_{x}\cong \Omega _{n}^{-}(q)$ again by Main Theorem of \cite%
{Bre} but applied to $Y$. Thus, $X_{x}$ is isomorphic either to $\Omega _{n}^{-}(q)$ or to $SO _{n}^{-}(q)$. If $\rho$ is the index of $X_{x}$ in $Y$, then $\rho =1,2$ and hence $k^{2}=\frac{q^{n(n+2)/4}}{\rho}(q^{n/2}-1)$, which has no solutions by \cite{Rib}, A3.1, since $n\geq 4$ by Remark \ref{mag3}. Thus, $Y\in \mathcal{C}_{3}(X)$.

The group $X$ contains a unique conjugacy class of subgroups isomorphic to $%
Y $, except for $X\cong P\Omega _{n}^{-}(q)$ and $O_{n/2}(q^{2}) $ with $qn/2$ odd by \cite{KL}, Section 4.3, since $\Phi _{nf}^{\ast
}(p)\mid \left\vert X_{x}\right\vert $ and since $n \geq 6$ (see Proposition 4.3.20.(I) for the exception). Then $X_{x}=Y$ by Lemma \ref{Nov}(1) when $Y$ is none of the
exceptions. In the exceptional cases, $X_{x}$ is a type 1 novelty with
respect to $Y$ by Lemma \ref{Nov}(2). However, this is impossible by \cite%
{KL}, Tables 3.5.H--I, for $n\geq 13$ and \cite{BHRD} for $3\leq n\leq 12$.
Thus $X_{x}=Y$ and hence $X_{x}$ is a maximal $\mathcal{C}_{3}$-subgroup of $X$.

Suppose that $X_{x}$ is not a large subgroup of $X$. Then $\left\vert
X_{x}\right\vert _{p}<\left\vert \mathrm{Out(}X\mathrm{)}\right\vert $ by
Lemma \ref{Sing}. Hence, either $X\cong PSL_{n}(q)$ and $X_{x}$ is of type $%
GL_{1}(q^{n})$, or $X\cong PSU_{n}(q^{1/2})$, with $n$ odd, and $X_{x}$ is
of type $GU_{1}(q^{n/2})$, or $X\cong P\Omega _{6}^{-}(q)$ by \cite{ABD3},
Lemma 3.13, since $\Phi _{nf}^{\ast }(p)\mid \left\vert X_{x}\right\vert $.

Assume that $X\cong PSL_{n}(q)$ and $X_{x}$ is of type $%
GL_{1}(q^{n})$. Then  
\[
q^{n^{2}-2}<\left\vert X\right\vert \leq \left\vert G\right\vert \leq \left( (q^{n}-1)nf\right) ^{3}\leq
q^{3n+3}n^{3} 
\]%
since $G_{x}$ is a large subgroup of $G$ normalizing a Singer cyclic subgroup of $X$ and by \cite{AB}, Corollary 4.3.(i). Hence $n=3,4$. If $n=4$, then
\begin{equation}\label{sbs}
k^{2}=\frac{q^{6}(q-1)^{2}(q^{2}-1)(q^{2}+q+1)}{4} 
\end{equation}%
by \cite{KL}, Proposition 4.3.6.(II). On the other hand, $\frac{k+1}{(k+1,f)}\mid 4 \frac{q^{4}-1}{(q-1)(4,q-1)} $ by
Lemma \ref{Orbits}, and hence $k+1\leq 4f\frac{q^{4}-1}{(q-1)(4,q-1)}$, which, compared to (\ref{sbs}), yields $q \leq 5$. However, none of these values fulfills (\ref{sbs}). Thus $n=3$ and hence $\left\vert X_{x}\right\vert
=3 \frac{q^{2}+q+1}{(3,q-1)}$ again by \cite{KL}, Proposition 4.3.6.(II). Therefore, 
\[
k^{2}=\frac{q^{3}(q-1)^{2}(q+1)}{3 }. 
\]%
If $q\equiv 1 \pmod 3$, then both $q$ and $q+1$ are squares, but this contradicts 
\cite{Rib}, A5.1. If $q\equiv 0 \pmod 3$, then both $q^{3}/3$ and $q+1$ are squares
and hence $q=3$ again by \cite{Rib}, A5.1. So $\left\vert X_{x}\right\vert
_{p}=3$, whereas $\left\vert \mathrm{Out(}X\mathrm{)}\right\vert =2$. Therefore, $%
q\equiv 2 \pmod 3$ and hence both $q$ and $\frac{q+1}{3}$ are squares, but this is impossible by \cite{Rib}, P5.5.

Assume $X\cong PSU_{n}(q^{1/2})$, with $n$ odd, and $X_{x}$ is
of type $GU_{1}(q^{n/2})$. Then  
\[
(q^{1/2}-1)q^{\frac{n^{2}-3}{2}}<\left\vert X\right\vert \leq \left\vert G\right\vert \leq\left(
(q^{n/2}+1)nf\right) ^{3}\leq (q^{n/2}+1)^{3}q^{3/2}n^{3} 
\]%
since $G_{x}$ is a large subgroup of $G$ normalizing a Singer cyclic subgroup of $X$, since $f$ is even and hence $q^{1/2}\geq f$, and by \cite{AB}, Corollary 4.3.(ii). Thus either $n=3$, or $n=5$ and $q^{1/2}=2$, since $n$ is odd. We have seen above that $PSU_{5}(2)$ cannot occur, hence the latter is ruled out. Thus $n=3$ and hence $%
\left\vert X_{x}\right\vert =3 \frac{q-q^{1/2}+1}{(3,q^{1/2}+1)}$ and%
\begin{equation}\label{bmonday}
k^{2}=\frac{q^{3/2}(q^{1/2}-1)(q^{1/2}+1)^{2}}{3 }\text{.} 
\end{equation}
by \cite{KL}, Proposition 4.3.6.(II). Then $q^{1/2}\equiv 0,1 \pmod 3 $ by \cite{Rib}, A3.1. Since $\frac{k+1}{(k+1,f)}\mid 3\left( q-q^{1/2}+1\right) $ by
Lemma \ref{Orbits}, it follows that $k+1\leq 3f\left( q-q^{1/2}+1\right)$, which, compared to (\ref{bmonday}), yields $q^{1/2}=3,4$. However, none of these fulfills (\ref{bmonday}).

Finally, assume that $X\cong P\Omega _{6}^{-}(q)$. Then $%
\left\vert \mathrm{Out(}X\mathrm{)}\right\vert =(4,q^{3}+1)2f$ and $X_{x}$
is a subgroup of $\frac{1}{(2,q-1)}GU_{3}(q)$ by \cite{KL}, Propositions
4.3.18.(II) and 4.3.20.(II), as $\Phi _{6f}^{\ast }(p)\mid \left\vert
X_{x}\right\vert $. Since $q\mid
\left\vert X_{x}\right\vert _{p}$ by Corollary \ref{kralj}(v), and since $\left\vert
X_{x}\right\vert _{p}<(4,q^{3}+1)2f$, it follows that $q<8f$ and hence $q=2,3,5$ or $7$. Actually, $%
q\neq 2,3$ as $P\Omega _{6}^{-}(2)$ is ruled out in Lemma \ref{cyclont}, since $%
\Phi _{6}^{\ast }(2)=1$, and $P\Omega _{6}^{-}(3)$ is ruled out above in
this proof. Since $\left[ X:X_{x}\right] =k^{2}$, it is not difficult to obtain that $2\cdot 7\cdot 13$ or $43$
divides $\left\vert X_{x}\right\vert $ according to whether $q=5$ or $7$
respectively. However, no cases arise by \cite{BHRD}, Tables 8.10--8.11,
since $X_{x}$ is a maximal $\mathcal{C}_{3}$-subgroup of $X$.
\end{proof}

\begin{lemma}
\label{e=n-1}If $\left( \Phi _{(n-1)f}^{\ast }(p),k+1\right) >1$, then $%
X_{x} $ is a large subgroup of $X$.
\end{lemma}

\begin{proof}
Assume that $\left( \Phi _{(n-1)f}^{\ast }(p),k+1\right) >1$. If $X\cong PSU_{n}(q^{1/2})$, with $n$ even, then either 
$n=4$ and $PSL_{3}(2^{t})\leq X_{x}$ with $(t,q)=(1,3),(1,5),(2,3)$, or $X_{x}$
preserves either a point or a hyperplane of $PG_{n-1}(q)$ by \cite{BP},
Theorem 4.2. If $(t,q)=(2,3)$, then $X_{x}\cong PSL_{3}(4)$ by \cite{At} and
hence $k^{2}=162$, a contradiction. If $(t,q)=(1,3)$, then $PSL_{3}(2)\leq
X_{x}\leq M$, where $M$ is one of the following maximal subgroups of $X$ by \cite{At}: 
$PSL_{3}(4)$, $PSU_{3}(3)$ or $A_{7}$. Either $X_{x}\cong PSL_{3}(2)$ or $X_{x}=M$, since $PSL_{3}(2)$ is maximal in $M$.
Actually, the unique admissible case is $X_{x}\cong A_{7}$ and $k=36$, since $\left[ X:X_{x}\right] =k^{2}$. However, this case is ruled out, since $k+1$ does not divide the order of $%
P\Gamma U_{4}(3)$, and hence that of $G$.

Assume that $X_{x}$ preserves either a point or a hyperplane of $PG_{n-1}(q)$%
. Actually, $G_{x}$ is the stabilizer of a non-isotropic point $P$ of $%
PG_{n-1}(q)$ by Proposition \ref{greatsaxl} and by \cite{KL}, Proposition 4.1.6.(II), since $\Phi _{ef}^{\ast }(p)\mid \left\vert X_{x}\right\vert $. Hence, the actions of $G$ on the point set of $\mathcal{D}$
and on the set of non-singular points of $PG_{n-1}(q)$ are equivalent. Thus, $k^{2}=q^{\left( n-1\right) /2}\frac{q^{n/2}-1}{q^{1/2}+1}$ and hence $f\equiv 0 \pmod 4$ as $q^{\left( n-1\right) /2}$ must be a
square, being $n$ is even. Then
\begin{equation}
k^{2}=s^{2\left( n-1\right) }\frac{\left( s^{n}+1\right) \left(
s^{n}-1\right) }{s^{2}+1}\text{,}   \label{kq}
\end{equation}%
where $s=q^{1/4}$. Then $s$ is odd by \cite{Rib}, A3.1. Moreover, if we consider a
further non-singular point $y$ of $PG_{n-1}(q)$ such that $y\in x^{\perp }$%
, we see that $\left\vert y^{G_{x}}\right\vert $ divides $%
(s^{2}+1)(s^{2(n-1)}+1)$. Then $k+1\mid $ $(s^{2}+1)(s^{2(n-1)}+1)$ by Lemma %
\ref{PP}(3). Actually, $k+1\mid $ $\frac{(s^{2}+1)(s^{2(n-1)}+1)}{4}$ as $s$
is odd, which, compared to (\ref{kq}), leads to a contradiction as $n \geq 4$. Thus, $X\ncong PSU_{n}(q^{1/2})$ for $\left( \Phi _{(n-1)f}^{\ast
}(p),k+1\right) >1$.

Assume that $X\cong P\Omega _{n}(q)$, with $nq$ odd. Assume also that $\left\vert
X_{x}\right\vert _{p}<\left\vert \mathrm{Out(}X\mathrm{)}\right\vert $. Then 
$\left\vert X_{x}\right\vert _{p}<2f$, as $\left\vert \mathrm{Out(}X%
\mathrm{)}\right\vert =2f$. Since $X_{x}<X<PSL_{n}(q)$ and $\Phi
_{(n-1)f}^{\ast }(p)$ divides $\left\vert X_{x}\right\vert $, it follows that $X_{x}^{\symbol{94}}$ is one of the groups listed in Theorem 3.1 of \cite{BP}.

Assume that $X_{x}^{\symbol{94}}$ is a nearly simple subgroup of $SL_{n}(q)$%
. Suppose that $\mathrm{Soc(}X_{x}\mathrm{)}$ is alternating. Then $X_{x}$
arises from the action of $A_{\ell }$ on its fully deleted permutation
module, since $nq$ is odd. Then $q$ is prime, $n=\ell -1$ or $\ell -2$
according as $q$ does not divide or does divide $n$ respectively. In both cases, we have $%
\Phi _{n-1}^{\ast }(q)=n$, with $n=4,6$, whereas $n$ is odd.

Assume that $\mathrm{Soc(}X_{x}\mathrm{)}$ is a sporadic group. Then $%
\mathrm{Soc(}X_{x}\mathrm{)}\cong M_{11}\trianglelefteq X_{x}$, $n=5$ and $%
q=3$. So $\left\vert X_{x}\right\vert _{3}=3^{2}$, whereas $\left\vert 
\mathrm{Out(}X\mathrm{)}\right\vert =2$.

Assume that $\mathrm{Soc(}X_{x}\mathrm{)}$ is a simple Lie type group in
cross characteristic. $\left\vert
X_{x}\right\vert _{p}<\left\vert \mathrm{Out(}X\mathrm{)}\right\vert $
implies $n=7$, $q=5$ and $\mathrm{Soc(}X_{x}\mathrm{)}$ isomorphic to one of
the groups $PSL_{2}(7)$, $PSL_{2}(8)$ or $PSU_{3}(3)$. However, these cases
are ruled out. Indeed, $31$ must divide $\left\vert X_{x}\right\vert $, since $\left\vert%
P\Omega _{7}(5)\right\vert$ is divisible by $31$, and not by $31^{2}$, and since $\left[
X:X_{x}\right] =k^{2}$.

Assume that $\mathrm{Soc(}X_{x}\mathrm{)}$ is a simple group of Lie type in
characteristic $p$. Then $n=7$, $q$ is odd, with $q$ dividing $\left\vert
X_{x}\right\vert $ by Theorem 3.1 of \cite{BP}. So $q<\left\vert \mathrm{Out(%
}X\mathrm{)}\right\vert =2f$, which has no solutions for $q$ odd.

Clearly $X_{x}^{\symbol{94}}$ is not a Symplectic Type Example. Also, $%
X_{x}^{\symbol{94}}$ is not a Classical Example since $X_{x}<X\cong P\Omega
_{n}(q)$. If $X_{x}^{\symbol{94}}$ is an Imprimitive Example, then $q=p$ and 
$X_{x}^{\symbol{94}}\leq GL_{1}(p)\wr S_{n}$, where either $q=3$ and $n=5,7$,
or $q=5$ and $n=7$ by Theorem 3.1 of \cite{BP}, being $n$ odd. The case $%
q=3$ and $n=5$ cannot occur, since $P\Omega _{5}(3)\cong PSU_{4}(2)$ has
been ruled out above. The groups $P\Omega _{7}(3)$ and $P\Omega _{7}(5)$ are
ruled out, as they do not have a transitive permutation representation of
square degree by \cite{BHRD}, Tables 8.39 and 8.40. Therefore, $X_{x}$ fixes
either a point or a hyperplane of $PG_{n-1}(q)$. So does $G_{x}$, as $%
X_{x}\trianglelefteq G_{x}$. Then $G_{x}$ should be one of the groups listed in \cite{ABD3}, Lemma
3.13, since $\left\vert X_{x}\right\vert _{p}<\left\vert \mathrm{Out(}X%
\mathrm{)}\right\vert $, whereas it is not, and we obtain a contradiction. Thus, $\left\vert
X_{x}\right\vert _{p}\geq \left\vert \mathrm{Out(}X\mathrm{)}\right\vert $
and hence $X_{x}$ is a large subgroup of $X$ by Lemma \ref{Sing}(2).
\end{proof}

\begin{theorem}
\label{kappa+1}$X_{x}$ is a large subgroup of $X$ such that $\left( \Phi
_{ef}^{\ast }(p),\left\vert X_{x}\right\vert \right) >1$.
\end{theorem}

\begin{proof}
Since the assertion follows from Lemma \ref{e=n} and \ref{e=n-1} for $e=n$
or $e=n-1$, respectively, in order to prove the theorem, it remains to
tackle the case $e=n-2$ and $X\cong P\Omega _{n}^{+}(q)$. Clearly, $\left( \Phi _{(n-2)f}^{\ast }(p),\left\vert X_{x}\right\vert
\right) >1$.

Suppose that $X_{x}$ is not a large subgroup of $X$. Then $\left\vert
X_{x}\right\vert _{p}<\left\vert \mathrm{Out(}X\mathrm{)}\right\vert $ by
Lemma \ref{Sing}(2). Recall that $n\geq 6$ by Remark \ref{mag3}, and that $%
\left\vert \mathrm{Out(}X\mathrm{)}\right\vert $ is either $2(4,q^{n/2}-1)f$
or $6(4,q^{n/2}-1)f$ according as~$n\neq 8$ or $n=8$ respectively.

Since $X_{x}<X<PSL_{n}(q)$ and $\Phi _{(n-2)f}^{\ast }(p)$ divides $%
\left\vert X_{x}\right\vert $, the group $X_{x}^{\symbol{94}}$ is one the
groups listed in Theorem 3.1 of \cite{BP}.

Assume that $X_{x}^{\symbol{94}}$ is a nearly simple subgroup of $SL_{n}(q)$%
. Suppose that $\mathrm{Soc(}X_{x}\mathrm{)}$ is alternating. If $X_{x}^{%
\symbol{94}}$ arises from the action of $A_{\ell }$ on its fully deleted
permutation module, then $q$ is a prime (hence $f=1$), $n=e+2$, $A_{e+4}\leq
X_{x}\leq S_{e+4}\times Z_{2}$ where $%
(p,e)=(2,4),(2,10),(2,12),(2,18),(3,4),(3,6),(5,6)$. If $(p,e)\neq (3,4)$,
then $\left\vert X_{x}\right\vert _{2}\geq 2^{6}$, $\left\vert
X_{x}\right\vert _{3}\geq 9$ and $\left\vert X_{x}\right\vert _{5}\geq 25$,
whereas $\left\vert \mathrm{Out(}X\mathrm{)}\right\vert \leq 24$ in the
first case and $\left\vert \mathrm{Out(}X\mathrm{)}\right\vert \leq 8$ in
the remaining ones. So, these cases are excluded. Then $(p,e)=(3,4)$, $%
X\cong P\Omega _{8}^{+}(3)$ and $A_{8}\leq X_{x}\leq S_{8}\times Z_{2}$.
However, this is impossible since $13$ must divide $\left\vert
X_{x}\right\vert $, since $P\Omega _{8}^{+}(3)$ is divisible by $13$, and not by $%
13^{2}$, and since $\left[ X:X_{x}\right] =k^{2}$.

Suppose that $\mathrm{Soc(}X_{x}\mathrm{)}$ does not arise from the action of $A_{\ell }$ on its
fully deleted permutation module. Then $n=8$ and either $q=5$ and $\ell
=7,8,9,10$, or $q=3$ and $\ell =9,8$. The above argument excludes $q=3$. Case $q=5$ is ruled out too. Indeed, $P\Omega _{8}^{+}(5)$ is
divisible by $31$ and not by $31^{2}$.

Assume that $\mathrm{Soc(}X_{x}\mathrm{)}$ is either a sporadic group or a simple Lie type group in cross
characteristic. Set $S\cong \mathrm{Soc(}X_{x}\mathrm{)}$. By Theorem 3.1 of 
\cite{BP}, one of the following holds:

\begin{enumerate}
\item $(n,q)=(6,3)$ and $S\cong M_{12},PSL_{3}(4)$;

\item $(n,q)=(8,3)$ and $S\cong P\Omega _{8}^{+}(2)$, $PSp_{6}(2)$;

\item $(n,q)=(8,5)$ and $S\cong PSL_{3}(4)$, $P\Omega _{8}^{+}(2)$, $%
PSp_{6}(2)$, or $Sz(8)$;

\item $(n,q)=(14,2)$ and $S\cong PSL_{2}(13),$ $G_{2}(3)$;

\item $(n,q)=(20,2)$ and $S\cong J_{2},PSL_{2}(19)$.
\end{enumerate}

$M_{12},PSL_{3}(4)$ are not subgroups of $P\Omega _{6}^{+}(3)$, hence Case
(1) cannot occur. Cases (2) and (3) cannot occur by Lemma \ref{cyclont}. Finally, $127$ divides the order of $P\Omega
_{14}^{+}(2)$ and of $P\Omega _{20}^{+}(2)$. Then $127^{2}$ must divide the order of 
$S$, since $\left[ X:X_{x}\right] =k^{2}$, and since $S\trianglelefteq
X_{x}\leq \mathrm{Aut(}S\mathrm{)}$. However, this is impossible for $S$ as
in Cases (4) or (5), and hence these are ruled out.

The group $X_{x}$ cannot be a Classical Example by \cite{KL}, Section 4.8, as $X_{x}<X$. The same
argument used above to rule out Case (1) excludes the possibility for $X_{x}$ to
be a Symplectic Type Example as well. Now, assume that is an Imprimitive
Example. Then $q=p$, $e=n-2$, $\Phi _{e}^{\ast }(p)=e$ and $X_{x}^{\symbol{94%
}}\leq GL_{1}(q)\wr S_{n}$ with $%
(n,q)=(6,2),(12,2),(14,2),(20,2),(6,3),(8,3),(8,5)$. Actually, $(n,q)\neq
(6,2)$ since $P\Omega _{6}^{+}(2)\cong A_{8}$ has been ruled out in Lemma
3.5 of \cite{MF}. In the remaining cases, there is a prime $w$ such that $%
w\mid \left\vert X\right\vert $, $w^{2}\nmid \left\vert X\right\vert $. Such
a prime $w$ is $31$, $127$, $257$, $13$, $13$, $31$, respectively. Then $%
w\mid \left\vert X_{x}\right\vert $, since $\left[ X:X_{x}\right] =k^{2}$,
whereas $w>n$ and $w \nmid q-1$. So, $X_{x}$ cannot be an Imprimitive Example. Thus, again by
Theorem 3.1 of \cite{BP}, either $X_{x}$ preserves a point or a hyperplane of $%
PG_{n-1}(q)$, or $X_{x}^{\symbol{94}}$ is an Extension Field Example.

Assume that $X_{x}^{\symbol{94}}$ is an Extension Field Example. Since $%
X_{x}^{\symbol{94}}$ preserves an orthogonal form on $V_{n}(q)$, Theorem
3.1 of \cite{BP} implies that either $n\equiv 2 \pmod 4$, $q$ is odd and $%
X_{x}^{\symbol{94}}\leq \Gamma O_{n/2}(q^{2})$, or $n\equiv 0 \pmod 4$
and $X_{x}^{\symbol{94}}\leq \Gamma U_{n/2-1}(q)$ respectively. Thus $X_{x}$ lies in a
maximal member $Y$ of $\mathcal{C}_{3}(X)$ of type $O_{n/2}(q^{2})$ or $%
GU_{n/2-1}(q)$ respectively. If $X_{x}<Y$, then $X_{x}$ is a type 1 novelty with respect
to $Y$ by Lemma \ref{Nov}(2). However, this is impossible by \cite{KL},
Tables 3.5.H--I, for $n\geq 13$ and \cite{BHRD} for $6\leq n\leq 12$. Therefore, $%
X_{x}$ is a maximal $\mathcal{C}_{3}$-subgroup of $X$. Thus, $\left\vert
X_{x}\right\vert _{p}\geq q^{2}$ for $n\neq 8$, and $\left\vert
X_{x}\right\vert _{p}\geq q^{3}$ for $n=8$. Hence, $\left\vert
X_{x}\right\vert _{p}<\left\vert \mathrm{Out(}X\mathrm{)}\right\vert $
implies $n=8$ and $q=2$, as $\left\vert \mathrm{Out(}X\mathrm{)}\right\vert $
is either $2(4,q^{n/2}-1)f$ or $6(4,q^{n/2}-1)f$ according as $n\neq 8$, or $%
n=8$, respectively. However, $X\ncong P\Omega _{8}^{+}(2)$ by Lemma \ref%
{cyclont}.

Finally, assume that $X_{x}$ preserves either a point or a hyperplane of $%
PG_{n-1}(q)$. Note that $X_{x}$ does not lie in a maximal parabolic subgroup of $X$ by Proposition \ref{greatsaxl}. Then $X_{x}$ preserves either a non-degenerate
point, or a non-degenerate line of $PG_{n-1}(q)$ of type $-$ by \cite{KL}, Proposition
4.1.6.(II), since $\Phi _{n-2}^{\ast }(p)\mid \left\vert
X_{x}\right\vert $. So does $G_{x}$. Then $X\cong P\Omega _{6}^{+}(q)$ by \cite{ABD3}, Lemma 3.13. However, this is impossible by Lemma \ref{e=n}, since $X\cong PSL_{4}(q)$. This completes the proof.
\end{proof}

\section{Proof of Theorem \protect\ref{main}}

This final section is devoted to the completion of the proof of Theorem \ref{main}. The proof strategy is as follows. $X_{x}$ lies in a large maximal subgroup $Y$ of $X$ by Theorem \ref{daleko}. Moreover, $Y$ is classified in \cite{AB}. We filter the $Y$-candidates with respect to $\left( \Phi _{ef}^{\ast }(p),\left\vert
Y\right\vert \right) >1$. Then, by comparing the information contained in \cite{KL} and \cite{BHRD} on the number of conjugacy classes of the subgroups of $X$ isomorphic to $Y$, with those contained in Lemma \ref{Nov}, we obtain $X_{x}=Y$. Hence, $X_{x}$ is a known large maximal subgroup of $X$. At this point, we can precisely evaluate the diophantine equation $k^{2}=[X:X_{x}]$, and, by combining some of the results contained in \cite{Rib} with properties arising from the geometry of the classical groups, we reduce to the case $X \cong PSL_{3}(3)$, $k=12$ and hence to Cases (3) and (4) of Theorem \ref{main}.

\bigskip

\begin{theorem}
\label{LergeGeo}$X_{x}$ lies in a large maximal geometric subgroup of $X$.
\end{theorem}

\begin{proof}
Let $Y$ be a maximal subgroup of $X$ containing $X_{x}$. Then $Y$ is a large
subgroup of $X$ such that $\left( \Phi _{ef}^{\ast }(p),\left\vert
Y\right\vert \right) >1$ by Theorem \ref{kappa+1}. Assume that $Y$
is a member of $\mathcal{S}(X)$. Then the pairs $(X,Y)$ are listed in \cite%
{AB}, Proposition 4.28. Filtering such pairs $(X,Y)$ with respect to the property $%
\left( \Phi _{ef}^{\ast }(p),\left\vert Y\right\vert \right) >1$, one of the
followings occurs:

\begin{enumerate}
\item $X\cong P\Omega _{8}^{+}(q)$ and $Y$ is isomorphic either to $\Omega
_{7}(q)$ or to $Sp_{6}(q)$ according to whether $q$ is odd or even
respectively.

\item $X\cong \Omega _{7}(q)$ and $Y\cong G_{2}(q)$.

\item $X\cong Sp_{6}(q)$, $q$ even, and $Y\cong G_{2}(q)$.

\item $X\cong PSp_{4}(q)$, $q$ an odd power of $%
2 $, $q>2$, and $Y\cong Sz(q)$.

\item $(X,Y)$ are listed in the Table \ref{signprime}.
\end{enumerate}

\begin{table}
\centering
\caption{\label{signprime}\small{Admissible pairs $(X,Y)$ and significant
primes dividing the order of $X$.}}
\begin{tabular}{|l|l|l|l|l|l|}
\hline
$X$ & $(n,q,e)$ & $\left\vert \mathrm{Out(}X\mathrm{)}\right\vert $ & $\Phi
_{ef}^{\ast }(p)$ & $Y$ & $s$ \\
\hline 
$PSL_{n}(q)$ & $(5,3,5)$ & $2$ & $11^{2}$ & $M_{11}$ & $13$ \\ 
& $(4,2,4)$ & $2$ & $5$ & $A_{7}$ & $-$ \\ 
& $(4,7,4)$ & $2$ & $5^{2}$ & $PSU_{4}(2)$ & $19$ \\
\hline 
$PSp_{n}(q)$ & $(12,2,12)$ & $1$ & $13$ & $S_{14}$ & $31$ \\ 
& $(6,5,6)$ & $2$ & $7$ & $J_{2}$ & $31$ \\ 
& $(4,7,4)$ & $2$ & $5^{2}$ & $A_{7}$ & $-$ \\
\hline 
$PSU_{n}(q)$ & $(4,3,3)$ & $8$ & $5$ & $A_{7},PSL_{3}(4)$ & $-$ \\
\hline 
$P\Omega _{n}^{-}(q)$ & $(18,2,12)$ & $2$ & $19$ & $A_{20}$ & $257$ \\ 
& $(12,2,12)$ & $2$ & $13$ & $A_{13}$ & $31$ \\ 
& $(10,2,10)$ & $2$ & $11$ & $A_{12}$ & $17$ \\
\hline 
$P\Omega _{n}(q)$ & $(7,3,6)$ & $2$ & $7$ & $PSp_{6}(2),S_{9}$ & $13$ \\ 
& $(7,5,6)$ & $2$ & $7$ & $PSp_{6}(2)$ & $31$ \\ 
\hline
$P\Omega _{n}^{+}(q)$ & $(14,2,12)$ & $2$ & $13$ & $A_{16}$ & $127$ \\ 
& $(8,3,6)$ & $24$ & $7$ & $P\Omega _{8}^{+}(2)$ & $13$ \\ 
& $(8,5,6)$ & $24$ & $7$ & $P\Omega _{8}^{+}(2)$ & $31$ \\
\hline
\end{tabular}
\end{table}

Assume that Cases (1)--(4) hold. Since $X_{x}$ is a large subgroup of $X$,
it follows that $\left\vert X_{x}\right\vert >\left\vert X\right\vert ^{1/3}$%
. By using the lower bounds for $\left\vert X\right\vert $ provided in Lemma 4.3 of\ 
\cite{AB}, we see that $\left\vert X_{x}\right\vert $ is greater than $q^{28/3}/2$, 
$q^{7}/\sqrt[3]{4}$, $q^{7}/\sqrt[3]{2}$ and $q^{10/3}/\sqrt[3]{2}$ in Cases
(1)--(4) respectively.

Assume that $X\cong P\Omega _{8}^{+}(q)$. Then $Y$ is isomorphic either to $%
\Omega _{7}(q)$, or to $Sp_{6}(q)$, according to whether $q$ is odd or even
respectively. Note that, $q=2$ is ruled out by Lemma \ref{cyclont}. Also, $%
q=3$ cannot occur. Indeed, in this case $13\mid \left\vert X\right\vert $, $%
13^{2}\nmid \left\vert X\right\vert $, but $13$ does not divide $\left\vert
X_{x}\right\vert $ and $k^{2}$. Hence $q>3$.\\
Assume that $X_{x} \neq Y$. Then either $X_{x}\cong \Omega
_{6}^{-}(q).Z_{\epsilon }$, where $\epsilon =1,2$, or $X_{x}\cong G_{2}(q)$, or $X_{x} \cong Sp_{2}(q^{3}) \cdot Z_{3}$ with $q$ even and $2<q<2^{8}$ by \cite{BHRD}, Tables 8.28--8.30, 8.33--8.34 and 8.39--8.42, since $\left\vert
X_{x}\right\vert >$ $q^{28/3}/2$ and $\left( \Phi _{6f}^{\ast
}(p),\left\vert X_{x}\right\vert \right) >1$. It is not difficult to check that, $k^{2}=[X:X_{x}]$ has no solutions for $X_{x} \cong Sp_{2}(q^{3}) \cdot Z_{3}$ with $q$ even and $2<q<2^{8}$. So, this case is ruled out.

If $X_{x}\cong \Omega _{6}^{-}(q).Z_{\epsilon }$, then 
\[
k^{2}=\frac{q^{6}(q+1)(q^{2}+1)(q-1)^{2}(q^{2}+q+1)}{%
(4,q^{4}-1)\epsilon } 
\]%
and hence $q^{2}+q+1$ is a square, but this is impossible by \cite{Rib},
A7.1.

If $X_{x}\cong G_{2}(q)$, then $k^{2}=\frac{q^{6}(q^{4}-1)^{2}}{(4,q^{4}-1)}$
and hence $k=\frac{q^{3}(q^{4}-1)}{(2,q-1)}$. Therefore $k+1=q^{3}(q^{2}-1)%
\frac{q^{2}+1}{(2,q-1)}+1$. Set $\theta =(k+1,6(4,q^{4}-1)f)$. Then $\frac{%
k+1}{\theta }\mid \left\vert X_{x}\right\vert $ by Lemma \ref{Orbits}.
Furthermore, $\theta =(k+1,3f)$, since $k+1$ is odd. Thus $\frac{k+1}{\theta 
}\mid \frac{q^{6}-1}{q-1}$, since $\frac{k+1}{\theta }$ is coprime to $%
q^{2}(q^{2}-1)^{2}$. Then $q^{3}(q^{2}-1)\frac{q^{2}+1}{(2,q-1)}<\frac{%
q^{6}-1}{q-1}\cdot 3f$, which has no solutions for $q>3$. \\
Finally, assume that $X_{x}=Y$. We treat $X_{x} \cong \Omega _{7}(q)$ and $X_{x} \cong Sp_{6}(q)$ at the same time. Then $ k^{2}=q^{3 }(q^{2}-1)\frac{(q^{2}+1)}{(2,q-1)}$, and since $\frac{(q^{2}+1)}{(2,q-1)}$ is odd, it follows that $q^{2}-1$ is a square. However, this is impossible by \cite{Rib}, A3.1. \\ 
Assume that $X$ is either $\Omega _{7}(q)$, or $Sp_{6}(q)$ for $q$ even, and 
$Y\cong G_{2}(q)$. Assume that $X_{x}\neq Y$. Since $\left\vert
X_{x}\right\vert >q^{7}/\sqrt[3]{2}$ and $\left( \Phi _{6f}^{\ast
}(p),\left\vert X_{x}\right\vert \right) >1$, it follows from \cite{BHRD},
Tables 8.30 and 8.41--8.43, that $X_{x}\cong SU_{3}(q):Z_{\epsilon }$, where 
$\epsilon =1,2$, or $q=3^{2m+1}$, $m\geq 1$, and $X_{x}\cong $ $^{2}G_{2}(q)$%
. Thus, 
\[
k^{2}=\frac{q^{6}(q^{2}-1)^{2}(q^{2}+q+1)(q^{2}+1)}{(2,q-1)\theta }\text{,} 
\]%
where either $\theta =(q+1)\epsilon $, or $\theta =1$ respectively. So, $q^{2}+q+1$
is a square in both cases, but this is impossible by \cite{Rib}, A7.1. Thus $%
X_{x}=Y$ and hence $k^{2}=\frac{q^{3}(q^{4}-1)}{(2,q-1)}$, which, as we have already seen, does not have solutions. 

Assume that $X\cong PSp_{4}(q)$ and $Y\cong Sz(q)$, where $q$ is an odd
power of $2$, $q>2$. Since $\left\vert X_{x}\right\vert >q^{10/3}/\sqrt[3]{2}
$ and $\left( \Phi _{4f}^{\ast }(p),\left\vert X_{x}\right\vert \right) >1$,
it follows that $X_{x}=Y$ by \cite{BHRD}, Table 8.16. Hence $%
k^{2}=q^{2}\left( q-1\right) \left( q+1\right) ^{2}$, but this contradicts 
\cite{Rib}, A3.1, since $q$ is an odd power of $2$ and $q>2$. So far, we
have ruled out Cases (1)--(4).

Assume that Case (5) holds. Then $(X,Y)$ are listed in Table \ref{signprime}%
, in which the symbol $s$ in Column 6, when recorded, denotes the largest
prime divisor of $\left\vert X\right\vert $ such that $s\nmid \Phi _{ef}^{\ast
}(p)$, $s \nmid \left\vert Y\right\vert$ and $s^{2}\nmid \left\vert X\right\vert $. The previous
constraints imply $s\mid \left[ X:X_{x}\right] $ and hence $s^{2}\mid
\left\vert X\right\vert $, which is a contradiction. Hence, only the groups,
whose corresponding cell in Column 6 of Table \ref{signprime} is empty are
admissible, and these are in Lines 2 and 6--7 of the same table. The group $%
X\cong PSL_{4}(2)\cong A_{8}$ is ruled out in Lemma 3.5 of \cite{MF}. If $%
X\cong PSp_{4}(7)$ and $X_{x}\cong A_{7}$, then $7^{3}$ is the highest power
of $7$ dividing $\left[ X:X_{x}\right] $, a contradiction. Finally, if $%
X\cong PSU_{4}(3)$ and $X_{x}\cong A_{7},PSL_{3}(4)$, then $k^{2}=1296$ or $%
162$. The latter is ruled out since it is not a square, whereas the former
yields $k=36$. Then $k+1=37$ must divide the order of $G$ by Lemma \ref%
{PP}(3). However, this is impossible since $G\leq P\Gamma U_{4}(3)$, and the order of $P\Gamma U_{4}(3)$ has not such
a divisor. This completes the proof.
\end{proof}

\begin{lemma}
\label{PSL}If $X\cong PSL_{n}(q)$, then $n=q=3$ and $k=12$.
\end{lemma}

\begin{proof}
Assume that $X\cong PSL_{n}(q)$. Then $n\geq 3$ by Remark \ref{mag3}, and $%
X_{x}$ lies in a large maximal geometric subgroup $Y$ of $X$ by Theorem %
\ref{LergeGeo}. Then $Y$ is one of the groups listed in Proposition 4.7 of \cite{AB}. It follows from Theorem 3.5 of \cite{He1/2} and Proposition
5.2.15.(i) of \cite{KL}, that $Y\notin \mathcal{C}_{1}(X)$ and $Y\notin 
\mathcal{C}_{2}(X)\cup \mathcal{C}_{5}(X)$, respectively, since $\left( \Phi
_{nf}^{\ast }(p),\left\vert Y\right\vert \right) >1$. For the same reason $%
(X,Y)$ is not equal to $(PSL_{4}(5),2^{4}.A_{6})$ or to $(PSL_{3}(4),PSU_{3}(2))$. Thus, one of
the following holds:

\begin{enumerate}
\item[(i).] $Y$ is a $\mathcal{C}_{3}$-subgroup of type $GL_{n/t}(q^{t})$,
where $t=2$, or $t=3$ and either $q\in \left\{ 2,3\right\} $, or $q=5$ and $%
n $ is odd.

\item[(ii).] $Y\in \mathcal{C}_{8}(X)$.
\end{enumerate}

Assume that Case (i) holds. Then $Y\cong SL_{n/t}(q^{t})\cdot
(q^{t}-1)(q-1)^{-1}\cdot t$ by \cite{KL}, Proposition 4.3.6.(II). Also $%
X_{x}=Y$ by Lemma \ref{Nov}(1) and by \cite{KL}, Proposition 4.3.6.(I), for $%
n\geq 13$ and by \cite{BHRD} for $n<13$.

Assume that $t=3$. Note that $\left\vert X\right\vert \leq \left\vert 
\mathrm{Out(}X\mathrm{)}\right\vert ^{2}\cdot \left\vert X_{x}\right\vert
\cdot \left\vert X_{x}\right\vert _{p^{\prime }}^{2}$ by Lemma \ref{Sing}%
(1). Also, since $2f\leq q\,$\ implies $\left\vert \mathrm{Out(}X\mathrm{)}%
\right\vert \leq q(q-1)$, we get $\left\vert X\right\vert \leq
q^{2}(q-1)^{2}\cdot \left\vert X_{x}\right\vert \cdot \left\vert
X_{x}\right\vert _{p^{\prime }}^{2}$. This in turn yields the inequality
(4.11) of \cite{ABD3} (see part (3) in the proof of Proposition 4.2), which
leads to $n/t=q=2$, or to $n=t=3$ and $q=2,3$ or $5$. The former is ruled
out, since $X$ is not isomorphic to $PSL_{6}(2)$ by Lemma \ref{cyclont}, whereas $(n,q)=(3,2)$ is ruled out in Proposition \ref{Due
nonVale}, since $PSL_{3}(2)\cong PSL_{2}(7)$. Finally, in the remaining
cases, $[X:X_{x}]=144$ and $4000$ for $q=3$ or $5$ respectively. Thus, only
the former is admissible, since $[X:X_{x}]$ is a square, and hence $k=12$.

Assume that $t=2$. Then $n$ is even and $X_{x}\cong SL_{n/2}(q^{2})\cdot
(q+1)\cdot 2$. Then $\left\vert X_{x}\right\vert \leq 2(q+1)\frac{q^{n^{2}/2}}{q^{4}-1} \leq 2q^{n^{2}/2-2}$ and $\left\vert X\right\vert \geq q^{n^{2}-2}$ by \cite{AB}, Corollary 4.3.(i), hence $k^{2}=[X:X_{x}]>q^{n^{2}/2}/2$. On the other hand, if $%
n\geq 8$ then $k+1$ divides $(q^{n}-1)(q^{n-2}-1)$ by Lemma \ref{PP}(3) and
by \cite{ABD3}, Lemma 3.6. Therefore, $q^{n^{2}/2}/2<k^{2}<q^{2n-2}$ and
hence $q^{(n-2)^{2}}<4$, which has no solutions for $n \geq 8$.

If $n=6$, then 
\begin{equation}
k^{2}=\frac{\allowbreak q^{9}\left( q-1\right)\left( q^{3}-1\right)
\left( q^{5}-1\right) }{2(6,q-1)}\text{.}  \label{k^2}
\end{equation}%
On the other hand, $k+1\mid \left\vert X_{x}\right\vert \cdot \left\vert 
\mathrm{Out(}X\mathrm{)}\right\vert $ by Lemma \ref{Orbits}, and hence $k+1\mid
(q^{2}+1)(q^{3}+1)(q+1)^{2} \cdot 3f$, with $k+1$ odd. So, we get 
\begin{equation}
k+1\leq (q^{2}+1)(q^{3}+1)(q+1)q\text{.}  \label{k+1}
\end{equation}%
However, there are no solutions in (\ref{k^2}) by comparing it with (\ref{k+1}).

If $n=4$, then 
\[
k^{2}=(q^{2}+q+1)\frac{q^{4}(q-1)^{2}}{2(4,q-1)} 
\]
and hence $q^{2}+q+1$ is a square, but this contradicts \cite%
{Rib}, A7.1.

Assume that Case (ii) holds. Then, by \cite{KL}, Propositions 4.8.3.(II),
48.4.(II) and 4.8.5.(II), one of the following holds:

\begin{enumerate}
\item[(I).] $Y\cong PSp_{n}(q)\cdot [\frac{(q-1,n/2)(q-1,2)}{(q-1,n)}]$ with $n\geq 4$.

\item[(II).] $Y\cong PSU_{n}(q^{1/2})\cdot [\frac{(q^{1/2}+1,n)c}{(q-1,n)}]$, with $c=\frac{q-1}{(q^{1/2}+1,(q-1)/(q-1,n))}$, $n$ odd and $n\geq 3$.

\item[(III).] $Y\cong PSO_{n}^{-}(q)\cdot 2$, with $q$ odd, and $n\geq 4$.
\end{enumerate}

If $X_{x}<Y$, then $G_{x}$ is a type 1 novelty with respect to $Y$ by Lemma %
\ref{Nov}(2). However, this is impossible by \cite{KL}, Tables 3.5.H--I for $%
n\geq 13$ and by \cite{BHRD} for $n<13$. Thus $X_{x}=Y$ in Cases (I)--(III).

Assume that case (I) holds. Then $\left\vert X_{x}\right\vert \leq
4q^{n(n+1)/2}$ and $\left\vert X\right\vert > q^{n^{2}-2}$, by \cite{AB},
Corollary 4.3.(i) and (iii) respectively. Hence $%
k^{2}>q^{(n^{2}-n-4)/2}$. On the other hand, if $n\geq 8$ then $k+1$ divides 
$(q^{n}-1)(q^{n-2}-1)$ by Lemma \ref{PP}(3) and by \cite{ABD3}, Lemma 3.6.
So $q^{(n^{2}-n-4)/2}/4<k^{2}<q^{2n-2}$ and we again reach a contradiction.
Thus, $3\leq n<8$ and hence $n=4$ or $6$ as $n$ is even. If $n=6$, then 
\begin{equation}\label{umoran}
k^{2}=\frac{q^{6}\left( q-1\right) ^{2}\left( q^{2}+q+1\right) \left(
q^{4}+q^{3}+q^{2}+q+1\right) }{(q-1,3)} 
\end{equation}
and hence $q^{4}+q^{3}+q^{2}+q+1$ is a square. So $q=3$ by \cite{Rib}, A8.1. However, it does not fulfill (\ref{umoran}). 

If $n=4$, then $k^{2}=\frac{\allowbreak q^{2}\left( q^{3}-1\right) }{%
(q-1,2)}$. On the other hand, $k+1$ divides $2f \cdot q(q^{2}-1)$ by Lemma \ref{PP}(3), and hence $k+1$ divides $2f \cdot (q+1)$ as $k$ is coprime to $q\frac{q-1}{(2,q-1)}$. So, $k<2f \cdot (q+1)$, which, compared to $k^{2}=\frac{\allowbreak q^{2}\left( q^{3}-1\right) }{%
(q-1,2)}$ yields a contradiction.

Assume that Case (II) holds. Then $\left\vert X_{x}\right\vert < \frac{%
(q^{1/2}+1,n)(q^{1/2}-1)}{(q^{1/2}+1,n)}q^{\frac{n^{2}-1}{2}}<2q^{n^{2}/2}$ and $\left\vert
X\right\vert \geq q^{n^{2}-2}$, by \cite{AB}, Lemma 4.2.(ii) and Corollary
4.3.(i) respectively. Hence $k^{2}>q^{n^{2}/2-2}/2$. On the other hand, if 
$n\geq 5$ then $k+1$ divides $q^{n/2}+1$ by Lemma \ref{PP}(3) and by 
\cite{ABD3}, Lemma 3.6, as $n$ is odd. Thus $q^{n^{2}/2-2}/2<q^{n/2}+1$, which has
no solutions for $n\geq 5$. Therefore, $n=3$ and hence%
\begin{equation}
k^{2}=q^{3/2}(q+q^{1/2}+1)\frac{(q^{1/2}-1,3)(q+1)}{(q-1,3)}%
.  \label{ksuv}
\end{equation}%
So $q+q^{1/2}+1$ is a square, but this contradicts \cite{Rib}, A7.1.

Finally, assume that Case (III) holds. Then 
\begin{equation}
k^{2}=\frac{q^{n^{2}/4}(4,q^{n/2}+1)\left( q^{n/2}-1\right) }{2 \rho(q-1,n)}%
\prod_{j=1}^{n/2-1}(q^{2j+1}-1)\text{,}  \label{kk}
\end{equation}%
with $\rho=1,2$, and hence $q^{n-1}-1$ divides $k^{2}$. Also $\Phi _{(n-1)f}^{\ast }(p)\mid
k^{2}$, where $\Phi _{(n-1)f}^{\ast }(p)>1$, by \cite{Rib}, P1.7, since $%
n\geq 4$ and $n$ is even. Thus $\Phi _{(n-1)f}^{\ast }(p)^{1/2}\mid
\left\vert G_{B}\right\vert $, as $k\mid \left\vert G_{B}\right\vert $.

Note that $\left( \Phi _{(n-1)f}^{\ast }(p),\left\vert X_{x}\right\vert
\right) =1$, as $X_{x}\cong PSO_{n}^{-}(q)\cdot 2$, with $q$ odd and $n\geq 4$. Also, $\left( \Phi _{(n-1)f}^{\ast }(p),\left\vert \mathrm{Out(}X\mathrm{)}%
\right\vert \right) =1$, as $\Phi _{(n-1)f}^{\ast }(p)\equiv 1 \pmod{(n-1)f}$ by \cite{KL}, Proposition 5.2.15.(i). Thus, $\left( \Phi
_{(n-1)f}^{\ast }(p),\left\vert G_{x}\right\vert \right) =1$ and $\left(
\Phi _{(n-1)f}^{\ast }(p),\left[ G_{B}:X_{B}\right] \right) =1$, as both $%
\left[ G_{x}:X_{x}\right] $ and $\left[ G_{B}:X_{B}\right] $ divide $%
\left\vert \mathrm{Out(}X\mathrm{)}\right\vert $. Then $\Phi _{(n-1)f}^{\ast
}(p)^{1/2}\mid \left\vert X_{B}\right\vert $ and $\Phi _{(n-1)f}^{\ast
}(p)^{1/2}\mid \left[ X:X_{B}\right] $, as $\left\vert G_{x}\right\vert
=\lambda (k+1)\left\vert G_{x,B}\right\vert $ and $\left\vert G_{B}\right\vert
=k\left\vert G_{x,B}\right\vert $. Therefore, $\Phi
_{(n-1)f}^{\ast }(p)^{1/2}$ divides both $\left\vert X_{B}^{\symbol{94}%
}\right\vert $ and $\left[ SL_{n}(q):X_{B}^{\symbol{94}}\right] $. Moreover, 
$q$ divides $\left\vert X_{B}^{\symbol{94}}\right\vert $ by Corollary \ref%
{kralj}(i) for $n\geq 4$, since $\left\vert X_{B}\right\vert _{p}=\left\vert
X_{B}^{\symbol{94}}\right\vert _{p}$. Then $X_{B}^{\symbol{94}}$ lies in a
parabolic subgroup of type $P_{1}$ of $SL_{n}(q)$ by Proposition \ref{hdu}%
(2). Hence $X_{B}$ fixes a point of $PG_{n-1}(q)$. Therefore, $\frac{q^{n}-1%
}{q-1}$ divides $\left[ X:X_{B}\right] $ and hence $b$. Thus $\Phi
_{nf}^{\ast }(p)\mid b$. It follows from (\ref{kk}) that, $\Phi _{nf}^{\ast
}(p)$ is coprime to $k$. Thus $\Phi _{nf}^{\ast }(p)\mid k+1$, since $%
b=k(k+1)\lambda $ and $\lambda \mid k$. Then $X_{x}\in \mathcal{C}_{3}(X)$
by Lemma \ref{e=n}, as $X_{x}$ is maximal in $X$. However, this is
impossible, since $X_{x}\in \mathcal{C}_{8}(X)$.
\end{proof}

\bigskip

Before analyzing the case $X\cong PSL_{3}(3)$ and $k=12$, we recall some useful
facts about $PSL_{3}(3):\left\langle \sigma \right\rangle $ and about the action
of some of its subgroups on $PG_{2}(3)$.\\

\bigskip

The group $PSL_{3}(3)$ has two conjugacy classes of subgroups of order $3$,
say $\mathcal{K}_{1}$ and $\mathcal{K}_{2}$ by \cite{At}. If $H_{1}\in \mathcal{K}_{1}$,
then $H_{1}$ is a group of elations of $PG_{2}(3)$ having the same center $C$
and the same axis $\ell $ (e.g. see \cite{HP}). Hence, $H_{1}$ fixes seven
flags of $PG_{2}(3)$. If $H_{2}\in \mathcal{K}_{2}$, then $H_{2}$ normalizes
six suitable Sylow $13$-subgroups of $PSL_{3}(3)$, and fixes a unique flag
of $PG_{2}(3)$. No further points or lines of $PG_{2}(3)$ are fixed by $H_{2}
$. Finally, $\sigma $ is a symmetric polarity of $PG_{2}(3)$. Thus $\sigma $
switches points and lines of $PG_{2}(3)$ and hence preserves both $\mathcal{K%
}_{1}$ and $\mathcal{K}_{2}$. Therefore, $\mathcal{K}_{1}$ and $\mathcal{K}%
_{2}$ are also conjugacy $PSL_{3}(3):\left\langle \sigma \right\rangle $-classes.

\bigskip

\begin{lemma}
\label{L1}The following hold:

\begin{enumerate}
\item If $H_{1}\in \mathcal{K}_{1}$, then $N_{PSL_{3}(3)}(H_{1})\cong
3_{+}^{1+2}:E_{4}$ and $N_{PSL_{3}(3):\left\langle \sigma \right\rangle
}(H_{1})=3_{+}^{1+2}:D_{8}$.

\item If $H_{2}\in \mathcal{K}_{2}$, then $N_{PSL_{3}(3)}(H_{2})\cong
E_{3}\times S_{3}$ and $N_{PSL_{3}(3):\left\langle \sigma \right\rangle
}(H_{2})\cong S_{3}\times S_{3}$.
\end{enumerate}
\end{lemma}

\begin{proof}
Let $H_{1}\in \mathcal{K}_{1}$. Then $H_{1}$ is $(C,\ell )$-elation group of 
$PG_{2}(3)$ and hence $N_{PSL_{3}(3)}(H_{1})$ is the stabilizer in $%
PSL_{3}(3)$ of the flag $(C,\ell )$. Thus $N_{PSL_{3}(3)}(H_{1})\cong
U:E_{4} $, and hence $N_{PSL_{3}(3):\left\langle \sigma \right\rangle
}(H_{1})=U:D_{8}$, where $U$ is a Sylow $3$-subgroup of $PSL_{3}(3)$, since $%
\mathcal{K}_{1}$ and $\mathcal{K}_{2}$, are conjugacy $PSL_{3}(3):\left%
\langle \sigma \right\rangle $-classes. This proves (1).

Let $H_{2}\in \mathcal{K}_{2}$ and let $\mathcal{C}$ be any non-degenerate
conic of $PG_{2}(3)$. It is well known that $PSL_{3}(3)_{\mathcal{C}%
}\cong PGL_{2}(3)\cong S_{4}$ and that $PSL_{3}(3)_{\mathcal{C}}$ does not
contain elations of $PG_{2}(3)$ (e.g. see \cite{HP}, Theorem 2.37). Then any
Sylow $3$-subgroup of $PSL_{3}(3)_{\mathcal{C}}$ lie in $\mathcal{K}_{2}$.
Hence, we may assume that $H_{2}$ is such a group. Then $S_{3}\leq
N_{PGL_{2}(3)}(H_{2})$ and hence $N_{PGL_{2}(3)}(H_{2})\cong E_{3}\times
S_{3}$, since $C_{PSL_{3}(3)}(H_{2})\cong E_{9}$ by \cite{At}. Now, it is
easy to see that $N_{PGL_{2}(3):\left\langle \sigma \right\rangle
}(H_{2})\cong S_{3}\times S_{3}$, which is (2).
\end{proof}

\begin{lemma}
\label{L2}Let $T\leq PSL_{3}(3):\left\langle \sigma \right\rangle $ . Then
the followings hold:

\begin{enumerate}
\item If $T$ contains at least any two elements of $\mathcal{K}_{1}$, and $%
\left\vert T\right\vert \mid 24$, then $T\leq PSL_{3}(3)$ and $T\cong
SL_{2}(3)$.

\item If $T$ contains at least any two elements of $\mathcal{K}_{2}$, and $%
\left\vert T\right\vert =12$ or $24$, then $A_{4}\trianglelefteq T\cap
PSL_{3}(3)$ and $T\cong A_{4}$, $S_{4}$ or $A_{4}\times Z_{2}$.
\end{enumerate}
\end{lemma}

\begin{proof}
Let $H,K$ be any two elements of $\mathcal{K}_{1}$ contained $T$. Then they are
elation groups of $PG_{2}(3)$ with center $C_{i}$ and axis $\ell _{i}$,
where $i=1,2$, respectively. If either $C_{1}=C_{2}$ or $\ell _{1}=\ell _{2}$%
, then $\left\langle H,K\right\rangle $ is elementary abelian of order $9$
by \cite{HP}, Theorem 4.14, but this contradicts $\left\vert T\right\vert
\mid 24$. Therefore, $C_{1}\neq C_{2}$ and $\ell _{1}\neq \ell _{2}$. Then $%
\left\langle H,K\right\rangle $ fixes the point $\ell _{1}\cap \ell _{2}$
and the line $C_{1}C_{2}$. If $\ell _{1}\cap \ell _{2}\in C_{1}C_{2}$, then $%
\ell _{1}\cap \ell _{2}\in \left\{ C_{1},C_{2}\right\} $, as $\ell_{1} \neq \ell_{2}$. Suppose that $\ell
_{1}\cap \ell _{2}=\left\{ C_{1}\right\} $, then $K$ fixes $C_{1}$ and there
is $\gamma \in K$ such that $\ell _{1}^{\gamma }\neq \ell _{1}$ as $%
C_{1}\neq C_{2}$. So $\left\langle H,H^{\gamma }\right\rangle \leq T$, with $%
\left\langle H,H^{\gamma }\right\rangle $ elementary abelian of order $9$
again by \cite{HP}, Theorem 4.14, and we reach a contradiction. The case $\ell
_{1}\cap \ell _{2}=\left\{ C_{2}\right\}$ is ruled out similarly. Thus, $\ell
_{1}\cap \ell _{2}\notin C_{1}C_{2}$. Then $SL_{2}(3)\trianglelefteq
\left\langle H,K\right\rangle \leq GL_{2}(3)$, as the stabilizer in $%
PSL_{3}(3)$ of $\ell _{1}\cap \ell _{2}$ and of $C_{1}C_{2}$ is $GL_{2}(3)$.
Therefore, $T=\left\langle H,K\right\rangle \cong SL_{2}(3)$, as $\left\langle
H,K\right\rangle \leq T$ and $\left\vert T\right\vert \mid 24$. This proves
(1).

Assume that $T$ contains at least two elements $L,M$ of $\mathcal{K}_{2}$.
Then $T_{0}$ contains $L,M$, where $T_{0}=T\cap PSL_{3}(3)$. By \cite{At}, either $%
T_{0}\leq W:GL_{2}(3)$, where $W$ is an elementary abelian group of order $9$
consisting of elations of $PG_{2}(3)$, or $T_{0}\leq S_{4}$ and $T_{0}$
preserves a conic of $PG_{2}(3)$. If $T_{0}\leq W:GL_{2}(3)$,
then $T_{0}\cap W=1$, since the Sylow $3$-subgroups of $T_{0}$ belong to $%
\mathcal{K}_{2}$. Hence, $T_{0}$ is isomorphic to a subgroup of $GL_{2}(3)$.
Since $GL_{2}(3)$ does not contain subgroups of order $12$ with at least two
distinct Sylow $3$-subgroups, and since $\left\vert T\right\vert =12$ or $24$%
, it follows that $T_{0}\cong SL_{2}(3)$. Then $T_{0}$ fixes an anti-flag of $%
PG_{2}(3)$, as $Z(T)\cong Z_{2}$ is an homology of $PG_{2}(3)$. However this case is ruled out, since any element of $\mathcal{K}_{2}$ fixes just a flag and no more
points or lines of $PG_{2}(3)$. Thus, $T_{0}\leq S_{4}$, $T_{0}$ preserves a
conic of $PG_{2}(3)$ and hence $A_{4}\trianglelefteq T_{0}\leq S_{4}$, since 
$\left\langle H,K\right\rangle \leq T_{0}$. Then either $T_{0}=T$ and $%
T\cong A_{4}$, $S_{4}$, or $T_{0} \neq T$, $T_{0}\cong A_{4}$ and $T\cong S_{4}$, $%
A_{4}\times Z_{2}$, since $\left\vert T\right\vert =12$ or $24$. This proves
(2).
\end{proof}

\begin{proposition}
\label{ExmplPSL}If $X\cong PSL_{3}(3)$, then $\mathcal{D}$ is isomorphic to
one of the $2$-designs constructed in Example \ref{Ex5}.
\end{proposition}

\begin{proof}
Assume that $G$ is either isomorphic to $PSL_{3}(3)$ or to $PSL_{3}(3):\left\langle \sigma
\right\rangle $, where $\sigma $ is of order $2$, and that $k=12$. If $(x,B)$ is any flag of $\mathcal{D}$, then either $G_{x} \cong F_{39}$, or $G_{x} \cong F_{78}$, respectively, by \cite{At}.
Moreover $b=12\cdot 13\cdot \lambda $ and the cyclic subgroups of order $3$
lying in $G_{x}$ belong to the conjugacy class $\mathcal{K}_{2}$ defined
above. Hence $\left\vert G_{B}\right\vert =\frac{36c}{\lambda }$, where $c=%
\left[ G:PSL_{3}(3)\right] $. Then $\lambda \mid 3c$, as $12\mid \left\vert
G_{B}\right\vert $, with $\lambda \geq 2$. Thus, $\left( \lambda ,c\right) =(2,2),(3,1)$, $(3,2)$ or $%
(6,2)$.

\begin{enumerate}
\item[(i).] $\left( \lambda ,c\right) \neq (2,2)$.
\end{enumerate}

Assume that $\left( \lambda ,c\right) =(2,2)$. Then $G\cong
PSL_{3}(3):\left\langle \sigma \right\rangle $, $\left\vert G_{B}\right\vert
=36$, and $G_{x,B}\cong Z_{3}$, with $ G_{x,B} \in \mathcal{K}_{2}$. Then the point-$G_{x}$%
-orbits on $\mathcal{D}$ are $\left\{ x\right\} $, one of length $13$, two ones
of length $26$ and two ones of length $39$ by \cite{AtMod}. Then $B$ intersects
each of these orbits in $1,1,2,2,3,3$ points respectively by Lemma \ref{PP}%
(3). Hence, $G_{x,B}$ fixes at least six points on $B$. On the other hand,
Lemma \ref{L1}(2) implies $\left[ N_{G}(G_{x,B}):N_{G_{x}}(G_{x,B})\right]
=6 $. Therefore, $\left\vert \mathrm{Fix(}G_{x,B}\mathrm{)}\right\vert =6$.
Thus, $\mathrm{Fix(}G_{x,B}\mathrm{)}\subset B$ and hence the number of
blocks of $\mathcal{D}$ containing $\mathrm{Fix(}G_{x,B}\mathrm{)}$ is at most $2$, as $%
\lambda =2$. Thus either $r=\left\vert \mathcal{K}_{2}\right\vert $ or $%
r=2\left\vert \mathcal{K}_{2}\right\vert $ and we reach a contradiction, as $%
r=24$ but $\left\vert \mathcal{K}_{2}\right\vert =312$.

\begin{enumerate}
\item[(ii).] If $\left( \lambda ,c\right) =(3,1)$, then $\mathcal{D}$ is
isomorphic to the $2$-design constructed in Example \ref{Ex5}(1).
\end{enumerate}

Assume that $\left( \lambda ,c\right) =(3,1)$. Then $G\cong PSL_{3}(3)$ and $%
G_{B}$ acts regularly on $B$. Let $H$ be a Sylow $3$-subgroup of $G_{B}$. By
Lemma \ref{L1}(1)--(2), either $H\in \mathcal{K}_{1}$ and $G_{B}\cong D_{12}$%
, or $H\in \mathcal{K}_{2}$ and $G_{B}\cong A_{4}$, according to whether $H$
is or is not normal in $G_{B}$ respectively. With the aid of \cite{GAP} we
see that, no flag-transitive $2$-designs occur with $G_{B}\cong D_{12}$, whereas, up to
isomorphism, only one flag-transitive $2$-$(12^{2},12,3)$ design occurs with $G_{B}\cong A_{4}$. Thus, such a $2$-design is
necessarily isomorphic to that constructed in Example \ref{Ex5}(1).

\begin{enumerate}
\item[(iii).] $\left( \lambda ,c\right) \neq (3,2)$.
\end{enumerate}

Assume that $\left( \lambda ,c\right) =(3,2)$. Then $G\cong
PSL_{3}(3):\left\langle \sigma \right\rangle $, $G_{x,B}\cong Z_{2}$ and $%
r=39$. Since $(G^{\prime })_{x}\cong F_{39}$ by \cite{At}, it follows that $%
(G^{\prime })_{x,B}=1$. Therefore, $(G^{\prime })_{x}$ acts transitively on
the set of blocks of $\mathcal{D}$ incident with $x$. Also, $G^{\prime }$
acts point-transitively on $\mathcal{D}$, as $G$ acts point-primitively on $%
\mathcal{D}$. Thus, $G^{\prime }$ acts flag-transitively on $\mathcal{D}$.
Therefore, $\mathcal{D}$ is isomorphic to the $2$-design constructed in Example \ref{Ex5}(1) by (ii), and hence $G^{\prime }$
is its full flag-transitive automorphism group.
However, this is impossible, since $G\cong PSL_{3}(3):\left\langle \sigma
\right\rangle $ acts flag-transitively on $\mathcal{D}$ by our assumption.

\begin{enumerate}
\item[(iv).] If $\left( \lambda ,c\right) =(6,2)$, then $\mathcal{D}$ is
isomorphic to the $2$-design constructed in Example \ref{Ex5}(2).
\end{enumerate}

If $\left( \lambda ,c\right) =(6,2)$, then $G\cong PSL_{3}(3):\left\langle
\sigma \right\rangle $, $r=78$, $G_{x,B}=1$ and $\left\vert G_{B}\right\vert
=12$. By using Lemma \ref{L1}(1)--(2), it is easy to see that, either $G_{B}$
is isomorphic to $Z_{12}$, $D_{12}$ (three conjugacy classes) or to $A_{4}$.
With the aid of \cite{GAP} we see that, no flag-transitive $2$-designs occur with $%
G_{B}\cong Z_{12}$ or $D_{12}$, whereas, up to isomorphism, only one flag-transitive $2$-$%
(12^{2},12,6)$ design occurs with $G_{B} \cong A_{4}$. Thus, such $2$-design is
necessarily isomorphic to that constructed in Example \ref{Ex5}(2).
\end{proof}

\begin{lemma}
\label{PSp}$X$ is not isomorphic to $PSp_{n}(q)$ for $n\geq 4$.
\end{lemma}

\begin{proof}
Assume that $X\cong PSp_{n}(q)$ with $n \geq 4$. Then $X_{x}$ lies in a large maximal
geometric subgroup $Y$ of $X$ by Theorem \ref{LergeGeo}. Hence, $Y$ is
one of the groups listed in \cite{AB}, Proposition 4.22. It follows from
Theorem 3.5 of \cite{He1/2} and Proposition 5.2.15.(i) of \cite{KL}, that $%
Y\notin \mathcal{C}_{1}(X)$ and $Y\notin \mathcal{C}_{2}(X)\cup \mathcal{C}%
_{5}(X)$, respectively, since $\left( \Phi _{nf}^{\ast }(p),\left\vert
Y\right\vert \right) >1$. For the same reason $(X,Y)\neq
(PSp_{8}(3),2^{6}.\Omega _{6}^{-}(2))$, $(PSp_{4}(5),2^{4}.\Omega _{4}^{-}(2))$%
. Thus, one of the following holds:

\begin{enumerate}
\item $Y$ is a $\mathcal{C}_{8}$-subgroup of $X$;

\item $Y$ is a $\mathcal{C}_{3}$-subgroup of $X$ of type $Sp_{n/2}(q^{2})$, $%
Sp_{n/3}(q^{3})$ or $GU_{n/2}(q)$;

\item $(X,Y)$ is either $(PSp_{4}(7),2^{4}.O_{4}^{-}(2))$ or $%
(PSp_{4}(3),2^{4}.\Omega _{4}^{-}(2))$.
\end{enumerate}

Assume that Case (1) holds. Then $Y\cong O_{n}^{-}(q)$ with $q$ even by \cite%
{KL}, Proposition 4.8.6.(II), as $\left( \Phi _{nf}^{\ast
}(p),\left\vert Y\right\vert \right) >1$. Also $X_{x}=Y$ by Lemma \ref{Nov}%
(1) and by \cite{KL}, Proposition 4.8.6.(I). Then $k^{2}=q^{n/2}(q^{n/2}-1)$%
, since $k^{2}=\left[ X:X_{x}\right] $, and hence $q^{n/2}-1$ is a square.
However, this is impossible by \cite{Rib}, A3.1, since $n\geq 4$.

Assume that Case (2) holds. If $Y$ is a $\mathcal{C}_{3}$-subgroup of $X$ of
type $Sp_{2i}(q^{t})$, $t=2,3$, where $n=2it$, then $Y\cong
PSp_{2i}(q^{t})\cdot Z_{t}$ by \cite{KL}, Proposition 4.3.10.(II). Also $%
X_{x}=Y$ by Lemma \ref{Nov}(1) and by \cite{KL}, Proposition 4.3.10.(I). It
follows from Lemma \ref{Sing}(1) that $\left\vert X\right\vert \leq
\left\vert \mathrm{Out(}X\mathrm{)}\right\vert ^{2}\cdot \left\vert
X_{x}\right\vert \cdot \left\vert X_{x}\right\vert _{p^{\prime }}^{2}$,
which in turn yields the first (centered) inequality at page 21 of \cite%
{ABD3} (see part (4) in the proof of Proposition 4.4). Then $i=1$ for $t=3$
and $i\leq 2$ for $t=2$. The case $t=3$ is clear, whereas some remarks
should be done for $t=2$. In \cite{ABD3} the authors assert that the
replication number of the $2$-designs they investigate divides $%
\left\vert \mathrm{Out(}X\mathrm{)}\right\vert (q^{n/t}-1)$ by making use of
their their Lemmas 3.6 and 3.7.(c)--(d). In our context, Lemmas 3.6 of \cite%
{ABD3} still works. Indeed, it is a general group-theoretical result. The role
of Lemma 3.7(c) is replaced by the inequality $\left( r/\lambda \right)
^{2}>v$, since $r=(k+1)\lambda $ and $v=k^{2}$. Finally, the role of Lemma
3.7.(d) is replaced by Lemma \ref{PP}(3). Thus, we get $r/\lambda \mid
\left\vert \mathrm{Out(}X\mathrm{)}\right\vert (q^{4i}-1)$. At this point we
may proceed as in \cite{ABD3} to gain $q^{4i^{2}/8}<k^{2}<\left( r/\lambda
\right) ^{2}<2q^{8i+1}$, as $\left\vert \mathrm{Out(}X\mathrm{)}\right\vert
^{2}\leq 2q$. Therefore $i\leq 2$.

Assume that $t=3$ and $i=1$. Then $X\cong PSp_{6}(q)$ and $X_{x}\cong
PSp_{2}(q^{3})\cdot Z_{3}$. Hence, $k^{2}=q^{3}(q^{2}-1)^{2}(q^{2}+1)/3$.
Since $3$ does not divide $q^{2}+1$, this one is a square. However, this is
impossible by \cite{Rib}, A3.1.

Assume that $t=2$ and $i=2$. Then $X\cong PSp_{8}(q)$ and $X_{x}\cong
PSp_{4}(q^{2})\cdot Z_{2}$. Then $k^{2}=q^{8}\left( q^{2}-1\right)
^{2}\left( q^{4}+q^{2}+1\right) /2$, but this is impossible by \cite{Rib},
A7.1.

Assume that $t=2$ and $i=1$. Then $X\cong PSp_{4}(q)$ and $X_{x}\cong
PSp_{2}(q^{2})\cdot Z_{2}$. Then $k^{2}=q^{2}(q^{2}-1)/2$. Then $q$ is odd
by \cite{Rib}, A3.1. Then $\frac{q\pm 1}{2}%
=z_{1}^{2}$ and $2(q\mp 1)=z_{2}^{2}$ for some positive integers $%
z_{1},z_{2} $, according to whether $q$ is equivalent to $1$ or $3$ modulo $%
4 $ respectively. So $z_{1}^{2}\mp 1=(z_{2}/2)^{2}$, which has no solutions again by \cite{Rib}, A3.1.

Finally, assume that $Y$ is a $\mathcal{C}_{3}$-group of type $GU_{n/2}(q)$.
Then $X_{x}=Y\cong Z_{\frac{q-1}{2}}.PGU_{n/2}(q).Z_{2}$ by Lemma \ref{Nov}(1) and by \cite{KL}%
, Proposition 4.3.7.(I)--(II). Moreover, $n/2$ is odd, since $\left( \Phi
_{nf}^{\ast }(p),\left\vert Y\right\vert \right) >1$, and hence $n\geq 6$.
Then 
\begin{equation}
k^{2}=\frac{1}{(2,q-1)}q^{\frac{1}{8}n\left( n+2\right)
}\prod_{i=1}^{n/2}(q^{i}+(-1)^{i})  \label{kvadrat}
\end{equation}%
and hence $q^{n/2-1}+1$ divides $k^{2}$. Moreover, (\ref{kvadrat})
is not fulfilled for $(n/2-1,q)=(6,2)$. Then $\Phi _{(n-2)f}^{\ast }(p)\mid
k^{2}$, where $\Phi _{(n-2)f}^{\ast }(p)>1$, by \cite{Rib}, P1.7, and hence $%
\Phi _{(n-2)f}^{\ast }(p)^{1/2}\mid \left\vert G_{B}\right\vert $. Moreover, 
$\left( \Phi _{(n-2)f}^{\ast }(p),\left\vert X_{x}\right\vert \right) =1$,
since $n/2$ is odd and $n\geq 6$, and $\left( \Phi _{(n-2)f}^{\ast
}(p),\left\vert \mathrm{Out(}X\mathrm{)}\right\vert \right) =1$, as $\Phi
_{(n-2)f}^{\ast }(p)\equiv 1 \pmod{(n-2)f}$ by \cite{KL}, Proposition
5.2.1.5.(i). Thus $\left( \Phi _{(n-2)f}^{\ast }(p),\left\vert
G_{x}\right\vert \right) =1$ and $\left( \Phi _{(n-2)f}^{\ast }(p),\left[
G_{B}:X_{B}\right] \right) =1$ as both $\left[ G_{x}:X_{x}\right] $ and $%
\left[ G_{B}:X_{B}\right] $ divide $\left\vert \mathrm{Out(}X\mathrm{)}%
\right\vert $. Then $\Phi _{(n-2)f}^{\ast }(p)^{1/2}\mid \left\vert
X_{B}\right\vert $ and $\Phi _{(n-2)f}^{\ast }(p)^{1/2}\mid \left[ X:X_{B}%
\right] $ as $\left\vert G_{x}\right\vert =\lambda (k+1)\left\vert
G_{x,B}\right\vert $ and $\left\vert G_{B}\right\vert = k \left\vert
G_{x,B}\right\vert $. Therefore $\Phi _{(n-2)f}^{\ast }(p)^{1/2}$ divides
both $\left\vert X_{B}^{\symbol{94}}\right\vert $ and $\left[ SL_{n}: X_{B}^{\symbol{94}}\right] $. Moreover, $p$ divides $\left\vert X_{B}^{\symbol{94%
}}\right\vert $ by Corollary \ref{kralj}(ii), as $\left\vert X_{B}^{\symbol{94}%
}\right\vert _{p}=\left\vert X_{B}\right\vert _{p}$ and $n\geq 6$. Then $%
X_{B}$ lies in a parabolic subgroup of type $P_{2}$ or $P_{1,n-1}$ of $%
SL_{n}(q)$ by Proposition \ref{hdu}(3).

Assume that $X_{B}^{\symbol{94}}$ lies in a parabolic subgroup of type $%
P_{2} $ of $SL_{n}(q)$. Then $X_{B}$ fixes a line $\ell $ of $PG_{n-1}(q)$,
and hence $\ell $ is non-degenerate by \cite{KL}, Propositions 4.1.3.(II)
and 4.1.19.(II), since $\Phi _{(n-2)f}^{\ast }(p)^{1/2}\mid \left\vert
X_{B}^{\symbol{94}}\right\vert $. Also $X_{B}\leq SL_{2}(q)\circ Sp_{n-2}(q)$. Since $\Phi
_{(n-2)f}^{\ast }(p)^{1/2}$ divides $\left\vert \left( X_{B}\right) _{\ell
^{\perp }}^{\ell ^{\perp }}\right\vert $ and $\left[ SL_{n-2}(q):\left(
X_{B}\right) _{\ell ^{\perp }}^{\ell ^{\perp }}\right] $, where $\left( X_{B}\right) _{\ell
^{\perp }}^{\ell ^{\perp }}$ denotes the group induced by $\left( X_{B}\right) _{\ell
^{\perp }}$ on $\ell^{\perp }$, it follows that
either $n-2=4$ and $q=7$, or $\left( X_{B}\right) _{\ell ^{\perp }}^{\ell
^{\perp }}\leq GL_{1}(q^{n-2})\cdot Z_{n-2}$ and $p\mid n-2$ by Proposition %
\ref{hdu}(1). The former does not fulfill (\ref{kvadrat}), the latter
implies $\left\vert X_{B}\right\vert _{p}\leq q\cdot (n-2)_{p}$. Then $q^{%
\frac{n^{2}-12}{12}}\leq q(n-2)$ by Corollary \ref{kralj}(ii), which yields $n=6$ and $q \leq 5$, as $n\geq 6$. However, none of these values fulfills (\ref{kvadrat}).

Assume that $X_{B}$ lies in a parabolic subgroup of type $P_{1,n-1}$ of $%
SL_{n}(q)$. Then $X_{B}$ is a subgroup of $X$ fixing a point $\left\langle
w\right\rangle $ of $PG_{n-1}(q)$. Then $X_{B}\leq \lbrack
q^{n-1}]:Z_{q-1}\circ Sp_{n-2}(q)$ by \cite{KL}, Proposition 4.1.19.(II). A
similar argument to that used above shows that $\Phi _{(n-2)f}^{\ast
}(p)^{1/2}$ divides both $\left\vert \left( X_{B}\right) _{\left\langle
w\right\rangle ^{\perp }/\left\langle w\right\rangle }^{\left\langle
w\right\rangle ^{\perp }/\left\langle w\right\rangle }\right\vert $ and $%
\left[ Sp_{n-2}(q):\left( X_{B}\right) _{\left\langle w\right\rangle ^{\perp
}/\left\langle w\right\rangle }^{\left\langle w\right\rangle ^{\perp
}/\left\langle w\right\rangle }\right] $. Thus $\left( X_{B}\right)
_{\left\langle w\right\rangle ^{\perp }/\left\langle w\right\rangle
}^{\left\langle w\right\rangle ^{\perp }/\left\langle w\right\rangle }\leq
GL_{1}(q^{n-2})\cdot Z_{n-2}$ and hence $\left\vert X_{B}\right\vert
_{p}\leq q^{n-1}\cdot (n-2)_{p}$. Then $q^{\frac{n^{2}-12}{12}}\leq
q^{n-1}(n-2)$ by Corollary \ref{kralj}(ii). Thus either $n=6$ or $n=10$, as $n/2$ is odd. 
If $n=6$, then (\ref{kvadrat}) becomes 
\[
k^{2}=\frac{q^{6}\left( q^{2}+1\right) \left( q-1\right) ^{2}}{(2,q-1)}%
\left( q^{2}+q+1\right),
\]%
and we reach a contradiction by \cite{Rib}, A7.1.\\
If $n=10$, then
\[
k^{2}=\frac{q^{15}\left( q-1\right) \left( q^{2}+1\right) \left( q^{3}-1\right) \left( q^{4}+1\right) \left( q^{5}-1\right)}{(2,q-1)},
\]
and we obtain a contradiction, as $k+1$ divides $(q+1)^{3}(q^{3}+1)(2,q-1)f$ by Lemma \ref{PP}(3).\\
Finally, assume that case (3) holds. If $%
(X,Y)=(PSp_{4}(7),2^{4}.O_{4}^{-}(2))$, then $\Phi _{4}^{\ast }(7)=25$ is
the highest power of $5$ dividing $\left\vert X\right\vert $. Since $5\mid
\left\vert X_{x}\right\vert $ and $\left[ X:X_{x}\right] $ is a square, it
follows that $5^{2}$ divides $\left\vert X_{x}\right\vert $ and hence $%
\left\vert Y\right\vert $, a contradiction.

If $(X,Y)=(PSp_{4}(3),2^{4}.\Omega _{4}^{-}(2))$, then $5\mid \left\vert
X_{x}\right\vert $ and hence $k\mid 2^{3}\cdot 3^{2}$. Then $k=3,4,8,9$,
since $k+1\mid \left\vert G\right\vert $, with $G\leq PSp_{4}(3).Z_{2}$, by Lemma %
\ref{PP}(3). However, $X$ does not have $k^{2}$ as a transitive
permutation representation degree for any such values of $k$
by \cite{At}.
\end{proof}

\begin{lemma}
\label{PSU}$X$ is not isomorphic to $PSU_{n}(q)$ for $n\geq 3$.
\end{lemma}

\begin{proof}
Assume that $X\cong PSU_{n}(q^{1/2})$, with $n \geq 3$. Then $X_{x}$ lies in a large maximal
geometric subgroup $Y$ of $X$ by Theorem \ref{LergeGeo}, and hence $Y$ is one of
the groups listed in \cite{AB}, Proposition 4.17. Also $\left(
\Phi _{ef}^{\ast }(p),\left\vert X_{x}\right\vert \right) >1$, and hence $%
\left( \Phi _{ef}^{\ast }(p),\left\vert Y\right\vert \right) >1$, where $e=n$
or $n-1$ according as $n$ is odd or even respectively. 

Assume that $n$ is odd. Filtering the list of subgroups in \cite{AB},
Proposition 4.17, with respect to the property of $\left( \Phi _{nf}^{\ast
}(p),\left\vert Y\right\vert \right) >1$, we see that the unique admissible
case is when $Y$ is a $\mathcal{C}_{3}$-subgroup of $X$ of type $GU_{n/3}(3^{3})$. Then $Y\cong
Z_{7}.PSU_{n/3}(3^{3}).Z_{(n/3,7)}.Z_{3}$ by \cite{KL}, Proposition 4.3.6(II). It
follows from Lemma \ref{Sing}(1) that $\left\vert X\right\vert \leq
\left\vert \mathrm{Out(}X\mathrm{)}\right\vert ^{2}\cdot \left\vert
Y\right\vert \cdot \left\vert Y\right\vert _{p^{\prime }}^{2}$ which yields
inequality (6) in the proof of Proposition 4.3 of \cite{ABD1} for $t=q=3$%
. Thus $(n/t,t)=(1,3)$, and hence $k^{2}=\allowbreak 288$, which is a
contradiction.

Assume that $n$ is even. Recall that $n\geq 4$ by Remark \ref{mag3}. The
constraint $\left( \Phi _{(n-1)f}^{\ast }(p),\left\vert Y\right\vert \right)
>1$ reduces the list of the subgroups in \cite{AB}, Proposition 4.17 to the
case $Y\in \mathcal{C}_{1}(X)$. Then $Y$ is the stabilizer in $X$ of a non-isotropic point of $PG_{n-1}(q)$ by \cite{KL}%
, Propositions 4.1.4.(II) and 4.1.18.(II). Moreover, $X_{x}=Y$ by Lemma \ref{Nov}%
(1) and by \cite{KL}, Proposition 4.1.4.(I). So, $k^{2}=q^{(n-1)/2}\frac{%
q^{n/2}-1}{(q^{1/2}+1)}$, which is exactly (\ref{kq}) in Lemma \ref{e=n-1}, with $s=q^{1/2}$. Hence, the same argument used in Lemma \ref{e=n-1} rules out this case.
\end{proof}

\begin{lemma}
\label{POmegadisp}$X$ is not isomorphic to $P\Omega _{n}(q)$, with $nq$
odd, for $n\geq 5$.
\end{lemma}

\begin{proof}
Assume that $X\cong P\Omega _{n}(q)$ with $nq$ is odd and $n\geq 5$. The group $X_{x}$ lies in a large maximal geometric
subgroup $Y$ of $X$ such that $\left( \Phi _{(n-1)f}^{\ast
}(p),\left\vert Y\right\vert \right) >1$ by Theorem \ref{LergeGeo}. Then $Y$
is one of the groups listed in \cite{AB}, Proposition 4.23, and hence $Y\in 
\mathcal{C}_{1}(X)$, since $\left( \Phi _{(n-1)f}^{\ast }(p),\left\vert
Y\right\vert \right) >1$. Thus either $X \cong P \Omega_{7}(q)$, with $q=3,5$, and $Y$ of type $O^{\varepsilon^{\prime}}_{1}(q) \wr S_{7}$, or $Y \in \mathcal{C}_{1}(X)$. The former is ruled out since $13 \mid \left\vert X \right\vert$, but $13^{2} \nmid \left\vert X \right\vert$, and $13 \nmid \left\vert X_{x} \right\vert$. Thus, $Y \in \mathcal{C}_{1}(X)$ and hence $Y\cong \Omega _{n-1}^{-}(q).Z_{2}$ is
the stabilizer of a non-singular point of $PG_{n-1}(q)$ by \cite{KL},
Propositions 4.1.6.(II) and 4.1.20.(II). Moreover, $X_{x}=Y$ by Lemma \ref%
{Nov}(1) and by \cite{KL}, Proposition 4.1.6.(I). Then $k^{2}=q^{(n-1)/2}%
\frac{q^{(n-1)/2}-1}{2}$, as $k^{2}=\left[ X:X_{x}\right] $ and $q$ is odd.
Hence, 
\begin{equation}
k^{2}-1=\frac{(q^{(n-1)/2}+1)\left( q^{(n-1)/2}-2\right) }{2}\text{.}
\label{mu}
\end{equation}%
On the other hand, $X$ has subdegrees $q^{(n-1)}-1$, $q^{(n-3)/2}\frac{%
(q^{(n-1)/2}+1)}{2}$ and $(q-3)/2$ times $q^{(n-3)/2}(q^{(n-1)/2}+1)$ (e.g.
see \cite{Saxl}, p. 331). Thus $q^{(n-3)/2}\frac{(q^{(n-1)/2}+1)}{2}$ is a
subdegree for $G$ too, and hence $k+1\mid \frac{(q^{(n-1)/2}+1)}{2}$ by
Lemma \ref{PP}(3). So, $q^{(n-1)/2}-2\mid k-1$ by (\ref{mu}), and hence $q=2$%
, $n=5$ and $k^{2}=6$, which is a contradiction.
\end{proof}

\bigskip

We recall some useful facts about the orthogonal group in order to tackle the remaining cases $X\cong P\Omega
_{n}^{\varepsilon }(q)$, where $\varepsilon =\pm $.

Let $\mathbb{L}=GF(q^{2})$, $\mathbb{K}=GF(q)$, $q$ odd, let $\zeta $ be generator
of $\mathbb{L}^{\ast }$ and let $\omega =\zeta ^{\frac{q+1}{2}}$ and $%
\epsilon =\zeta ^{\frac{q-1}{2}}$. Then $\omega ^{2}$ is a generator of 
$\mathbb{K}^{\ast }$, $\omega ^{q}=-\omega $ and $\epsilon \zeta =\omega $%
, and $\left\{ 1,\omega \right\} $ is a basis of $\mathbb{L}$ over $\mathbb{K}
$. If $\theta \in \mathbb{L}$, then $\theta =\alpha +\beta \omega $, with $%
\alpha ,\beta \in \mathbb{K}$. If we fix a basis $\left\{
e_{1},...,e_{n/2}\right\} $ of the $\mathbb{L}$-vector space $\mathbb{L}%
^{n/2}$, where $n$ is an even integer, we may consider the $\mathbb{K}$-isomorphism%
\[
\Phi :\mathbb{L}^{n/2}\longrightarrow \mathbb{K}^{n},(\theta _{1},...,\theta
_{n})\longmapsto (\alpha _{1},\beta _{1},...,\alpha _{n},\beta _{n})\text{,} 
\]%
where $\theta _{i}=\alpha _{i}+\beta _{i}\omega $ for each $i=1,...,n$. We
will denote both $y$ and $\Phi (y)$ simply by $y$. It will be clear from the
context if we are regarding $y$ as $n/2$-dimensional $\mathbb{L}$-vector, or as
a $n$-dimensional $\mathbb{K}$-vector. Hence, $\Phi (\left\langle
y\right\rangle _{\mathbb{L}})=\left\langle y,\omega y\right\rangle _{\mathbb{%
K}}$ is a $2$-dimensional subspace of $\mathbb{K}^{n}$.

A set of lines of $PG_{n-1}(q)$ partitioning the point set of $PG_{n-1}(q)$ is
said \emph{spread} of $PG_{n-1}(q)$. The set $\mathcal{S}$ of $1$%
-dimensional $\mathbb{L}$-subspaces of $\mathbb{L}^{n/2}$ is mapped by $\Phi $
onto a spread of $PG_{n-1}(q)$ by \cite{CK}, Proposition 2.1. As for the
vectors, we denote both $\mathcal{S}$ and its image under $\Phi $ simply by $%
\mathcal{S}$.\\
If $Q_{\#}$ is quadratic form on $%
\mathbb{L}^{n/2}$ polarized by the symmetric bilinear form $F_{\#}$, and $T:%
\mathbb{L}\longrightarrow \mathbb{K},t\longmapsto t+t^{q}$ is the trace form
from $\mathbb{L}$ onto $\mathbb{K}$, it follows that $Q=T\circ Q_{\#}$ is a
quadratic form on $\mathbb{K}^{n}$ polarized by the symmetric bilinear
form $F=T\circ F_{\#}$. As pointed out in \cite{CK}, $ \Phi $ maps singular $1$-dimensional subspaces
of $\mathbb{L}^{n/2}$ onto totally singular $2$-dimensional subspaces of $%
\mathbb{K}^{n}$, non-degenerate $1$-dimensional subspaces of $\mathbb{L}%
^{n/2} $ onto non-degenerate $2$-dimensional subspaces of $\mathbb{K}^{n}$
and provides an embedding of $O_{n/2}^{\varepsilon^{\prime}}(q^{2})$ in $O_{n}^{\varepsilon}(q)$. The approach in \cite{CK} is geometric, its group-theoretical counterpart is contained in \cite{KL}, Propositions 4.3.14, 4.3.16 and 4.3.20. \\ We are actually interested in $\left(\varepsilon,\varepsilon^{\prime} \right) =\left( -,-\right)$ for $n/2$ even. Note that, $O_{n/2}^{-}(q^{2}).Z_{2}$ preserves $\mathcal{S}$ by \cite{CK}, Proposition 2.4. Then $X_{\mathcal{S}} \cong P\Omega _{n/2}^{-}(q^{2}).Z_{2}$ by \cite{CK}, Main Theorem (b), and by \cite{KL}, Proposition 4.3.16.(II). We use these information to prove the following lemma, which will play a central role in excluding some of the cases where $X_{x}$ is a $\mathcal{C}_{3}$-subgroup of $X$.

\bigskip

\begin{lemma}\label{2in1}
If $X\cong P\Omega_{n}^{\varepsilon}(q)$, where $\varepsilon =\pm$, then there is a collineation $\alpha \in X-X_{\mathcal{S}}$ such that $\Omega^{-} _{n/2-1}(q^{2})\leq X_{\mathcal{S},\mathcal{S}^{\alpha }}$.
\end{lemma}

\begin{proof}
Let $\ell \in \mathcal{S}$ corresponding to a non-degenerate $2$-dimensional
subspace $U$ of $\mathbb{K}^{n}$ of type $-$. Such a
subspace does exists. Indeed, $U$ is either $\left\langle z,\omega z\right\rangle
_{\mathbb{K}}$ or $\left\langle z,\epsilon z\right\rangle _{\mathbb{K}}$
according to whether $q$ is equivalent to $1$ or $3$ modulo $4$ by \cite{CK}%
, Proposition 2.3.(c). Moreover, $D(Q)=\boxtimes $ by \cite{KL}, Proposition
2.5.10.(ii), as $n/2$ is even. Then $X_{\ell }\cong \left( \Omega _{2}^{-}(q)\times \Omega _{n-2}^{+}(q)\right) .[4]$ by \cite{KL}, Proposition 4.1.6.(II). If $\ell ^{\perp }$ denotes the $(n-3)$-dimensional subspace of $PG_{n-1}(q)$ corresponding to $U^{\perp }$, we see that $X(\ell ^{\perp })$, the pointwise stabilizer in $X$ of $\ell ^{\perp }$, contains $\Omega _{2}^{-}(q)$.

The subspace $U$ is the image under $\Phi$ of a non-degenerate $1$-dimensional subspace of $\mathbb{L%
}^{n/2}$. Hence, by \cite{KL}, Proposition 4.1.6.(II), and by \cite{CK}, Proposition 2.3.(c), we see that 
\begin{equation*}
X_{\mathcal{S},\ell }\cong \left\lbrace
\begin{split}  
& \Omega _{n/2-1}(q^{2})\hphantom{.2}\hphantom{oho} \text{ for } q(n-1)/8 \text{ odd} \\
& \Omega _{n/2-1}(q^{2}).2\hphantom{oho} \text{ for } q(n-1)/8 \text{ even.} \\
\end{split}
\right.
\end{equation*}

Therefore $ X_{\mathcal{S}}(\ell ^{\perp })=1$. Thus $ X(\ell ^{\perp }) \neq%
X_{\mathcal{S}}(\ell ^{\perp })$, and hence there is $\alpha \in X(\ell ^{\perp })-X_{\mathcal{S}}(\ell ^{\perp })$ such that $\Omega^{-} _{n/2-1}(q^{2})\leq X_{\mathcal{S},\mathcal{S}^{\alpha }}$.
\end{proof}

\bigskip
Let $\mathbb{L}$, $\mathbb{K}$ and $\zeta $ as above, but this time $q$ can also be even. Then $\left\{ 1,\zeta \right\} $ is a basis of $\mathbb{L}$ over $\mathbb{K}
$. If $\theta \in \mathbb{L}$, then $\theta =\alpha +\beta \zeta $, with $%
\alpha ,\beta \in \mathbb{K}$. If we fix a basis $\left\{
e_{1},...,e_{n/2}\right\} $ of the $\mathbb{L}$-vector space $\mathbb{L}%
^{n/2}$, where $n$ is an even integer, we may consider the $\mathbb{K}$-isomorphism%
\[
\Psi :\mathbb{L}^{n/2}\longrightarrow \mathbb{K}^{n},(\theta _{1},...,\theta
_{n})\longmapsto (\alpha _{1},\beta _{1},...,\alpha _{n},\beta _{n})\text{,} 
\]%
where $\theta _{i}=\alpha _{i}+\beta _{i}\zeta $ for each $i=1,...,n$. As done for the map $\Phi$, we
will denote both $y$ and $\Psi (y)$ simply by $y$, but differently from $\Phi$, the map $\Psi$ works also for $q$ even.\\
Let $\varphi$ be a Hermitian form on $\mathbb{L}^{n/2}$ and let $Q$ be a quadratic form on $%
\mathbb{K}^{n}$ defined by $Q(z)=\varphi (z,z)$. Then $\Psi $ maps totally isotropic 
subspaces of $\mathbb{L}^{n/2}$ onto totally singular subspaces of $\mathbb{K}^{n}$, and non-degenerate   subspaces of $\mathbb{L}^{n/2}$ onto non-degenerate subspaces of $\mathbb{K}^{n}$ by \cite{Dye0}, Lemma 1.(iii), and provides an embedding of $U_{n/2}(q^{2})$ in $O_{n}^{\varepsilon }(q)$.

The group $X$ preserves a non-degenerate quadric $\mathcal{Q}$ of $PG_{n-1}(q)$, which
is either elliptic or hyperbolic according to whether $\varepsilon =-$ or $+$
respectively. A partition of $\mathcal{Q}$ by a set of skew, singular lines of $%
PG_{n-1}(q)$ is said \emph{spread of lines} of $\mathcal{Q}$. The isotropic $1$-dimensional subspaces of $\mathbb{L}^{n/2}$ are the points of an hermitian variety $\mathcal{H}$ of $PG_{n/2-1}(q^{2})$ and are mapped by $\Psi $ onto a spread of lines of $\mathcal{Q}$ by \cite{Dye0}, Theorem 1. We will denote both $\mathcal{H}$, and its image under $\Psi $, simply, by $%
\mathcal{H}$. Then 
\begin{equation}\label{moretuge}
X_{\mathcal{H}}\cong \left[\frac{q+1}{\left( q+1,3-\varepsilon \right) }\right]%
.PSU_{n/2}(q).\left[ \mu \left( q+1,n/2\right) \right] \text{,} 
\end{equation}
where $\mu =1$ or $(n/2,2)$ according as $q$ is odd or even, respectively,
by \cite{Dye0}, Theorems 2 and 3, and by \cite{KL}, Proposition 4.3.18.

\begin{lemma}
\label{2in1Unitary}If $X\cong P\Omega
_{n}^{\pm }(q)$, then there is a collineation $\beta \in X-X_{\mathcal{H}}$ such that $$\left[
X_{\mathcal{H}}:X_{\mathcal{H},\mathcal{H}^{\beta }}\right] \mid
2(q+1,n/2)(q+1,n/2-2)q^{n-3}(q^{n/2}-(-1)^{n/2})(q^{n/2-1}-(-1)^{n/2-1}).$$
\end{lemma}

\begin{proof}
Let $ w_{1},w_{2}$ be two orthonormal vectors of $\mathbb{L}^{n/2}$ with respect $\varphi$, let $W=\langle w_{1},w_{2}\rangle$ and let $\ell$ be the corresponding line of $PG_{n/2-1}(q^{2})$. Since $Q(x)=\varphi(x,x)$, it is clear that $\langle w_{i}\rangle$, $i=1,2$, is a non-degenerate $2$-dimensional subspace of type $-$ of $\mathbb{K}^{n}$. Then $W$ is a non-degenerate $4$-dimensional subspace of type $+$ of $\mathbb{K}^{n}$ by Proposition 2.5.11.(ii).
If $\ell^{\perp }$ denotes the $n-3$ subspace of $PG_{n-1}(q)$
corresponding to $W^{\perp }$, then $X \left( \ell^{\perp } \right)$ contains a
subgroup isomorphic to $\Omega _{4}^{+}(q)$ inducing on $\ell^{\perp }$ either the identity, or a central involution, by \cite{KL},
Proposition 4.1.6.(II). Thus $ \left\vert X \left( \ell^{\perp } \right)_{p} \right\vert=q^{2}$. On the
other hand, $\left\vert X_{\mathcal{H}}(\ell^{\perp }) \right\vert_{p}=q$ by \cite{KL}, Proposition 4.1.4.(II). Therefore, $X(\ell^{\perp }) \neq X_{\mathcal{H}}(\ell^{\perp })$, and hence there is $\beta \in X(\ell^{\perp }) - X_{\mathcal{H}}(\ell^{\perp })$ such that $X_{\mathcal{H}}\left( \ell \right)\leq X_{\mathcal{H},%
\mathcal{H}^{\beta }}$ by \cite{KL}, Proposition 4.1.4.(II). Then $\left[ X_{\mathcal{H}}:X_{\mathcal{H}}(\ell )\right] $ divides $2(q+1,n/2)(q+1,n/2-2)q^{n-3}(q^{n/2}-(-1)^{n/2})(q^{n/2-1}-(-1)^{n/2-1})2$ again by \cite{KL}, Proposition 4.1.4(II). So does $\left[
X_{\mathcal{H}}:X_{\mathcal{H},\mathcal{H}^{\beta }}\right] $, as $X_{\mathcal{H}}\left( \ell \right)\leq X_{\mathcal{H},%
\mathcal{H}^{\beta }}$,  which is the assertion.
\end{proof}

\begin{lemma}
\label{POmegaMeno}$X$ is not isomorphic to $P\Omega _{n}^{-}(q)$ for $%
n\geq 6$.
\end{lemma}

\begin{proof}
Assume that $X\cong P\Omega _{n}^{-}(q)$ with $n \geq 6$. Actually, $P\Omega _{6}^{-}(q)$ cannot occur by Lemma \ref{PSU}, since $P\Omega _{6}^{-}(q)\cong PSU_{4}(q)$. Thus $n \geq 8$. The group $X_{x}$ lies in a large maximal
geometric subgroup $Y$ of $X$ by Theorem \ref{LergeGeo}. Then either $Y\in 
\mathcal{C}_{1}(X)$, or $Y$ is a $\mathcal{C}_{3}$-subgroup of $X$ either of type $%
O_{n/2}^{-}(q^{2})$ or $GU_{n/2}(q)$ according as $n/2$ is even or
odd, respectively, by \cite{AB}, Proposition 4.23, since $\left( \Phi
_{nf}^{\ast }(p),\left\vert Y\right\vert \right) >1$. Actually, the case $%
Y\in \mathcal{C}_{1}(X)$ is ruled out by Theorem 3.5 of \cite{He1/2}. Thus,
either $Y\cong P\Omega _{n/2}^{-}(q^{2}).Z_{2}$ or $Y\cong \frac{q+1}{\left(
q+1,4\right) }.PSU_{n/2}(q).\left[ \left( q+1,n/2\right) \right] $ by \cite%
{KL}, Propositions 4.3.16.(II) and 4.3.18.(II). Moreover, $X_{x}=Y$ by Lemma %
\ref{Nov}(1) and by \cite{KL}, Propositions 4.3.16.(I) and 4.3.18.(I).

Assume that $X_{x}\cong P\Omega _{n/2}^{-}(q^{2}).Z_{2}$. Then 
\begin{equation}
k^{2}=\frac{1}{2}q^{n^{2}/8}\prod_{i=0}^{n/4-1}(q^{4i+2}-1)\text{.}
\label{pentagon}
\end{equation}%
Then $n^{2}/8$ is even, as $n/2$ is even. If $q$ is even, then the
highest power of $2$ dividing $k^{2}$ is $2^{f\frac{n^{2}}{8}-1}$, but this is a
contradiction as $fn^{2}/8-1$ is odd. Therefore, $q$ is odd and hence (\ref{pentagon}) becomes%
\begin{equation}
k^{2}=q^{n^{2}/8}\frac{(q^{2}-1)^{n/4}}{2}\prod_{i=1}^{n/4-1}\left( \frac{%
q^{4i+2}-1}{q^{2}-1}\right) \text{.}  \label{sesamo}
\end{equation}%
Moreover, the actions
of $X$ on a suitable $X$-orbit of spreads of $PG_{n-1}(q)$ and on the point set of $\mathcal{D}$
are equivalent (see above). Thus we may identify the two sets.\\
Suppose that $n/4$ is even and let $2^{j}$ be the highest power of $2$
dividing $q^{2}-1$. Then $2^{jn/4-1}$ must be a square, as $q^{n^{2}/8}\prod_{i=1}^{n/4-1}%
\left( \frac{q^{4i+2}-1}{q^{2}-1}\right) $ is odd, whereas it is not, as $jn/2-1$ is odd, and so we reach a contradiction. Thus $n/4$ is odd and hence $n\geq 12$, as $n \geq 8$.

Since there is $\alpha \in X-X_{x}$ such that $\Omega _{n/2-1}(q^{2})\leq
X_{x,x^{\alpha }}$ by Lemma \ref{2in1}, it follows that $%
\Omega _{n/2-1}(q^{2})\leq G_{x,x^{\alpha }}$. Thus $k+1 \mid 8q^{n/2-2}(q^{n/2}+1)f$, as $k+1$ divides $\left[G_{x}:G_{x,x^{\alpha }}\right] $ by Lemma \ref{PP}(3), and hence $k+1\mid 8(q^{n/2}+1)f$, as $p$ divides $k$. So, 
\[
q^{n^{2}/8}(q^{2}-1)^{n/4}<k<64(q^{n/2}+1)^{2}f^{2}<16(q^{n/2}+1)^{2}q^{2}%
\text{,} 
\]%
which has no solutions for $n\geq 12$.

Assume that $X_{x}\cong \frac{q+1}{\left( q+1,4\right) }.PSU_{n/2}(q).\left[
\left( q+1,n/2\right) \right] $. Let $\mathcal{Q}$ be the elliptic quadric
of $PG_{n-1}(q)$ preserved by $X$. Then the actions of $X$ on the sets of the spreads of lines
of $\mathcal{Q}$ and on the point set of $\mathcal{D}$ are equivalent. Therefore, we may identify the points of $\mathcal{D}$ with the spreads of lines of the elliptic quadric $\mathcal{Q}$. Hence 
\begin{equation}\label{kappa}
k^{2}=q^{n(n-2)/8}\prod_{i=2}^{n/2-1}\left( q^{i}+\left( -1\right)
^{i}\right) \text{,}
\end{equation}%
where $n\geq 10$, as $n/2$ is odd and $n\geq 8$.\\
There is $\beta \in X-X_{x}$ such that $[X_{x}:X_{x,x^{\beta }}]$ divides $2(q+1,n/2)(q+1,n/2-2)q^{n-3}(q^{n/2}+1)(q^{n/2-1}-1)$ by Lemma \ref{2in1Unitary}, since $n/2$ is odd. Then
\begin{equation}\label{danas}
k+1 \mid f(q+1,n/2)(q+1,n/2-2)(q^{n/2}+1)(q^{n/2-1}-1) 
\end{equation}
by Lemma \ref{PP}(3), since $k$ is even and $p$ divides $k$. Now, comparing (\ref{kappa}) and (\ref{danas}), and using \cite{AB}, Lemma 4.1(i), we obtain
\begin{equation*}
q^{n(n-2)/4-1 }<k^{2}<q^{2}(q+1)^{4}(q^{n/2}+1)^{2}(q^{n/2-1}-1)^{2}<64q^{2n+4}
\end{equation*}
and hence $n=10$, as $n/2$ is odd. Then (\ref{kappa}) becomes
\begin{equation} \label{dosta}
k^{2}=q^{10}(q^{2}+1)(q^{3}-1)(q^{4}+1)
\end{equation}
which, together with (\ref{danas}), yields $k+1 \mid f(q+1,15)(q^{5}+1)(q+1)$. Then $k^{2}<225q^{2}(q+1)^{2}(q^{5}+1)^{2}$ which compared to \ref{dosta} yields $q=2,3$. However, these values of $q$ do not fulfill (\ref{dosta}), and hence this case is ruled out. This completes the
proof.
\end{proof}

\begin{lemma}
\label{POmegaPiu}$X$ is not isomorphic to $P\Omega _{n}^{+}(q)$ for $n \geq 4$.
\end{lemma}

\begin{proof}
Assume that $X\cong P\Omega _{n}^{+}(q)$ with $n \geq 4$. Actually, $n\geq 8$ by Remark \ref{mag3} and since $P\Omega _{6}^{+}(q)\cong PSL_{4}(q)$ is ruled out in Lemma \ref{PSL}. The group $X_{x}$ lies in a large maximal geometric subgroup $Y$ of $X$ by
Theorem \ref{LergeGeo}. Then $Y$ is one of the the groups listed in \cite{AB}%
, Proposition 4.23. Filtering such a list with respect to the constraint $%
\left( \Phi _{(n-2)f}^{\ast }(p),\left\vert Y\right\vert \right) >1$, we see
that either $Y\in \mathcal{C}_{1}(X)$, or $Y$ is a $\mathcal{C}_{3}$-group
of type $O_{n/2}(q^2)$ or $GU_{n/2}(q)$ according to whether $n/2$ is odd or even, respectively, or $(X,Y)=(P\Omega
_{8}^{+}(3),O_{1}(3)\wr S_{8}),(P\Omega
_{8}^{+}(3),2^{6}.\Omega _{6}^{+}(2))$. The latter is immediately ruled out.
Indeed, $13\mid \left\vert X\right\vert $ but $13^{2}\nmid \left\vert
X\right\vert $ and $13 \nmid \left\vert Y\right\vert $.

Assume that $Y$ is a maximal $\mathcal{C}_{3}$-subgroup of $X$ of type $%
O_{n/2}(q^2)$. Then $X$ has one or two conjugacy classes of subgroups
isomorphic to $Y$ by \cite{KL}, Propositions 4.3.20(I) and 2.5.10(i), according to whether $q \equiv 3 \pmod 4$ or $q \equiv 1 \pmod 4$ respectively. In the former case $X_{x}=Y$ by Lemma \ref{Nov}(1), whereas in the latter $G_{x}$ is a type 1 novelty with respect to $Y$ by Lemma \ref{Nov}(2). However, this is impossible by \cite{KL}, Tables 3.5.H--I. Thus $X_{x}=Y$ in each case, and hence
\begin{equation}\label{nemati}
k^{2}=q^{\left( n/2-1\right) ^{2}}\frac{(q-1)^{2}(q+1)}{(4,q^{n/2}-1)}\cdot \frac{%
q^{n/2}-1}{q-1}
\end{equation}
by \cite{KL}, Proposition 4.3.20.(II). Then $\frac{%
q^{n/2}-1}{q-1}$ is a square, being coprime to the other factor of $\frac{4k^{2}}{(q+1)^{2}}$ or $\frac{k^{2}}{(q+1)^{2}}$ according to whether $q \equiv 1 \pmod 4$ or $q \equiv 3 \pmod 4$ respectively. Hence $(n,q)=(10,3)$ by \cite{Rib}, A8.1, since $n/2$ is odd. However, such a pair does not fulfill (\ref{nemati}), and hence this case is excluded.

Assume that $Y$ is a maximal $\mathcal{C}_{3}$-subgroup of $X$ of type $%
GU_{n/2}(q)$ with $n/2$ even. Then $X$ has two conjugacy classes of subgroups
isomorphic to $Y$ by \cite{KL}, Proposition 4.3.18.(I). If $X_{x}<Y$, then $%
G_{x}$ is a type 1 novelty with respect to $Y$ by Lemma \ref{Nov}(2).
However, this is impossible by \cite{KL}, Tables 3.5.H--I. Thus $X_{x}=Y$
and hence%
\begin{equation}\label{vala}
k^{2}=\frac{1}{2} q^{n\left( n-2\right) /8} \prod_{i=1}^{n/2-1}(q^{i}+(-1)^{i}) 
\end{equation}
Arguing as in the $P\Omega ^{-}(q)$-case, and bearing in mind that $n/2$ is even, we see that \begin{equation}\label{divandan}
k+1 \mid f(q+1,n/2)(q+1,n/2-2)(q^{n/2}-1)(q^{n/2-1}+1) 
\end{equation}
and hence 
\begin{equation*}
q^{n(n-2)/4-1 }/2<k^{2}<q^{2}(q+1)^{4}(q^{n/2}+1)^{2}(q^{n/2-1}-1)^{2}<64q^{2n+4},
\end{equation*}
which yields either $n=8$, or $n=12$ and $q<64$. It is easy to check that, there are no $q$ less than $64$ fulfilling \ref{vala} for $n=12$. Thus $n=8$ and hence (\ref{divandan}) becomes
\begin{equation*}
k^{2}=\frac{1}{2}q^{6}\left( q^{2}+1\right) (q-1)^{2}\left( q^{2}+q+1\right).
\end{equation*}
So $q^{2}+q+1$ is a square. However, this is impossible by \cite{Rib}, A7.1. Thus $Y$ is not a $\mathcal{C}_{3}$-subgroup of $X$.

Assume that $Y\in \mathcal{C}_{1}(X)$. Then $Y$ is not a parabolic
subgroup of $X$ by \cite{KL}, Proposition 4.1.20.(II), since $%
\left( \Phi _{(n-2)f}^{\ast }(p),\left\vert Y\right\vert \right) >1$. Moreover, by \cite{KL}%
, Propositions 4.1.26.(II)--4.1.7.(II) and 2.5.10.(i), one of the following
cases occurs:

\begin{enumerate}
\item $Y$ is the stabilizer in $X$ of a non-singular point of $PG_{n-1}(q)$:

\begin{enumerate}
\item $Y\cong \Omega _{n-1}(q)$ with $q\equiv 1 \pmod 4$, or $q\equiv 3 \pmod 4$ and $n/2$ even;

\item $Y\cong \Omega _{n-1}(q).Z_{2}$, with $q\equiv 3 \pmod 4$ and $n/2$
odd;

\item $Y\cong Sp_{n-2}(q)$, with $q$ even.
\end{enumerate}

\item $Y$ is the stabilizer in $X$ of a non-singular line of type $-$ of $PG_{n-1}(q)$:

\begin{enumerate}
\item $Y\cong \left( Z_{\frac{q+1}{(q+1,2)}}\times \Omega
_{n-2}^{-}(q)\right) .Z_{2}$, with $q$ even, or $q\equiv 1 \pmod4 $;

\item $Y\cong \left( Z_{\frac{q+1}{2}}\times \Omega _{n-2}^{-}(q)\right)
.[4] $, with $q\equiv 3 \pmod 4$ and $n/2$ odd;

\item $Y\cong Z_{2}. \left( Z_{\frac{q+1}{4}}\times P\Omega _{n-2}^{-}(q)\right) .[4]$,
with $q\equiv 3 \pmod 4$ and $n/2$ even.
\end{enumerate}
\end{enumerate}

The group $X$ has a unique conjugacy class of subgroups isomorphic to $Y$ in
case (1c) and in cases (2a)--(2c) by \cite{KL}, Propositions
4.1.6.(I)--4.1.7.(I). Then $X_{x}=Y$ in these cases by Lemma \ref{Nov}(1).

In cases (1a) and (1b), $X$ has two conjugacy classes of subgroups
isomorphic to $Y$ again by \cite{KL}, Proposition 4.1.6.(I). If $X_{x}<Y$,
then $G_{x}$ is a type 1 novelty with respect to $Y$ by Lemma \ref{Nov}(2),
but this is impossible by \cite{KL} Tables 3.5.H--I. Thus, $Y=X_{x}$ in each
case.

Assume that $X_{x}$ is the stabilizer in $X$ of a non-singular point of $%
PG_{n-1}(q)$. As pointed in \cite{Saxl}, Section 5 and (1.b), the lengths of
the $X_{x}$-orbits distinct from $\left\{ x\right\} $ are as follows:

\begin{description}
\item[$q\equiv 1 \pmod 4$] one orbit of length $q^{n-2}-1$, one of length $%
q^{n/2-1}(q^{n/2-1}+1)/2$, $(q-1)/4$ ones of length $q^{n/2-1}(q^{n/2-1}-1)$, and 
$(q-5)/4$ ones of length $q^{n/2-1}(q^{n/2-1}+1)$.

\item[$q\equiv 3 \pmod 4$] one orbit of length $q^{n-2}-1$, one of length $%
q^{n/2-1}(q^{n/2-1}-1)/2$, $(q-3)/4$ ones of length $q^{n/2-1}(q^{n/2-1}-1)$, and 
$(q-3)/4$ others of length $q^{n/2-1}(q^{n/2-1}+1)$

\item[$q \equiv 0 \pmod 2$] one of length $q^{n-2}-1$, $q/2$ ones of length $%
q^{n/2-1}(q^{n/2-1}-1)$ and $(q-2)/2$ ones of length $q^{n/2-1}(q^{n/2-1}+1)$.
\end{description}

In each of the three cases the greatest common divisor of the lengths of the $X_{x}$-orbits is $1$. Thus $k+1\mid (2,q-1)\left\vert \mathrm{Out(}X\mathrm{)}\right\vert $, since $\frac{k+1}{\left( k+1,\left\vert \mathrm{Out(}X\mathrm{)}\right\vert
\right) }$ divides the length of each $X_{x}$-orbit distinct from ${\lbrace x \rbrace}$ by Lemma \ref{Orbits}(2). Hence, $k<(2,q-1)\cdot (4,q^{n/2}-1)\cdot 2\theta f$, where
either $\theta =3$ and $n=8$, or $\theta =1$. Since $k^{2}=[X:X_{x}]$, it follows that $k^{2}_{p}=q^{(n-2)/2}$, hence $q^{(n-2)/4}\mid k$. Thus $q^{(n-2)/4}<(2,q-1)\cdot (4,q^{n/2}-1)\cdot
2\theta f$ and hence either $n=8$ and $q=4,9$, or $n=10$ and $q=3$. Therefore, either $X\cong P\Omega
_{8}^{+}(4)$ and $X_{x}\cong Sp_{6}(2)$, or $X \cong P\Omega _{8}^{+}(9)$ and 
$X_{x} \cong \Omega _{7}(9)$, or $X \cong P \Omega_{10}^{+}(3)$ and $X_{x} \cong \Omega_{9}(3).Z_{2}$. The first two cases are ruled out, since $\left[ X:X_{x}\right] $ is not a square for any of them, whereas in the latter $k=99$. Then $b=9900\lambda$, where $\lambda \mid 99$ and $\lambda \geq 2$. Thus $3^{16}\cdot 7 \cdot 13 \cdot 41 \mid \left\vert X_{B} \right \vert$ and hence $\Omega_{9}(3) \unlhd X_{B} \leq \Omega_{9}(3).Z_{2}$ by \cite{BHRD}, Tables 8.58--8.59 and 8.66--8.67. So, either $[X:X_{B}]=99^2$ or $[X:X_{B}]=2*99^2$, and hence $b \mid 8*99^{2}$ since $\left\vert \mathrm{Out(}X\mathrm{)}\right\vert =4$. However, this is impossible since $100$ divides $b$. So, this case is ruled out. 

Finally, assume that $X_{x}$ is the stabilizer in $X$ of a non-singular line of type $-$ of $PG_{n-1}(q)$. Then $\left\vert X_{x}\right\vert =\left\vert \Omega
_{n-2}^{-}(q)\right\vert (q+1)\mu$,
where 
\begin{equation}
\centering
\mu = \left\{\begin{split}
2 & \hphantom{ciao}q\text{ even, or }q \equiv 3 \hspace{-2.5mm} \pmod4\text{ and } n/2 \text{ odd} \\
1/2 & \hphantom{ciao}q \equiv 3 \hspace{-2.5mm} \pmod{4}\text{ and } n/2 \text{ even} \\ 
1 & \hphantom{ciao}q \equiv 1 \hspace{-2.5mm} \pmod{4} \text{.}\\
\end{split}\right.
 \end{equation}

Set $\rho =\frac{\left( q^{n/2-1}+1,4\right) }{\left( q^{n/2}-1,4\right)\mu}$. Then $\rho$ is $2$ for $q \equiv 3 \pmod{4}$ and $n/2$ even, and $1/2$ otherwise, and hence%
\begin{equation}
k^{2}=q^{n-2}\rho (q-1)^{2}\left[ \frac{1}{%
(q+1)}\left( \frac{q^{n/2}-1}{q-1}\right) \left( \frac{q^{n/2-1}-1}{q-1}%
\right) \right] \text{.}  \label{sestougau}
\end{equation}%
Assume that $n/2$ is odd. Then $\frac{q^{n/2}-1}{q-1}$ is odd and $\frac{%
q^{n/2-1}-1}{q^{2}-1}$ is an integer. Moreover, they are coprime. Then $%
\frac{q^{n/2}-1}{q-1}$ is a square, since the maximal odd divisor of $\frac{%
(q-1)^{2}}{\left( 4,q^{n/2}-1\right) }$ is a square too. Then $(n,q)=(10,3)$
by \cite{Rib}, A7.1 and A8.1, since $n\geq 8$. However, such $(n,q)$ does
not fulfill (\ref{sestougau}).

Assume that $n/2$ is even. Arguing as above, we see that $\frac{q^{n/2-1}-1}{%
q-1}$ is a square and hence $(n,q)=(12,3)$ by \cite{Rib}, A7.1 and A8.1,
since $n\geq 8$, but such $(n,q)$ does not fulfill (\ref{sestougau}).
This completes the proof.
\end{proof}

\bigskip

Now, the proof of Theorem \ref{main} is a summary of the previous results.

\bigskip

\begin{proof}[Proof of Theorem \protect\ref{main}]
Suppose that $X$ is a simple group. Then $X$ is classical by Theorem 1 of 
\cite{MF}. Note that $X\cong $ $^{2}G_{2}(3)^{\prime }$ is classical,
since $^{2}G_{2}(3)^{\prime }\cong PSL_{2}(8)$. If $X\cong PSL_{2}(q)$, then
assertions (1) and (2) follow from Proposition \ref{Due nonVale}. If $%
X\ncong PSL_{2}(q)$, then the followings hold by Remark \ref{mag3}:

\begin{enumerate}
\item $n\geq 3$ and $X\cong PSL_{n}(q)$

\item $n\geq 3$ and $X\cong PSU_{n}(q^{1/2})$ for $q$ square;

\item $n\geq 4$ and $X\cong PSp_{n}(q)$;

\item $n\geq 5$ and $X\cong P\Omega _{5}(q)$;

\item $n\geq 6$ and $X\cong P\Omega _{n}^{-}(q)$;

\item $n\geq 6$ and $X\cong P\Omega _{n}^{+}(q)$.
\end{enumerate}

Cases (2)--(6) are excluded in Lemmas \ref{PSp}, \ref{PSU}, \ref{POmegadisp}%
, \ref{POmegaMeno} and \ref{POmegaPiu}, respectively. Thus $X\cong
PSL_{n}(q) $, with $n\geq 3$, and hence assertions (3) and (4) follow from
Lemma \ref{PSL} and Proposition \ref{ExmplPSL}.
\end{proof}

\end{document}